\documentclass[12pt]{amsart}

\usepackage[utf8]{inputenc} 
\usepackage[OT2,T1]{fontenc}

\usepackage[top=2.5cm,bottom=2.5cm,right=3cm,left=2cm]{geometry}

\usepackage{amsmath,amsthm,amssymb,mathrsfs,calligra,stmaryrd,dsfont, mathtools}
\usepackage{fouriernc}
\usepackage[final, babel=true]{microtype}
\usepackage{caption}
\usepackage{subcaption}
\usepackage{eucal}
\usepackage{verbatim}

\usepackage{centernot}
\usepackage{enumitem}

\usepackage{graphicx}
\usepackage{tikz}
\usetikzlibrary{matrix,arrows,cd}
\usepackage[all]{xy}

\usepackage{hyperref}
\usepackage{xcolor}
\usepackage{url}
\definecolor{darkgreen}{rgb}{0,0.45,0}  

\hypersetup{
	colorlinks,
	linktoc=all,
	linkcolor={darkgreen},
	citecolor={blue!50!black},
	urlcolor={blue!80!black}
} 

\DeclareFontFamily{U}{mathx}{\hyphenchar\font45}
\DeclareFontShape{U}{mathx}{m}{n}{
	<5> <6> <7> <8> <9> <10>
	<10.95> <12> <14.4> <17.28> <20.74> <24.88>
	mathx10
}{}
\DeclareSymbolFont{mathx}{U}{mathx}{m}{n}
\DeclareFontSubstitution{U}{mathx}{m}{n}
\DeclareMathAccent{\widecheck}{0}{mathx}{"71}
\DeclareMathAccent{\wideparen}{0}{mathx}{"75}

\newcommand\bR{\mathbb{R}}
\newcommand\bC{\mathbb{C}}
\newcommand\bN{\mathbb{N}}
\newcommand\bZ{\mathbb{Z}}
\newcommand\bQ{\mathbb{Q}}

\newcommand\bP{\mathbb{P}}

\newcommand\bT{\mathbb{T}}

\newcommand\Zp{{\mathbb{Z}_p}}
\newcommand\Qp{{\mathbb{Q}_p}}

\newcommand\cA{\mathcal{A}}
\newcommand\cB{\mathcal{B}}

\newcommand\cE{\mathcal{E}}
\newcommand\cF{\mathcal{F}}
\newcommand\cG{\mathcal{G}}
\newcommand\cH{\mathcal{H}}
\newcommand\cI{\mathcal{I}}

\newcommand\cL{\mathcal{L}}

\newcommand\cO{\mathcal{O}}

\newcommand\cU{\mathcal{U}}

\newcommand\cW{\mathcal{W}}

\newcommand\g{\mathfrak{g}}

\newcommand\gp{\mathfrak{p}}

\newcommand\sC{\mathscr{C}}

\newcommand\sS{\mathscr{S}}

\newcommand{\ob}[1]{\mkern 1.5mu\overline{\mkern-1.5mu#1\mkern-1.5mu}\mkern 1.5mu}

\newcommand{\loc}{\mathrm{loc}}
\newcommand{\Loc}{\mathrm{Loc}}
\newcommand{\str}{\mathrm{str}}
\newcommand{\Ind}{\mathrm{Ind}}

\newcommand{\res}{\mathrm{res}}

\newcommand{\ord}{\mathrm{ord}}
\newcommand{\Fil}{\mathrm{Fil}}

\newcommand{\alg}{\textrm{alg}}
\renewcommand{\char}{\textrm{char}}

\DeclareMathOperator{\Gal}{Gal}

\DeclareMathOperator{\Hom}{Hom}

\DeclareMathOperator{\im}{im}

\DeclareMathOperator{\coker}{coker}
\DeclareMathOperator{\Frac}{Frac}
\DeclareMathOperator{\rec}{rec}
\DeclareMathOperator{\ad}{ad}

\DeclareMathOperator{\Aut}{Aut}
\DeclareMathOperator{\Tr}{Tr}
\DeclareMathOperator{\Reg}{Reg}

\DeclareMathOperator{\cyc}{cyc}
\DeclareMathOperator{\GL}{GL}

\DeclareMathOperator{\Sel}{Sel}
\DeclareMathOperator{\Coleman}{Col}
\DeclareMathOperator{\diff}{d}

\newcommand{\ram}{\textrm{ram}}

\newcommand{\HH}{\mathrm{H}}
\newcommand{\ff}{\mathrm{f}}

\newcommand{\crys}{\mathrm{crys}}
\newcommand{\BK}{\mathrm{BK}}
\newcommand{\Iw}{\mathrm{Iw}}
\newcommand{\ur}{\mathrm{ur}}

\DeclareSymbolFont{cyrletters}{OT2}{wncyr}{m}{n}
\DeclareMathSymbol{\Sha}{\mathalpha}{cyrletters}{"58}

\newcommand{\DR}{\mathrm{DR}}

\makeatletter
\@namedef{subjclassname@2020}{\textup{2020} Mathematics Subject Classification}
\makeatother

\newtheoremstyle{thmstyle}
{\parskip} 
{\topsep} 
{\itshape} 
{} 
{\bfseries} 
{.} 
{.5em} 
{} 

\newtheoremstyle{defstyle}
{\parskip} 
{\topsep} 
{} 
{} 
{\bfseries} 
{.} 
{.5em} 
{} 

\theoremstyle{defstyle}
\newtheorem{definition}[subsubsection]{Definition}

\newtheorem{remark}[subsubsection]{Remark}
\newtheorem{notation}[subsubsection]{Notation}

\theoremstyle{thmstyle}
\newtheorem{theorem}[subsubsection]{Theorem}
\newtheorem{proposition}[subsubsection]{Proposition}
\newtheorem{lemme}[subsubsection]{Lemma}
\newtheorem{corollaire}[subsubsection]{Corollary}

\newtheorem{conjecture}[subsubsection]{Conjecture}
\newtheorem*{conjecture*}{Conjecture}
\newtheorem{question}[subsubsection]{Question}

\newtheorem{THM}{Theorem}

\newtheorem{CONJ}{Conjecture}

\newtheorem{CORO}{Corollary}

\begin{document}
\title{On generalized main conjectures and $p$-adic Stark conjectures for Artin motives}
\author{Alexandre Maksoud}
\email{maksoud.alexandre@gmail.com}
\address{Universität Paderborn, J2.305, Warburger Str. 100, 33098 Paderborn, Germany}

\begin{abstract}
	
		Given an odd prime number $p$ and a $p$-stabilized Artin representation $\rho$ over $\bQ$, we introduce a family of $p$-adic Stark regulators and we formulate an Iwasawa-Greenberg main conjecture and a $p$-adic Stark conjecture which can be seen as an explicit strengthening of conjectures by Perrin-Riou and Benois in the context of Artin motives. 
		We show that these conjectures imply the $p$-part of the Tamagawa number conjecture for Artin motives at $s=0$ and we obtain unconditional results on the torsionness of Selmer groups. We also relate our new conjectures with various main conjectures and variants of $p$-adic Stark conjectures that appear in the literature. 
		In the case of monomial representations, we prove that our conjectures are essentially equivalent to some newly introduced Iwasawa-theoretic conjectures for Rubin-Stark elements. We derive from this a $p$-adic Beilinson-Stark formula for finite-order characters of an imaginary quadratic field in which $p$ is inert. 

Along the way, we prove that the Gross-Kuz'min conjecture unconditionally holds for abelian extensions of imaginary quadratic fields.
\end{abstract}
\maketitle

\section{Introduction}
Iwasawa theory traditionally focuses on the construction of $p$-adic $L$-functions and on their relation with arithmetic invariants of number fields. Concurrently with the first major achievements of the theory \cite{mazurwiles,wiles1990iwasawa,rubin1991main}, several attempts were made in order to define the conjectural $p$-adic $L$-function of a motive \cite{coatesperrinriou,coatesmotivic,greenberg1994trivial} and its corresponding Selmer group \cite{greenberg1989iwasawa,greenberg153iwasawa}. The motive in question was assumed to admit a critical value in the sense of Deligne and to be ordinary at $p$, hypotheses which were relaxed in a work of Perrin-Riou \cite{perrin1995fonctions} using her extended logarithm map.  In \cite{benoiscrys}, Benois proposed later a conjectural description of the behavior of Perrin-Riou's $p$-adic $L$-function at ``trivial zeros'' via Nekov\'a\v{r}'s theory of Selmer complexes \cite{sel}. 



While being ordinary at unramified primes, motives coming from Artin representations are seldom critical and the arithmetic of their $L$-values at $s=0$ is described by ``refined Stark conjectures'', thanks to the work of various mathematicians beginning with Stark's and Rubin's influential papers \cite{stark2,stark4,rubinstark} and culminating in \cite{burnsLseries}. In the particular case of monomial Artin motives, that is, when the associated Artin representation is induced from a one-dimensional character, a long-term strategy to tackle these conjectures with the aid of Iwasawa theory was presented in \cite{BKSANT}. A general main conjecture (called ``higher rank Iwasawa main conjecture'') and an extra zero conjecture (called ``Iwasawa-theoretic Mazur-Rubin-Sano conjecture'') are formulated in terms of Rubin-Stark elements, but neither a $p$-adic $L$-function nor an $\cL$-invariant play a part in this work. Nevertheless, the authors verify that their extra zero conjecture generalizes the Gross-Stark conjecture. Not long after that, B{\"u}y{\"u}kboduk and Sakamoto in \cite{buyukboduksakamoto} used Coleman maps to deduce from the Mazur-Rubin-Sano conjecture a $p$-adic Beilinson formula for Katz's $p$-adic $L$-function.

The aim of our paper is to provide a unifying approach to the cyclotomic Iwasawa theory for Artin motives which is close in spirit to Perrin-Riou's and Benois' treatment and which generalizes many aspects of \cite{BKSANT,buyukboduksakamoto}. Not only does this encompass classical conjectures on Deligne-Ribet's and Katz's $p$-adic $L$-functions, but it also sheds light on recent constructions of $p$-adic $L$-functions attached to a classical weight one modular cuspform $f$ in \cite{bellaiche2016eigencurve,maks} and on new variants of the Gross-Stark conjecture for the adjoint motive of $f$ in \cite{DLR2}. More specifically, one of the central objects in this paper is the Selmer group $X_\infty(\rho,\rho^+)$ whose properties are studied by Greenberg and Vatsal in \cite{GV2020} and by the author in \cite{maks}. It depends on the choice of an ordinary $p$-stabilization $\rho^+$ of the $p$-adic realization $\rho$ of the Artin motive. Although there is no canonical choice for $\rho^+$ when $\rho$ is non-critical, a fruitful idea in this work is to think of $\rho^+$ as an \textit{additional parameter}. 

The question of the torsionness of Selmer groups over the Iwasawa algebra is a recurring theme in Iwasawa theory. When $\rho$ is an Artin representation, it was shown in \cite{maks} that delicate conjectures coming from $p$-adic transcendence theory arise through the study of $X_\infty(\rho,\rho^+)$ (see also \cite{GV2020}). A first crucial aspect of this work is to obtain, under an unramifiedness assumption at $p$, finer information on the structure of $X_\infty(\rho,\rho^+)$ with the aid of modified Coleman maps and of classical freeness results on $\Zp$-towers of global units. We also relate its characteristic ideal with Perrin-Riou's module of $p$-adic $L$-functions for the dual motive associated with $\rho$.

On the analytic side, our conjectural $p$-adic $L$-function interpolates the algebraic part of Artin $L$-values at $s=0$ given by a recipe of Stark \cite{stark2,tate}. ``Extra zeros'' in the sense of Benois abound in this setting, and we also formulate an extra zero conjecture which is compatible with the one in \cite{benoiscrys}. It involves a new $\cL$-invariant which computes Benois' $\cL$-invariant when the Artin motive is crystalline at $p$, and which also recovers various $\cL$-invariants that appear in the literature \cite{gross1981padic,riveroadjoint,rosetrotgervatsal,buyukboduksakamoto}.

In the next section we formulate our principal conjecture.

\subsection{The conjecture}
We fix once and for all embeddings $\iota_\infty:\ob{\bQ} \hookrightarrow \bC$ and $\iota_\ell : \ob{\bQ} \hookrightarrow \ob{\bQ}_\ell$ for all primes $\ell$.  Let $$\rho : G_{\bQ}\longrightarrow \GL_E(W)$$ be a $d$-dimensional representation of $G_\bQ=\Gal(\ob{\bQ}/\bQ)$ of finite image with coefficients in a number field $E \subseteq \ob{\bQ}$. We always assume that $W$ does not contain the trivial representation, \textit{i.e.}, one has $\HH^0(\bQ,W)=0$. For any character $\eta:G_\bQ \longrightarrow \ob{\bQ}^\times$ of finite order, we denote by $W_\eta$ the underlying space of $\rho\otimes\eta$, by $E_\eta\subseteq \ob{\bQ}$ its coefficient field and by $H_\eta\subseteq \ob{\bQ}$ the Galois extension cut out by $\rho\otimes\eta$. Our fixed embedding $\iota_\infty$ allows us to see $\rho\otimes\eta$ as an Artin representation, and we let $L(\rho\otimes\eta,s)$ be its Artin $L$-function. It is a meromorphic function over $\bC$, and it is known that $L(\rho\otimes\eta,1)\in\bC^\times$ provided that $\HH^0(\bQ,W_\eta)=0$. Assuming further that $\eta$ is even, the functional equation implies that $L((\rho\otimes\eta)^*,s)$ has a zero of order $d^+=\dim\HH^0(\bR,W)$ at $s=0$, where $(-)^*$ stands for the contragredient representation. The non-abelian Stark conjecture \cite{stark2} describes the transcendental part of its leading term $L^*((\rho\otimes\eta)^*,0)$ at $s=0$ as follows. Set 
\[\HH^1_\ff((\rho\otimes\eta)^*(1))=\Hom_{G_{\bQ}}(W_\eta,E_\eta \otimes_{\bZ} \cO_{H_\eta}^\times),\] 
where $\cO_{H_\eta}^\times$ is the group of units of $H_\eta$. It is a $E_\eta$-vector space of dimension $d^+$ by Dirichlet's unit theorem, and it provides a natural rational structure for a certain global Bloch-Kato Selmer group (see \eqref{eq:bloch-kato-kummer}).
There is a natural $\bC$-linear perfect pairing 

\begin{equation}\label{eq:intro_complex_pairing}
\left\{\begin{array}{rclcc} \bC \otimes \HH^0(\bR,W) & \times& \bC \otimes \HH^1_\ff((\rho\otimes\eta)^*(1)) &\longrightarrow &\bC \\
(1 \otimes w &,&1\otimes \psi) &\mapsto & \log_\infty(\psi(w)),
\end{array}\right.
\end{equation}
 where $\log_\infty : \bC \otimes \ob{\bQ}^\times \longrightarrow \bC$ is given by $ x \otimes a \mapsto -x \log|\iota_\infty(a)|$ (and $\log$ refers to the usual real logarithm). Stark's regulator attached to $\rho \otimes\eta$ is the determinant $\Reg_{\omega_\infty^+}(\rho\otimes\eta)\in \bC^\times$ of the pairing \eqref{eq:intro_complex_pairing} with respect to rational bases $\omega_{\infty}^+\in \det_E \HH^0(\bR,W)$ and $\omega_{\ff,\eta}\in \det_{E_\eta} \HH^1_\ff((\rho\otimes\eta)^*(1))$, and the non-abelian Stark conjecture implies
 \begin{equation}\label{eq:Stark_conj_over_E}
  \dfrac{L^*((\rho\otimes\eta)^*,0)}{\Reg_{\omega_\infty^+}(\rho\otimes\eta)} \in E_\eta^\times, \qquad \textrm{or, equivalently,} \qquad \frac{L(\rho\otimes\eta,1)}{(i\pi)^{d^-}\Reg_{\omega_\infty^+}(\rho\otimes\eta)} \in E_\eta^\times,
 \end{equation}
where $d^-=d-d^+$ (see Lemma \ref{lem:stark}).

Fix once and for all a prime number $p$ as well as an isomorphism 
\[j:\bC \simeq \ob{\bQ}_p\]
such that $\iota_p=j \circ \iota_\infty$. Letting $E_{p}$ be the completion of $\iota_p(E)$ inside $\ob{\bQ}_p$, one may see $\rho$ as a $p$-adic representation by putting $W_p=W \otimes_{E,\iota_p}E_p$. We call a \textit{$p$-stabilization} of $W_p$ any $G_{\Qp}$-stable linear subspace $W_p^+ \subseteq W_p$ of dimension $d^+$, where $G_{\Qp}$ is the local Galois group $\Gal(\ob{\bQ}_p/\Qp)$. Any $p$-stabilization $(W_p^+,\rho^+)$ yields a $p$-adic analogue of the complex pairing (\ref{eq:intro_complex_pairing}) by considering
\begin{equation}\label{eq:intro_p_adic_pairing}
\left\{\begin{array}{rclcc} \ob{\bQ}_p \otimes W_p^+ & \times& \ob{\bQ}_p \otimes \HH^1_\ff((\rho\otimes\eta)^*(1)) &\longrightarrow &\ob{\bQ}_p \\
(1 \otimes w &,&1\otimes \psi) &\mapsto & \log_p(\psi(w)),
\end{array}\right.
\end{equation}
where $\log_p: \ob{\bQ}_p \otimes \ob{\bQ}^\times \longrightarrow \ob{\bQ}_p$ is given by  $ x \otimes a \mapsto x \log_p^\Iw(\iota_p(a))$ (and $\log_p^\Iw$ is the Iwasawa logarithm). A $p$-stabilization $W_p^+$ is called $\eta$-\textit{admissible} if the $\ob{\bQ}_p$-linear pairing (\ref{eq:intro_p_adic_pairing}) is perfect, and it is simply called \textit{admissible} when $\eta$ is the trivial character $\mathds{1}$ (cf. Definition \ref{def:p_stab}).  
The $p$-adic Stark regulator attached to $\rho \otimes \eta$ is defined as the determinant $\Reg_{\omega_p^+}(\rho\otimes\eta)\in\ob{\bQ}_p$ of the pairing \eqref{eq:intro_p_adic_pairing} with respect to bases $\omega_p^+\in\det_{E_p}W_p^+$ and $\omega_{\ff,\eta}\in \det_{E_\eta} \HH^1_\ff((\rho\otimes\eta)^*(1))$. Note that $\Reg_{\omega_p^+}(\rho\otimes\eta)\in\ob{\bQ}_p$ vanishes precisely when its $\eta$-admissibility fails. 

A $p$-stabilization might \textit{a priori} be $\eta$-admissible or not, but in the special case where $W_p^+$ is \textit{motivic}, \textit{i.e.}, when it admits a $E$-rational structure (cf. Definition \ref{def:p_stab}), its $\eta$-admissibility follows from a theorem of Brumer when $d^+=1$ and from the weak $p$-adic Schanuel conjecture in general (see \S\ref{sec:p_stab}).

Let $\cO_p$ be the ring of integers of $E_p$ and fix a Galois-stable $\cO_p$-lattice $T_p$ of $W_p$. Then $T_p^+=W_p^+\bigcap T_p$ is a lattice inside $W_p^+$. We always assume that $\omega_\infty^+$ and $\omega_p^+$ are $T_p$-\textit{optimal}, that is, they are respective bases of $\det_{\cO_p}\HH^0(\bR,T_p)$ and $\det_{\cO_p} T_p^+$.


Let $\Gamma$ be the Galois group of the $\Zp$-cyclotomic extension $\bQ_\infty=\cup_n \bQ_n$ of $\bQ$ and let $\widehat{\Gamma}\subseteq \Hom(\Gamma,\ob{\bQ}^\times)$ be the set of $\ob{\bQ}$-valued characters of $\Gamma$ of finite order. Via $\iota_\infty$, the elements of $\widehat{\Gamma}$ correspond to (necessarily even) Dirichlet characters of $p$-power order and conductor, and via $\iota_p$ they become $p$-adic characters which we see as homomorphisms of $\cO_p$-algebras $\Lambda \longrightarrow \ob{\bQ}_p$, where $\Lambda=\cO_p[[\Gamma]]$ is the Iwasawa algebra. Consider the Greenberg-style Selmer group \[
X_\infty(\rho,\rho^+):= \ker \left[ \HH^1(\bQ_\infty,D_p) \longrightarrow \HH^1(\bQ_{p,\infty}^\ur,D_p^-) \times \prod_{\ell\neq p} \HH^1(\bQ^\ur_{\ell,\infty},D_p)  \right]^\vee,
\]
where $D_p$ (resp. $D_p^-$) is the divisible $\cO_p$-module $W_p/T_p$ (resp. $D_p/(\im W_p^+ \rightarrow D_p)$), $\bQ^\ur_{\ell,\infty}\subseteq \ob{\bQ}_\ell$ the maximal unramified extension of the completion of $\bQ_\infty$ along $\iota_\ell$ and $(-)^\vee$ the Pontryagin dual. Then $X_\infty(\rho,\rho^+)$ has a structure of $\Lambda$-module of finite type and we denote by $\char_\Lambda X_\infty(\rho,\rho^+)$ its characteristic ideal. By convention, $\char_\Lambda X_\infty(\rho,\rho^+)=\{0\}$ if $X_\infty(\rho,\rho^+)$ is not torsion.

\begin{CONJ}\label{conj:IMC}
	Let $p$ be an odd prime at which $\rho$ is unramified. Fix a $G_{\bQ}$-stable $\cO_p$-lattice $T_p$ of $W_p$ and a $p$-stabilization $(\rho^+,W_p^+)$ of $W_p$. Pick any $T_p$-optimal bases $\omega_p^+$ and $\omega_\infty^+$ as before. 
	\begin{description}
		\item[\textbf{EX}$_{\rho,\rho^+}$]
		 There exists an element $\theta_{\rho,\rho^+}$ in the field of fractions of $\Lambda$ which has no pole outside the trivial character and which satisfies the following interpolation property: for all non-trivial characters $\eta\in\widehat{\Gamma}$ of exact conductor $p^{n}$, one has 
	$$\begin{aligned}
	\eta(\theta_{\rho,\rho^+})&= {\det(\rho^-)(\sigma_p^{-n})} \ \Reg_{\omega_p^+}(\rho\otimes\eta)\cdot j\left(\frac{\tau(\eta)^{d^-}}{\tau(\rho\otimes\eta)} \ \dfrac{L^{*}\left((\rho \otimes \eta)^*,0\right)}{\Reg_{\omega_\infty^+}(\rho\otimes\eta)}\right),
	\end{aligned}$$
	where $\tau(-)$ is a Galois-Gauss sum (see \S\ref{sec:Galois_Gauss}), $\sigma_p$ the arithmetic Frobenius at $p$, and $\det(\rho^-)$ the determinant of the local Galois representation acting on $W_p^-=W_p/W_p^+$.
	\item[\textbf{IMC}$_{\rho,\rho^+}$] The statement \textbf{EX}$_{\rho,\rho^+}$ holds, and  $\theta_{\rho,\rho^+}$ is a generator of $\char_\Lambda X_\infty(\rho,\rho^+)$.
\item[\textbf{EZC}$_{\rho,\rho^+}$] Let $e$ be the dimension of $W_p^{-,0}:=\HH^0(\Qp,W_p^-)$ and assume that $\rho^+$ is admissible. Then \textbf{EX}$_{\rho,\rho^+}$ holds, $\theta_{\rho,\rho^+}$ vanishes at $\mathds{1}$ with multiplicity at least $e$, and one has
$$\begin{aligned}
\frac{1}{e!} \dfrac{\diff^e}{\diff \!s^e}\kappa^s(\theta_{\rho,\rho^+})\bigg|_{s=0}&=  (-1)^e \ \cE(\rho,\rho^+) \ \cL(\rho,\rho^+) \ \Reg_{\omega_p^+}(\rho) \cdot j\left( \tau(\rho)^{-1}\ \dfrac{L^{*}\left(\rho^*,0\right)}{\Reg_{\omega_\infty^+}(\rho)}\right),
\end{aligned}$$
where $\kappa^s\in\Hom_{\cO-\alg}(\Lambda,\ob{\bQ}_p)$ is the homomorphism induced by the character $\langle\chi_{\cyc}\rangle : \Gamma \simeq 1+p\Zp \subseteq \ob{\bQ}_p^\times$, raised to the power $s\in \Zp$,  $\cL(\rho,\rho^+)\in\ob{\bQ}_p$ is the $\cL$-invariant defined in \S\ref{sec:L-invariant}, and  $\cE(\rho,\rho^+)$ is a modified Euler factor given by
$$\cE(\rho,\rho^+)=\det(1-p^{-1}\sigma_p\big|W_p^+) \det(1-\sigma_p^{-1}\big|W_p^-/W_p^{-,0}).$$
	\end{description}
\end{CONJ}

In other words, we conjecture that the special values $L^*((\rho\otimes\eta)^*,0)$ for $\eta$ varying in $\widehat{\Gamma}$ and suitably corrected by a complex and a $p$-adic Stark regulator can be $p$-adically interpolated by a $p$-adic measure which generates the characteristic ideal of $X_\infty(\rho,\rho^+)$ over $\Lambda$.

The first two parts of Conjecture \ref{conj:IMC} should be thought as an explicit variant of Perrin-Riou's main conjecture \cite{perrin1995fonctions,colmezfonctions} for the arithmetic dual of $\rho$, while the last part is analogous  to Benois' extra zeros conjecture \cite{benoiscrys} (see \S\ref{sec:perrin-riou}). Conjecture \ref{conj:IMC} is strongly contingent on the nonabelian Stark conjecture.  

The measure $\theta_{\rho,\rho^+}$ is uniquely determined by the interpolation property \textbf{EX}$_{\rho,\rho^+}$ and it also depends on the choice of $\iota_\infty$ and $\iota_p$. The dependence of $\theta_{\rho,\rho^+}$ on the bases $\omega_\infty^+$, $\omega_p^+$ and $\omega_{\ff,\eta}$ turns out to be very mild (see Proposition \ref{prop:dependence_theta_on_bases}).

Conjecture \ref{conj:IMC} is compatible with a $p$-adic Artin formalism (Remark \ref{rem:p_adic_artin_formalism}).
The statement of the conjecture makes sense even when $X_\infty(\rho,\rho^+)$ is not of $\Lambda$-torsion, in which case Conjecture \ref{conj:IMC} follows from Theorem \ref{thm:intro_A} below.

We illustrate in \S\ref{sec:examples} some special cases of Conjecture \ref{conj:IMC} and their relations with classical conjectures in the literature. When $d^+=0$, the construction of $\theta_{\rho,\rho^+}$ is due to Greenberg and is based on the $p$-adic $L$-function of Deligne-Ribet and Cassou-Noguès. Wiles' deep work on the main conjecture then implies \textbf{IMC}$_{\rho,\rho^+}$. The statement \textbf{EZC}$_{\rho,\rho^+}$ when $d^+=0$ is equivalent to the Gross-Stark conjecture for $\rho$, the $\cL$-invariant being equal to Gross' regulator (see \S\ref{sec:exemple_odd_motives}). When $d^+=d$, both \textbf{EX}$_{\rho,\rho^+}$ and \textbf{EZC}$_{\rho,\rho^+}$ are implied by a conjecture of Tate on Deligne-Ribet's $p$-adic $L$-functions, and \textbf{IMC}$_{\rho,\rho^+}$ follows again from Wiles' theorem (see \S\ref{sec:exemple_even_motives}). 

When $0<d^+<d$, $\rho$ is non-critical and less is known. If $(d,d^+)=(2,1)$ and $\rho$ is irreducible, then $\rho$ comes from a classical weight one modular cuspform $f$. Under a regularity assumption on $f$, a potential candidate for $\theta_{\rho,\rho^+}$ is proposed in Ferrara's PhD thesis \cite{ferrara} (for $f$ dihedral) and in \S\ref{sec:weight_one} (for any $f$) where indirect evidence for \textbf{EX}\(_{\rho,\rho^+}\) and \textbf{IMC}\(_{\rho,\rho^+}\) are provided. 

The subcase where $f$ has complex multiplication by an imaginary quadratic field $k$, that is, $\rho$ is induced from a character $\chi$ of $k$, is better understood. Under mild assumptions on $\chi$, Conjecture \ref{conj:IMC} holds when $p$ splits in $k$. The measure $\theta_{\rho,\rho^+}$ is essentially a cyclotomic Katz $p$-adic $L$-function (\S\ref{sec:exemple_katz}-\ref{sec:example_odd_monomial}) and \textbf{EZC}$_{\rho,\rho^+}$ is an equivalent form of a $p$-adic Beilinson formula proven by Buyukboduk and Sakamoto. When $p$ is inert in $k$, \textbf{EX}$_{\rho,\rho^+}$ and \textbf{EZC}$_{\rho,\rho^+}$ hold when $\chi$ has prime-to-$p$ order, yielding a new $p$-adic Beilinson formula in Corollary \ref{CORO:plus-minus-p-adic-Beilinson}.

The last setting we consider is when $(d,d^+)=(3,1)$ and $\rho=\ad(\rho_g)$, where $\rho_g$ comes from a weight one cuspform $g$. In \S\ref{sec:examples_adjoint}, the terms $\Reg_{{\omega}^+_{p}}(\rho)$ and $\cL(\rho,\rho^+)$ are related in a precise manner to the Gross-Stark unit and the $p$-adic iterated integral attached to $g$ in Darmon-Lauder-Rotger's variant on the Gross-Stark conjecture \cite{DLR2}. This suggests a plausible link between the conjectural element $\theta_{\rho,\rho^+}$ and Hida's adjoint $p$-adic $L$-function, and between \textbf{EZC}$_{\rho,\rho^+}$ and Darmon-Lauder-Rotger's conjecture.

We now outline the main results of this paper.

\subsection{The main results}
We fix once and for all an odd prime $p$ at which $\rho$ is unramified. Our first result (Theorem \ref{thm:equivalence_torsionness_Sel}) generalizes \cite[Théorème A]{maks} and yields an interpolation formula for a generator of $\char_\Lambda X_\infty(\rho,\rho^+)$ which is similar to \textbf{EX}$_{\rho,\rho^+}$. As a piece of notation, we let $U_\infty=\varprojlim_n U_n$ be the $\Lambda$-module of all norm-coherent sequences of elements in the pro-$p$ completion $U_n$ of the group of units of $H\cdot \bQ_n$. It is a finitely generated $\Lambda$-module endowed with a linear action of $\Gal(H\cdot \bQ_\infty /\bQ_\infty)$.

\begin{THM}\label{thm:intro_A}
\begin{enumerate}
	\item[(1)] Fix a $G_{\bQ}$-stable $\cO_p$-lattice $T_p$ of $W_p$, a $p$-stabilization $\rho^+$ of $\rho$ and a $T_p$-optimal basis $\omega_{p}^+=t_1\wedge\ldots\wedge t_{d^+}$ of $\det W_p^+$. The following assertions are equivalent: 
\begin{enumerate}
	\item[(i)] the Selmer group $X_\infty(\rho,\rho^+)$ is of $\Lambda$-torsion, 
	\item[(ii)] $\rho^+$ is $\eta$-admissible for some non-trivial character $\eta\in \widehat{\Gamma}$, and
	\item[(iii)] $\rho^+$ is $\eta$-admissible for all but finitely many characters $\eta\in\widehat{\Gamma}$.
\end{enumerate} 
Moreover, there always exists a $p$-stabilization $\rho^+$ such that $X_\infty(\rho,\rho^+)$ is of $\Lambda$-torsion.

\item[(2)] If the equivalent conditions (i-ii-iii) hold and if $d^+>0$, then there exist $\Lambda$-linearly independent Galois-equivariant homomorphisms $\Psi_1,\ldots,\Psi_{d^+}\colon T_p \longrightarrow U_\infty$ which only depend on $T_p$, and there exists a generator $\theta^\alg_{\rho,\rho^+}$ of $\char_\Lambda X_\infty(\rho,\rho^+)$ such that 
\begin{equation*}
\eta(\theta^\alg_{\rho,\rho^+}) = \frac{p^{(n-1)\cdot d^+}}{\det(\rho^-)(\sigma_p^n)} \cdot j\left(\frac{\tau(\eta)^{d^-}}{\tau(\rho\otimes\eta)}\right) \cdot\det \left(\log_p|\Psi_j(t_i)|_\eta\right)_{1\leq i,j \leq d^+}
\end{equation*}
for all non-trivial characters $\eta\in\widehat{\Gamma}$ of conductor $p^n$, where $|\cdot|_\eta:U_\infty \longrightarrow U_{n-1}$ stands for a certain ``$\eta$-projection'' introduced in \S\ref{sec:torsionness_Sel}.
\end{enumerate}
\end{THM}
By the Baker-Brumer theorem, condition (iii) in Theorem \ref{thm:intro_A} (1) holds if $\rho$ is irreducible, $d^+=1$ and $\rho^+$ is motivic, so $X_\infty(\rho,\rho^+)$ is of $\Lambda$-torsion. Theorem \ref{thm:intro_A} (2) reduces the construction of $\theta_{\rho,\rho^+}$ to the problem of finding ``special elements'' $\Psi_1,\ldots,\Psi_{d^+}$ related to Artin $L$-values. When $\rho$ is monomial, these elements can be obtained from families of Rubin-Stark elements (see Theorem \ref{thm:intro_C} below).

The proof of Theorem \ref{thm:intro_A} makes essential use of the theory of Coleman maps \cite{coleman} and of classical results on the structure of $U_\infty$. 
We first compare $X_\infty(\rho,\rho^+)$ with a Bloch-Kato-style Selmer group and we apply Poitou-Tate duality theorem. We then combine Hochschild-Serre's exact sequence with classical results of Kuz'min \cite{kuzmin} and Belliard \cite{belliard} on the Galois structure of global ($p$-)units, and also a construction of modified Coleman maps in order to recover some information on the structure of $X_\infty(\rho,\rho^+)$ as a $\Lambda$-module.

The Tamagawa number conjecture of Bloch and Kato \cite{blochkato} expresses special values of motivic $L$-functions in terms of arithmetic invariants. The following result (Theorem \ref{coro:formule_BK}) is an instance of how one can tackle Bloch-Kato's conjecture via Iwasawa theory.

\begin{THM}\label{thm:intro_B}
	Assume that there exist a $G_{\bQ}$-stable $\cO_p$-lattice $T_p$ of $W_p$ and an admissible $p$-stabilization $\rho^+$ of $\rho$ such that $\cL(\rho,\rho^+)\neq 0$ and for which Conjecture \ref{conj:IMC} holds. Then the $p$-part of the Tamagawa Number Conjecture (in the formulation of Fontaine and Perrin-Riou \cite{fontaineperrinriou}) for $\rho$ is valid. In particular, for $p$ not dividing the order of the image of $\rho$, one has 
	$$\dfrac{L^{*}\left(\rho^*,0\right)}{\Reg_{\omega_\infty^+}(\rho)}  \sim_p \#\Hom_{\cO_p[G_{\bQ}]}(T_p,\cO_p \otimes\mathscr{C}\!\ell(H)),$$
	where $a\sim_p b$ means that $a$ and $b$ are equal up to a $p$-adic unit, and  $\mathscr{C}\!\ell(H)$ is the ideal class group of the field $H$ cut out by $\rho$. Here, the bases $\omega_{\ff}$ and $\omega_\infty^+$ used to compute $\Reg_{\omega_\infty^+}(\rho)$ are chosen to be $T_p$-optimal.
\end{THM}

The proof of Theorem \ref{thm:intro_B} rests upon a comparison between $X_\infty(\rho,\rho^+)$ and Benois' definition of Perrin-Riou's module of $p$-adic $L$-functions \cite{benoiscrys}. Once the relation between $X_\infty(\rho,\rho^+)$ and Perrin-Riou's theory is fully established, Theorem \ref{thm:intro_B} follows from the techniques of \textit{loc. cit.} involving a descent argument for Selmer complexes.

Our last result (Theorems \ref{thm:comparaison_mesure_p_adique_RS_elts}, \ref{thm:equivalence_IMCs} and \ref{thm:equivalence_MRS_EZC}) compares Iwasawa-theoretic conjectures in \cite[§3-4]{BKSANT} with Conjecture \ref{conj:IMC} in the monomial case.  

\begin{THM}\label{thm:intro_C}
	Suppose that $\rho=\Ind_k^\bQ\chi$, where $\chi:G_k\longrightarrow E^\times$ is a non-trivial character of a number field $k$ of prime-to-$p$ order and conductor. Assume the Rubin-Stark conjecture over ${\bQ}$ and over $\Zp$ for the family of characters $\chi \otimes \eta$, ${\eta\in \widehat{\Gamma}}$.
	\begin{enumerate}
		\item For every $p$-stabilization $\rho^+$, \textbf{EX}$_{\rho,\rho^+}$ holds true.
		\item For every $p$-stabilization $\rho^+$ such that $X_\infty(\rho,\rho^+)$ is of $\Lambda$-torsion, \textbf{IMC}$_{\rho,\rho^+}$ is equivalent to the higher-rank Iwasawa main conjecture (Conjecture \ref{conj:IMC_chi}).
		\item For every $p$-stabilization $\rho^+$, \textbf{EZC}$_{\rho,\rho^+}$ is implied by the Iwasawa-theoretic Mazur-Rubin-Sano conjecture (Conjecture \ref{conj:MRS_chi}).
		\item Conversely, there exists a finite family $\rho^+_1,\ldots,\rho^+_t$ of $p$-stabilizations such that, if \textbf{EZC}$_{\rho,\rho_i^+}$ holds for all $1\leq i\leq t$, and if there exists at least one such $p$-stabilization $\rho_j^+$ satisfying $\cL(\rho,\rho_j^+)\neq 0$, then the Iwasawa-theoretic Mazur-Rubin-Sano conjecture is valid.
	\end{enumerate}
\end{THM}
See Theorem \ref{thm:equivalence_MRS_EZC} (2) for a precise statement of the last assertion. 

Let us comment on the statement and the proof of Theorem \ref{thm:intro_C}. The $p$-adic $L$-function $\theta_{\rho,\rho^+}$ should interpolate leading terms of abelian $L$-functions for which Rubin formulated a Stark conjecture ``over $\bZ$'' \cite{rubinstark}. This involves some Rubin-Stark elements $(\varepsilon_n^{\chi})_{n\geq 1}$ and $u^\chi$, and the (weaker) corresponding Rubin-Stark conjecture over $\bQ$ reduces to the validity of the statement (\ref{eq:Stark_conj_over_E}) implicitly assumed in Conjecture \ref{conj:IMC}. Under the Rubin-Stark conjecture over $\Zp$, which is also assumed in the formulation of \cite{BKSANT}, the family $\varepsilon_\infty^\chi=\varprojlim_n \varepsilon_n^\chi$ has the right $p$-integrality property and the first claim of Theorem \ref{thm:intro_C} mainly follows from the machinery of refined Coleman maps. Our second claim is much in the spirit of theorems which compare a main conjecture ``with $p$-adic $L$-functions'' with a main conjecture ``without $p$-adic $L$-functions''. Its proof makes use of what we call \emph{extended Coleman maps}, together with a variant of the description of $X_\infty(\rho,\rho^+)$ used for Theorems \ref{thm:intro_A} and \ref{thm:intro_B}. For the last two claims, we first need to compute the constant term of our extended Coleman maps. Roughly speaking, the Mazur-Rubin-Sano conjecture connects $u^\chi$ and the bottom layer of $\varepsilon_\infty^\chi$, and can be seen as an equality $(\cE)$ between elements in a certain finite-dimensional vector space $X$. We reinterpret \textbf{EZC}$_{\rho,\rho^+}$ as the same equality, after applying a certain linear form on $X$ attached to $\rho^+$. The third claim then becomes straightforward, and under the conditions of the last claim one can produce a generating set of linear forms $\nu$ on $X$ such that $\nu((\cE))$ holds. Moreover, the non-vanishing assumption on the $\cL$-invariant ensures that the weak exceptional zero conjecture for $\chi$ (Conjecture \ref{conj:EZC_for_RS_elements}) holds, which is a part of the Mazur-Rubin-Sano conjecture.

Theorem \ref{thm:intro_C} provides strong evidence for Conjecture \ref{conj:IMC}, assuming the Rubin-Stark conjecture. As an application, we obtain the construction of plus-minus $p$-adic $L$-functions and a proof of a $p$-adic Beilinson formula in a new setting (see Corollary \ref{coro:formule_theta_p_inerte} for the general statement). 
\begin{CORO}\label{CORO:plus-minus-p-adic-Beilinson}
	Let $k$ be an imaginary quadratic field in which $p$ is inert and let $\chi$ be a character of $k$ of prime-to-$p$ order and conductor which does not descend to $\bQ$. Assume $\chi(p\cO_k)=1$ and, for simplicity, that $\sigma_p^{-1}\cdot\tau\in \ker \chi$, where $\tau\in G_\bQ$ is the complex conjugation determined by the embedding $\ob{\bQ}\subset \bC$. There exist two measures $\theta^\pm_\chi\in\Lambda$ such that, for all characters $\eta\in\widehat{\Gamma}$ of conductor $p^n$, we have
	\begin{equation}\label{eq:CORO:EX}
		\eta(\theta^\pm_\chi)= \left(1\mp\frac{\eta(p)}{p}\right)\left(1\pm{\eta(p)}\right)\frac{(\pm 1)^n}{\tau(\eta)}\log_p((u_{\chi\eta})^{\pm 1}\cdot\bar{u}_{\chi\eta})\cdot j\left(\frac{L'((\chi\eta)^{-1},0)}{\log_\infty(u_{\chi\eta}\cdot \bar{u}_{\chi\eta})}\right),
	\end{equation}
where $u_{\chi\eta}$ is a certain $\chi\eta$-unit \eqref{eq:u_chi} and $\bar{u}_{\chi\eta}$ denotes its complex conjugate. 

Moreover, $\mathds{1}(\theta^-_\chi)=0$, and we have
\begin{equation}\label{eq:CORO:EZC}
	\dfrac{\diff}{\diff \!s}\kappa^s(\theta^-_\chi)\bigg|_{s=0} =\left(1+\frac{1}{p}\right)\frac{2}{\ord_p(u_{\gp,\chi})}\det\begin{pmatrix}\log_p(u_\chi) & \log_p(u_{{\gp},\chi}) \\ \log_p(\bar{u}_{\chi}) & \log_p(\bar{u}_{{\gp},\chi})\end{pmatrix}\cdot j\left(\frac{L'(\chi^{-1},0)}{\log_\infty(u_{\chi}\cdot \bar{u}_{\chi})}\right), 
\end{equation}
where $u_{{\gp},\chi}$ is an analogue of a Gross-Stark unit \eqref{eq:u_p_chi}.
\end{CORO} 
The interpolation property \eqref{eq:CORO:EX} follows from Rubin-Stark's conjecture for imaginary quadratic fields, while \eqref{eq:CORO:EZC} is derived from the recent work of Bullach and Hofer \cite{bullachhofer} on the Iwasawa-theoretic Mazur-Rubin-Sano conjecture.

We conclude this introduction with the proof of the Gross-Kuz'min conjecture for abelian extensions of imaginary quadratic fields (Theorem \ref{thm:gross_kuzmin}).
\begin{THM}\label{thm:intro_D}
	Let $p$ be an odd prime and let $H$ be an abelian extension of an imaginary quadratic field $k$. Denote by $H_\infty$ the $\Zp$-cyclotomic extension of $H$, and by $M_\infty$ the maximal abelian pro-$p$ extension of $H_\infty$ that is unramified everywhere and completely split at $p$. Then the module of $\Gal(H_\infty/H)$-coinvariants of $\Gal(M_\infty/H_\infty)$ is finite. 
\end{THM}
So far, the Gross-Kuz'min conjecture has only been settled for abelian extensions of $\bQ$ \cite{greenberg1973} and under some rather restrictive assumptions on the splitting behaviour of $p$ and $\infty$ in the number field \cite{jaulent2003,kuzmin18,kleine}. Besides its applications to classical Iwasawa theory \cite{kuzmin,gross1981padic,federergross,kolster,nguyenquangdocapitulation,jaulent2017}, the Gross-Kuz'min conjecture is also an essential input in the strategy presented in \cite{BKSANT} for proving the equivariant Tamagawa Number Conjecture (eTNC). 

 By the descent formalism of Selmer complexes, the Gross-Kuz'min conjecture is implied by the non-vanishing of certain $\cL$-invariants (Theorem \ref{thm:exact_order_of_vanishing} (2)). This implication can be seen as a generalization to arbitrary number fields of classical results by Kolster and Federer-Gross for totally real and CM fields. The proof then consists in showing that these $\cL$-invariants do not vanish. To do so, we adapt a trick in \cite{betinadimitrovKatz} which ultimately relies on Roy's strong six exponential theorem \cite{roy}. 


\subsection{Organisation of the paper}
The plan of the paper is as follows. In \S\ref{sec:coleman} we define extended Coleman maps and we compute their constant terms. In \S\ref{sec:iwasawa} we introduce the objects which play a role in Conjecture \ref{conj:IMC} and we develop tools from Iwasawa theory in order to study our Selmer groups. In \S\ref{sec:conjectures_on_RS_elts} we recall some conjectures on Rubin-Stark elements that appear in the statement of Theorem \ref{thm:intro_C}. We study in \S\ref{sec:monomial_representations} the precise link between these conjectures and Conjecture \ref{conj:IMC}. In \S\ref{sec:examples} we specialize our study to various classical contexts, and we also give an application of our results to the Gross-Kuz'min conjecture.

\subsection*{Acknowledgments}
I would like to thank Denis Benois and Martin Hofer for their helpful comments on an earlier version of this article, and Dominik Bullach for his explanations on the Rubin-Stark conjecture. 
This research is supported by the Luxembourg National Research Fund, Luxembourg, INTER/ANR/18/12589973 GALF.

\section{Coleman Maps}\label{sec:coleman}
\subsection{Classical results on Coleman maps}\label{sec:coleman_intro}
Fix an odd prime number $p$. Let $K$ be an unramified finite extension of $\Qp$ and let $\varphi \in \Gal(K/\Qp)$ be the Frobenius automorphism.
The Galois group $\Gal(K(\mu_{p^\infty})/\Qp)$ splits into a product $\Gal(K/\Qp) \times \Gamma_{{\cyc}} \simeq \Gal(K/\Qp) \times \Gamma \times \Gal(K(\mu_p)/K)$, where we have put $\Gamma_{{\cyc}}=\Gal(K(\mu_{p^\infty})/K)$ and where $\Gamma$ is the Galois group of $\Zp$-cyclotomic extension of $K$.
The cyclotomic character induces an isomorphism $\chi_{\cyc}:\Gamma_{{\cyc}} \simeq \bZ^\times_p$ and $\Gamma$ can be identified with $1+p\Zp\subset \bZ^\times_p$. Let us fix once and for all a system of compatible roots of unity $\tilde{\zeta}:=(\zeta_n)_{n\geq 0}$ of $p$-power order (\textit{i.e.}, $\zeta_{n+1}^p=\zeta_n$, $\zeta_0=1$ and $\zeta_1 \neq 1$). Note that $\tilde{\zeta}$ is also norm-coherent, which means that it lives in the inverse limit $\varprojlim_n {K(\mu_{p^n})}^\times$ where the transition maps are the local norms. Coleman proved the following theorem (see \cite[Thm. A]{coleman} or \cite[Chapter I, §2]{deshalit}).

\begin{theorem}[Coleman]\label{thm:coleman_power_series}
	Let $v=(v_n)_{n\geq 0} \in \varprojlim_n {K(\mu_{p^n})}^\times$. There exists a unique power series $f_v(T) \in T^{\ord(v_0)}\cdot\cO_K[[T]]^\times$, called \emph{Coleman's power series of $v$}, which satisfies the following properties:
	\begin{itemize}
		\item for all $n\geq 1$, one has $f_v(\zeta_n-1)=\varphi^n(v_n)$,
		\item $\prod_{i=0}^{p-1}f_v(\zeta_1^i(1+T)-1)=\varphi(f_v)((1+T)^p-1)$, where we let $\varphi$ act on the coefficients of $f_v$.
	\end{itemize}
	Furthermore, the map $v\mapsto f_v(T)$ is multiplicative and it is compatible with the action of $\Gamma_{{\cyc}}$, \textit{i.e.}, one has $f_{\gamma(v)}(T)=f_v((1+T)^{\chi_{\cyc}(\gamma)}-1)$ for all $\gamma\in \Gamma_{\cyc}$.
\end{theorem}
Note that $f_v(0)$ and $v_0$ are not equal but are simply related by the formula $v_0=(1-\varphi^{-1})f_v(0)$ which is easily deduced from the second property satisfied by $f_u$. To better understand Coleman's power series, we consider the operator $\cL : \cO_K[[T]]^\times \longrightarrow K[[T]]$ given by 
$$\cL(f)(T)= \frac{1}{p}\log_p\left(\frac{f(T)^p}{\varphi(f)((1+T)^p-1)}\right),$$ 
where the $p$-adic logarithm $\log_p : \cO_K[[T]]^\times \longrightarrow K[[T]]$ is defined by $\log_p(\zeta)=0$ for $\zeta \in \mu(K)$ and $\log_p(1+Tg(T))= Tg(T)-\frac{1}{2}T^2g(T)^2+\frac{1}{3}T^3g(T)^3 + \ldots$ for $g(T)\in \cO_K[[T]]$. Moreover, it is known that $\cL$ takes values in $\cO_K[[T]]$. 

We can now define what is usually referred to as the Coleman map. One can associate to any $\cO_K$-valued measure $\lambda$ over $\Zp$ its Amice transform $\cA_\lambda(T)\in \cO_K[[T]]$ given by
$$\cA_\lambda(T) = \sum_{n=0}^{\infty} \int_\Zp \binom{x}{n}T^n\lambda(x) = \int_\Zp (1+T)^x\lambda(x).$$
This construction yields an isomorphism of $\cO_K$-algebras (the product of measures being the convolution product) $\cO_K[[\Zp]]\simeq \cO_K[[T]]$. When the Amice transform of a measure $\lambda$ is equal to $\cL(f_u(T))$ for some norm-coherent sequence of \textit{units} $u \in \varprojlim_n \cO_{K(\mu_{p^n})}^\times$, then one can show that $\lambda$ is actually the extension by zero of a measure over $\bZ^\times_p$. It is convenient to see such a $\lambda$ as a measure over $\Gamma_{{\cyc}}$ after taking a pull-back by $\chi_{\cyc}$. Let 
$$ \Coleman : \varprojlim_n \cO_{K(\mu_{p^n})}^\times \longrightarrow \cO_K[[\Gamma_{\cyc}]] $$
be the map sending $u$ to the $\cO_K$-valued measure $\lambda$ over $\Gamma_{\cyc}$ whose Amice transform is $\cL(f_u(T))$. Note that $\Coleman(uu')=\Coleman(u)+\Coleman(u')$ for norm-coherent sequences of units $u$ and $u'$.
\
\begin{proposition}
	There is an exact sequence of $\Gal(K/\Qp)\times \Gamma_{\cyc}$-modules 
	$$\xymatrix{
		0 \ar[r] & \mu(K) \times \Zp(1) \ar[r] & \varprojlim_n \cO_{K(\mu_{p^n})}^\times \ar[r]^\Coleman & \cO_K[[\Gamma_{\cyc}]] \ar[r] & \Zp(1) \ar[r] & 0
	}$$ 
	Here, the first map sends $\xi \in \mu(K)$ to $(\xi)_n$ and $a\in \Zp$ to $(\zeta_n^a)_n$, and the last one sends $\lambda$ to $\Tr_{K/\Qp}\int_{\Gamma_{\cyc}} \chi_{\cyc}(\sigma)\lambda(\sigma)$. Moreover, $\Gamma_{{\cyc}}$ acts on $\cO_K[[\Gamma_{\cyc}]]$ via $\int_{\Gamma_{\cyc}} f(\sigma)(\gamma \cdot \lambda)(\sigma) = \int_{\Gamma_{\cyc}} f(\gamma\sigma)\lambda(\sigma)$ for all $\gamma\in \Gamma_{{\cyc}}$ and all $\Qp$-valued continuous functions $f$ on $\Gamma_{\cyc}$.
\end{proposition}
\begin{proof}
	See \cite[Thm. 3.5.1]{coatessujathabook}.
\end{proof}

Now take invariants by $\Delta:=\Gal(K(\mu_p)/K)$. As $\Delta$ is of prime-to-$p$ order, one can identify $\cO_K[[\Gamma_{\cyc}]]^{\Delta}$ with $\cO_K[[\Gamma]]$ via the map sending $\lambda$ to the measure $\mu$ defined by $\int_{\Gamma}g(\gamma)\mu(\gamma)= \frac{1}{p-1}\int_{\Gamma_{\cyc}}g(\sigma \!\!\!\mod \Delta)\lambda(\sigma)$ for all continuous $g:\Gamma \longrightarrow \Qp$. The inverse map is given by the formula $\int_{\Gamma_{\cyc}} f(\sigma)\lambda(\sigma)=\int_{\Gamma} Tf(\gamma)\mu(\gamma)$, where $f:\Gamma_{\cyc} \longrightarrow \Qp$ is continuous and where $Tf: \Gamma \longrightarrow \Qp$ is the sum $Tf(\gamma)=\sum_{\delta\in\Delta} f(\gamma\delta)$. 
Let us denote by $K_n=K(\mu_{p^{n+1}})^\Delta$ the $n$-th layer of the $\Zp$-cyclotomic extension of $K$ for all $n\geq 0$, and let $K_{-1}=K_0=K$. The Coleman map thus restricts to a surjective map $\Coleman : \varprojlim_n \cO_{K_{n-1}}^\times \longrightarrow \cO_K[[\Gamma]]$ whose kernel is identified with $\mu(K)$. By restricting further to principal units $\cO^{\times,1}_{K_n}\subseteq \cO_{K_{n}}^\times$ (\textit{i.e.}, units that are congruent to 1 modulo the maximal ideal of $\cO_{K_{n}}$), one obtains an isomorphism 
\begin{equation}\label{eq:coleman_map}
\Coleman : \varprojlim_n \cO^{\times,1}_{K_{n-1}} \overset{\simeq}{\longrightarrow} \cO_K[[\Gamma]]
\end{equation}
of $\Zp[\Gal(K/\Qp)][[\Gamma]]$-modules (or, equivalently, of $\varphi$-linear $\Zp[[\Gamma]]$-modules) whose properties are now recalled. For $n\geq2$, one may see via $\iota_p$ any non-trivial $\ob{\bQ}$-valued Dirichlet character $\eta$ of $p$-power order and of conductor $p^n$ as a $p$-adic character of $\Gamma_{n-1}=\Gal(K_{n-1}/K)$ which does not factor through $\Gamma_{n-2}$. We denote by $e_\eta:= p^{1-n}\sum_{g\in \Gamma_{n-1}} \eta(g^{-1})g \in \Qp(\mu_{p^{n}})[\Gamma_{n-1}]$ the idempotent attached to $\eta$, and by $\g(\eta)=\sum_{a \mod p^n} \eta(a) \zeta_n^a $ its usual Gauss sum. 

\begin{lemme}\label{lem:val_spé_coleman}
	Let $u=(u_n)_{n\geq 1} \in \varprojlim_n \cO^{\times,1}_{K_{n-1}}$ and let $\mu=\Coleman(u)$. For all non-trivial characters $\eta:\Gamma\longrightarrow \ob{\bQ}^\times$ of conductor $p^n$, one has
	$$\int_{\Gamma}\eta(\gamma)\mu(\gamma) = \frac{p^{n-1}}{\g(\eta^{-1})} \cdot \varphi^n\left(e_\eta\log_p(u_n)\right).$$
	Moreover, if $u_0\neq1$, then $u_0\notin1+p\Zp$ and one has
	$$\int_{\Gamma}\mu = \frac{1-p^{-1}\varphi}{1-\varphi^{-1}}\left(\log_p(u_1)\right).$$
\end{lemme}
\begin{proof}
	When $\eta$ is non-trivial, one checks that
	
	\begin{align*}
	\g(\eta^{-1})\ \int_{\Gamma}\eta(\gamma)\mu(\gamma) &=(p-1)^{-1}\ \sum_{a\mod p^n} \eta^{-1}(a)\cdot\cL(f_u)(\zeta_n^a-1) \\
	&= p^{n-1}\cdot e_\eta\ \cL(f_u)(\zeta_n-1) \\
	&= p^{n-1}\cdot e_\eta\  \varphi^n\left(\log_p(u_n)-\frac{1}{p}\log_p(u_{n-1})\right).
	\end{align*}
	Since $\eta$ is of conductor $p^n$, the idempotent $e_\eta$ kills $u_{n-1}$, so the first equality holds. Assume now that $u_0\neq 1$. If $u_0\in 1+p\Zp$, then $u_0^{[K:\Qp]}=N_{K/\Qp}(u_0)$ would be a universal norm in $\bQ_{p,\infty}/\Qp$. As it is also a principal unit, the exact sequence (\ref{eq:exact_seq_local_norms}) below shows that it should be equal to $1$, and so does $u_0$. Hence $u_0 \notin 1+p\Zp$, and we have
	$$  \int_{\Gamma}\mu = (p-1)^{-1}\cdot\cL(f_u)(0)=(p-1)^{-1}\cdot(1-p^{-1}\varphi)\log_p(f_u(0))=\frac{1-p^{-1}\varphi}{1-\varphi^{-1}}\left(\log_p(u_1)\right),$$
	because of the relation $u_1^{p-1}=u_0=(1-\varphi^{-1})f_u(0)$.
\end{proof}

\subsection{Isotypic components}\label{sec:coleman_isotypic}
Let $\delta : G\longrightarrow \cO^\times$ be any character of a finite group $G$ which takes values in a finite flat extension $\cO$ of $\Zp$. There are two slightly different notions of $\delta$-isotypic components for an $\cO[G]$-module $M$, namely
$$M^{\delta}:=\{m\in M\ \big| \ \forall g\in G,\ g.m=\delta(g)m\}, \quad \textrm{or,}\quad M_{\delta}:=M \otimes_{\cO[G]} \cO,$$ 
where the ring homomorphism $\cO[G]\longrightarrow\cO$ is the one induced by $\delta$. The modules $M^{\delta}$ and $M_{\delta}$ are respectively the largest submodule and the largest quotient of $M$ on which $G$ acts via $\delta$ (see \cite[§2]{tsujisemilocal}) and they are simply called the $\delta$-part and the $\delta$-quotient of $M$. When $p$ does not divide the order of $G$, the natural map $M^{\delta}\longrightarrow M_{\delta}$ is an isomorphism, and $M^{\delta}$ equals $e_\delta M$, where $e_\delta=\#(G)^{-1}\sum_{g\in G} \delta(g^{-1})g\in\cO[G]$ is the usual idempotent attached to $\delta$.

When $\cO$ contains $\cO_K$ and for $G=\Gal(K/\Qp)$, $M=\cO_K \otimes_{\Zp} \cO$ as in \S\ref{sec:coleman_intro}, the same argument as in the proof of \cite[Lemme 3.2.3]{maks} shows that the internal multiplication $M\longrightarrow \cO$ induces by restriction an $\cO$-linear isomorphism $M^{\delta}\simeq \cO$. When one moreover fixes an isomorphism $\cO_K\simeq \Zp[G]$ (given by a normal integral basis of the unramified extension $K/\Qp$) one also has $M_{\delta}\simeq \cO[G] \otimes_{\cO[G]}\cO = \cO$.

\begin{definition}\label{def:coleman_isotypic}
	For any character $\delta$ of $G=\Gal(K/\Qp)$ with values in a finite flat $\Zp$-algebra $\cO$ containing $\cO_K$, we define $\delta$-isotypic component of the Coleman map to be the composite isomorphisms of $\cO[[\Gamma]]$-modules
	$$
	\Coleman^{\delta} : \varprojlim_n \left(\cO^{\times,1}_{K_{n-1}} \otimes_{\Zp}\cO\right)^\delta \overset{\Coleman}{\longrightarrow} \left(\cO_K \otimes_{\Zp} \cO\right)^\delta[[\Gamma]]\simeq \cO[[\Gamma]], 
	$$
	where $\Coleman$ is the restriction of isomorphism (\ref{eq:coleman_map}) to the $\delta$-parts and  the last isomorphism is induced by the internal multiplication $\cO_K \otimes_{\Zp} \cO \longrightarrow \cO$.
\end{definition} 
When $\delta$ is trivial, it will later be convenient to consider a natural extension of the map $\Coleman^\mathds{1}$ which we construct in the rest of the paragraph. 

We let $\widehat{\bQ}^\times_{p,n}$ be the pro-$p$ completion of $\bQ^\times_{p,n}$ for all $n\geq 0$. Concretely, once we fix a uniformizer $\varpi_n$ of ${\bZ}_{p,n}:={\cO}_{\bQ_{p,n}}$, we have $\widehat{\bQ}^\times_{p,n}={\bZ}^{\times,1}_{p,n} \times \varpi_n^{\Zp}$, while $\bQ_{p,n}^\times={\bZ}^{\times}_{p,n} \times \varpi_n^{\bZ}$. Let $\cA:=\ker\left(\Zp[[\Gamma]]\longrightarrow\Zp\right)$ be the augmentation ideal of the Iwasawa algebra $\Zp[[\Gamma]]$. We still denote by $\gamma$ the image of an element $\gamma\in \Gamma$ under the canonical injection $\Gamma \hookrightarrow \cO[[\Gamma]]^\times$.
\begin{lemme}\label{lem:coleman_11}
	The multiplication map 
	$$m:\cA \otimes_{\Zp[[\Gamma]]} \varprojlim_n \widehat{\bQ}^\times_{p,n-1}\otimes \cO \longrightarrow\varprojlim_n \widehat{\bQ}^\times_{p,n-1}\otimes\cO$$
	given by $a\otimes v \otimes 1 \mapsto a\cdot v\otimes 1$ is injective and has image $\varprojlim_n \bZ^{\times,1}_{p,n-1}\otimes \cO$.
\end{lemme}
\begin{proof}
		We may assume $\cO=\Zp$. We first check that $m$ is injective and we fix a topological generator $\gamma$ of $\Gamma$. As $\cA$ is $\Zp[[\Gamma]]$-free and generated by $\gamma-1$, any element of $\cA \otimes_{\Zp[[\Gamma]]} \varprojlim_n \widehat{\bQ}^\times_{p,n-1}$ can be written as a pure tensor $(\gamma-1)\otimes v$ for some $v$. If $(\gamma-1)\cdot v=1$, then $v_n\in \widehat{\bQ}^\times_p$ for all $n\geq 1$ hence $v=1$ since $\widehat{\bQ}^\times_p$ has no non-trivial $p$-divisible element. Thus $m$ is injective. 
		Inflation and restriction maps in discrete group cohomology provide an exact sequence
	\begin{equation*}
	\begin{tikzcd}
	0 \ar[r] & \HH^1(\Gamma,\Qp/\Zp) \ar[r, "\inf"] & \HH^1(\Qp,\Qp/\Zp) \ar[r, "\res"] & \HH^1(\bQ_{p,\infty},\Qp/\Zp)^{\Gamma} \ar[r] & \HH^2(\Gamma,\Qp/\Zp).
	\end{tikzcd}
	\end{equation*}
	As $\Gamma$ is pro-cyclic, the last term vanishes. Moreover, the Galois action on $\Qp/\Zp$ being trivial, all the $\HH^1$'s involved are $\Hom$'s and one easily deduces from local class field and from exactness of Pontryagin functor $\Hom(-,\Qp/\Zp)$ (see e.g. \cite[Thm. 6]{flood}) the following short exact sequence
	\begin{equation}\label{eq:exact_seq_local_norms}
	\begin{tikzcd}
	0 \ar[r] & \left(\varprojlim_n \widehat{\bQ}^\times_{p,n-1}\right)_{\Gamma} \ar[r, "v\mapsto v_1\;\;"] &  \widehat{\bQ}^\times_{p}\ar[r, "\rec\;\;"] & \Gamma \ar[r] & 0,
	\end{tikzcd}
	\end{equation}
	where the subscript $\Gamma$ means that we took $\Gamma$-coinvariants and where $\rec$ is the local reciprocity map. Since $M_\Gamma=M/\cA M$ for any $\Zp[[\Gamma]]$-module $M$ and since the map $\rec_{|1+p\Zp}$ is an isomorphism, it follows that both $\varprojlim_n \bZ^{\times,1}_{p,n-1}$ and the image of $m$ coincide with the kernel of the map $\varprojlim_n \widehat{\bQ}^\times_{p,n-1} \longrightarrow  \widehat{\bQ}^\times_{p}$ given by $v\mapsto v_1$, so they are equal.
\end{proof}
Let $\cO$ be any finite flat $\Zp$-algebra and let $\cI_\cO \subseteq \Frac(\cO[[\Gamma]])$ be the principal ideal of $\cO[[\Gamma]]$ consisting of quotients of $p$-adic measures on $\Gamma$ with at most one simple pole at the trivial character. Then $\cI_\cO$ is invertible, and its inverse is the augmentation ideal $\cA_\cO=\cA\otimes \cO$ of $\cO[[\Gamma]]$. The following definition makes sense by Lemma \ref{lem:coleman_11}.
\begin{definition}\label{def:coleman_isotypic_extended}
	The extended $\mathds{1}$-isotypic component of the Coleman map is the isomorphism
	$$\widetilde{\Coleman}^\mathds{1} : \varprojlim_n \widehat{\bQ}^\times_{p,n-1} \otimes \cO \overset{\simeq}{\longrightarrow} \cI_\cO$$
	given by $\widetilde{\Coleman}^\mathds{1}(v)=\frac{1}{a}\Coleman^\mathds{1}(av)$ for any non-zero element $a\in \cA_\cO$. As $\Coleman^\mathds{1}$ is $\cO[[\Gamma]]$-linear, this definition does not depend on the choice of $a$. 
	
	For any non-trivial $\cO$-valued character $\delta$ of $\Gal(K/\Qp)$, we also let $\widetilde{\Coleman}^\delta=\Coleman^\delta$.
\end{definition}

\subsection{Constant term of Coleman maps}\label{sec:coleman_cst_term}
We keep the same notations as in \S\ref{sec:coleman_intro} and \S\ref{sec:coleman_isotypic}. Fix a finite flat extension $\cO$ of $\Zp$ which contains $\cO_K$ and let $\delta$ be an $\cO$-valued character of $\Gal(K/\Qp)$. Then $\delta$ is the trivial character if and only if $\beta=\delta(\varphi)\in\cO^\times$ is equal to $1$. When $\delta$ is non-trivial, Lemma \ref{lem:val_spé_coleman} shows that the constant term of $\mu=\Coleman^\delta(u)$ is given by 
$$\int_{\Gamma}\mu = \frac{1-p^{-1}\beta}{1-\beta^{-1}}\log_p(u_1).$$
We now give a similar formula when $\delta$ is trivial which is a variant of  \cite[Prop. 1.3.7]{benois2014}. 
\begin{lemme}\label{lem:coleman_ord}
    Let $u\in \varprojlim_n \bZ^{\times,1}_{p,n-1}\otimes \cO$ and put $\mu=\Coleman^{\mathds{1}}(u)$. Write $u$ as $(\gamma-1)\cdot v$ for some $\gamma\in\Gamma$ and $v\in\varprojlim_n \widehat{\bQ}^\times_{p,n-1} \otimes \cO$. Then 
	\begin{equation}\label{eqn:terme_constant_coleman1}
	\int_{\Gamma} \mu = \left(1-p^{-1}\right)\cdot\log_p(\chi_{\cyc}(\gamma))\cdot\ord_p(v_1),
	\end{equation}
	where $\ord_p:\widehat{\bQ}^\times_p \otimes_{\Zp}\cO \longrightarrow \cO$ is the usual $p$-valuation map. 
\end{lemme}
\begin{proof}
 We may take $\cO=\Zp$. Let us choose any $\pi=(\pi_n)_{n\geq 1}\in \varprojlim_n \widehat{\bQ}^\times_{p,n-1}$ such that  $\pi_1=p$ (this is possible thanks to the short exact sequence (\ref{eq:exact_seq_local_norms})). Then $\pi\in \varprojlim_n \bQ^\times_{p,n-1}$ and $f_\pi(T)=T^{p-1}U(T)$ for some $U(T)\in \Zp[[T]]^\times$, so we have
	$$f_{(\gamma-1)\cdot\pi}(T)=\left(\frac{(1+T)^{\chi_{cyc}(\gamma)}-1}{T}\right)^{p-1}\cdot \frac{U((1+T)^{\chi_{cyc}(\gamma)}-1)}{U(T)}.$$
	Hence $\log_p(f_{(\gamma-1)\cdot\pi}(0))=(p-1)\cdot\log_p(\chi_{\cyc}(\gamma))$. If we write $v$ as $u' \cdot\pi^{\ord_p(v_1)}$ with $u'\in \varprojlim_n \bZ^{\times,1}_{p,n-1}$, then one has by linearity
	\[
	\begin{split}
	\int_{\Gamma}\mu &=\int_{\Gamma} \Coleman^{\mathds{1}}((\gamma-1)\cdot u') \quad\;\qquad+ \ord_p(v_1)\int_{\Gamma}\Coleman^{\mathds{1}}((\gamma-1)\cdot \pi)\\
	&= \int_{\Gamma} \left(\gamma\cdot \Coleman^{\mathds{1}}(u')-\Coleman^{\mathds{1}}(u')\right) + \ord_p(v_1)(p-1)^{-1}\cdot\cL(f_{(\gamma-1)\cdot \pi})(0)\\
	&= 0\qquad\qquad\qquad\qquad \qquad  \quad\!\!+ \ord_p(v_1)(1-p^{-1})(p-1)^{-1}\cdot\log_p(f_{(\gamma-1)\cdot\pi}(0))\\
	&= \left(1-p^{-1}\right)\cdot\log_p(\chi_{\cyc}(\gamma))\cdot\ord_p(v_1).  \end{split}\]\end{proof}

\begin{remark}
The dependence on $\gamma$ in \eqref{eqn:terme_constant_coleman1} is only apparent. Indeed, assume $$(\gamma-1)\cdot v=(\gamma^c-1)\cdot v'$$ 
	for $\gamma\in\Gamma-\{1\}$, $c\in\Zp-\{0\}$ and $v,v'\in\varprojlim_n \widehat{\bQ}_{p,n-1}^\times$. Then $\log_p(\chi_{\cyc}(\gamma^c))=c\cdot \log_p(\chi_{\cyc}(\gamma))$, and 
	$$\ord_p(v_1')= \ord_p((\frac{\gamma-1}{\gamma^c-1}\cdot v)_1)=  \ord_p(((c^{-1}+(*)(\gamma-1))\cdot v)_1)= \ord_p(v_1^{c^{-1}})=c^{-1}\cdot \ord_p(v_1),$$
	because $v_1^{\gamma-1}\in \bZ_{p}^{\times,1}$. Thus we see that the $c$'s cancel out, as predicted by \eqref{eqn:terme_constant_coleman1}.
\end{remark}

\section{Cyclotomic Iwasawa theory for Artin motives}\label{sec:iwasawa}
 Thoughout this section we keep the notations of the introduction: in particular, $\rho$ does not contain the trivial representation and is unramified at our fixed prime $p>2$. Without loss of generality, we also assume that $E$ contains the field $H$ cut out by $\rho$, so the completion $E_p=\ob{\iota_p(E)}$ contains $K:=\ob{\iota_p(H)}$. 

\subsection{$p$-stabilizations} \label{sec:p_stab}
\begin{definition}\label{def:p_stab}
	A \textit{$p$-stabilization} $(\rho^+,W_p^+)$ of $(\rho,W_p)$ is an $E_p$-linear subspace $W_p^+$ of $W_p$ of dimension $d^+$ which is stable under the action of the local Galois group $\Gal(\ob{\bQ}_p/\Qp)$. We say that $\rho^+$ is:
	\begin{enumerate}
		\item \textit{motivic} if $W_p^+$ is of the form $E_p \otimes_{E,\iota_p}W^+$, where $W^+$ is an $E$-linear subspace of $W$, and
		\item $\eta$-\textit{admissible} (for a given character $\eta\in\widehat{\Gamma}$) if the $p$-adic pairing (\ref{eq:intro_p_adic_pairing}) is non-degenerate.
	\end{enumerate} 
\end{definition}
To explain in greater details the $\eta$-admissibility property, consider a $p$-stabilization $W_p^+$ and fix a character $\eta\in\widehat{\Gamma}$. If we let $\omega_{\ff,\eta}=\psi_1\wedge\ldots\wedge\psi_{d^+}$ be a basis of $\det_E \HH^1_\ff((\rho\otimes\eta)^*(1))=\det_E \Hom_{G_{\bQ}}(W_\eta,E_\eta\otimes\cO_{H_\eta}^\times)$ and $\omega_p^+=w_1\wedge\ldots\wedge w_{d^+}$ be a basis of $\det_{E_p} W_p^+$, then the determinant of the $p$-adic pairing computed in these bases is given by 
\begin{equation}\label{eq:expression_reg_p}
\Reg_{\omega_p^+}(\rho\otimes\eta)=\det\left(\log_p(\psi_j(w_i))\right)_{1\leq i,j\leq d^+}\in E_{p,\eta}.
\end{equation}
Here, we denoted by $E_{p,\eta}\subseteq\ob{\bQ}_p$ the compositum of $E_p$ with the image of $\eta$, and by
\begin{equation}\label{eq:def_log_p}
\log_p : \ob{\bQ}_p \otimes_{\bZ} \ob{\bZ}^\times \longrightarrow \ob{\bQ}_p
\end{equation}
the map given by $\log_p(x\otimes a)=x\cdot \log_p^\textrm{Iw}(\iota_p(a))$, where $\log_p^\textrm{Iw}:\ob{\bQ}^\times_p\longrightarrow\ob{\bQ}_p$ is Iwasawa's $p$-adic logarithm. The $p$-regulator $\Reg_{\omega_p^+}(\rho\otimes\eta)$ does not vanish for (some, hence) all choices of bases if and only $W_p^+$ is $\eta$-admissible. 

We now recall a classical conjecture in $p$-adic transcendence theory \cite[Conjecture 3.10]{mazurcalegari}:

\begin{conjecture}[Weak $p$-adic Schanuel conjecture]\label{conj:schanuel}
	Let $a_1,\ldots,a_n$ be $n$ non-zero algebraic numbers contained in a finite extension $F$ of $\Qp$. If $\log_p^\Iw(a_1),\ldots,\log_p^\Iw(a_n)$ are linearly independent over $\bQ$, then the extension field $\bQ(\log_p^\Iw(a_1),\ldots,\log_p^\Iw(a_n))\subset F$ has transcendence degree $n$ over $\bQ$.
\end{conjecture}

\begin{lemme}\label{lem:schanuel+motivic_implies_admissible}
	Assume that $\rho$ is irreducible and that $W_p^+$ is motivic. Then $W_p^+$ is $\eta$-admissible for all characters $\eta\in\widehat{\Gamma}$ if $d^+=1$, or if Conjecture \ref{conj:schanuel} holds.
\end{lemme}
\begin{proof}
	 Since $W_p^+$ is motivic, one may choose a basis $\omega_p^+=w_1\wedge \ldots \wedge w_{d^+}$ of $\det_{E_p} W_p^+$ with $w_i\in W$ for $1\leq i \leq d^+$. Fix a character $\eta\in\widehat{\Gamma}$ and let $e_\eta$ be the idempotent associated with the irreducible representation $\rho\otimes\eta$. We have an isomorphism of $E_\eta[G_{\bQ}]$-modules
	 \[f \colon e_\eta \cdot\left(E_\eta \otimes \cO_{H_\eta}^\times \right) \simeq (W_\eta)^{\oplus d^+},\]
	 which we fix. We then define a basis $\omega_{\ff,\eta}=\psi_1\wedge\ldots\wedge\psi_{{d^+}}$ of $\det_{E_\eta}\HH^1_\ff((\rho\otimes\eta)^*(1))$ by letting $\psi_i = f^{-1}\circ p_i$ for $1\leq i \leq d^+$, where $p_i\colon W_\eta \hookrightarrow W_\eta^{d^+}$ is the inclusion of the $i$-th copy of $ W_\eta$ inside $W_\eta^{d^+}$. Observe that $\psi_i$ is injective by construction and that the linear subspaces $\im \psi_1,\ldots,\im \psi_{d^+}$ are in direct sum, so the family $(\psi_j(w_i))_{i,j}$ is $E_\eta$-linearly independent. Brumer's theorem \cite{brumer} implies that the restriction of (\ref{eq:def_log_p}) to $\ob{\bQ} \otimes \cO_{H_\eta}^\times$ is injective, so \eqref{eq:expression_reg_p} expresses $\Reg_{\omega_p^+}(\rho\otimes\eta)$ as a polynomial in  $p$-adic logarithms of  units in $\ob{\bQ} \otimes \cO_{H_\eta}^\times$ that are $\ob{\bQ}$-linearly independent. Therefore, $\Reg_{\omega_p^+}(\rho\otimes\eta)$ does not vanish if Conjecture \ref{conj:schanuel} holds, or if $d^+=1$.
\end{proof}
We end this section by exploring the link between the $\eta$-admissibility of a $p$-stabilization $W_p^+$ and the Leopoldt conjecture for $H_\eta$ and $p$. As in \cite[Prop. 2.9(ii)]{maksouddimitrov}, $W_p^+$ is $\eta$-admissible if and only if $E_{p,\eta}\otimes W_p^+$ does not meet the linear subspace 
\[W_p^{\circ}=\bigcap_{\psi}\ker\left[\log_p \circ \psi : W_{p,\eta} \longrightarrow E_{p,\eta}\otimes \cO_{H_\eta}^\times \longrightarrow E_{p,\eta}\right], 
\]
where the intersection runs over all  $\psi$ in $ E_{p,\eta}\otimes\HH^1_\ff((\rho\otimes\eta)^*(1))$. The dimension of $W_p^{\circ}$ can be written as $d^-+s$, where $s$ is the dimension of the kernel of the map \[\alpha_\eta:E_{p,\eta}\otimes\HH^1_\ff((\rho\otimes\eta)^*(1)) \longrightarrow \Hom(W_{p,\eta},E_{p,\eta})\]
induced by $\log_p$.
Kummer theory (as in the proof of Lemma \ref{lem:relation_cohomologie_et_unités}) provides a natural isomorphism
\begin{equation}\label{eq:bloch-kato-kummer}
	\HH^1_\ff(\bQ,W_{p,\eta}^*(1)) \simeq \Hom_{G_{\bQ}}(W_{p,\eta},E_{p,\eta} \otimes \cO_{H_\eta}^\times)=E_{p,\eta}\otimes\HH^1_\ff((\rho\otimes\eta)^*(1)) 
\end{equation}
between the domain of $\alpha_\eta$ and the usual Bloch-Kato Selmer group of the arithmetic dual $W_{p,\eta}^*(1)$ of $W_{p,\eta}$. One may also see $W_{p,\eta}^*(1)$ as $G_{\Qp}$-representation and consider the local Bloch-Kato Selmer group $\HH^1(\Qp,W_{p,\eta}^*(1))$, which is, again by Kummer theory, canonically isomorphic to $\Hom_{G_{\Qp}}(W_{p,\eta},E_{p,\eta}\otimes \cO^{\times,1}_{K_\eta})$, where $\cO^{\times,1}_{K_\eta}$ is the $\Zp$-module of principal units of the completion $K_\eta$ at $p$ of $H_\eta$. Under the above identifications, the map $\alpha_\eta$ is nothing but the composite map 
\begin{equation}\label{eq:loc_p_bloch_kato}
\begin{tikzcd} \HH^1_\ff(\bQ,W_{p,\eta}^*(1)) \ar[r, "\loc_p"] & \HH^1_\ff(\bQ_p,W_{p,\eta}^*(1)) \ar[r, "\simeq"] & \Hom(W_{p,\eta},E_{p,\eta}),
\end{tikzcd}
\end{equation}
where $\loc_p$ is the localization map at $p$, and the second map the isomorphism induced by the $p$-adic logarithm on $K_\eta$, $\log_p : E_{p,\eta} \otimes \cO^{\times,1}_{K_\eta} \longrightarrow E_{p,\eta}$. Let $\cU_{p,\eta}$ be the product of the principal units $\cO^{\times,1}_{K'_\eta}$ of $K'_\eta$, where $K'_\eta$ runs over all the completions of $H_\eta$ at primes above $p$. There is an alternative description of $\HH^1_\ff(\bQ_p,W_{p,\eta}^*(1))$ in terms of semi-local Galois cohomology which identifies $\loc_p$ of (\ref{eq:loc_p_bloch_kato}) with the map 
\begin{equation}\label{eq:loc_p_et_leopoldt}
\begin{tikzcd}\Hom_{G_{\bQ}}(W_{p,\eta},E_{p,\eta}\otimes \cO^\times_{H_\eta}) \rar &\Hom_{G_{\bQ}}(W_{p,\eta},E_{p,\eta}\otimes \cU_{p,\eta})
\end{tikzcd}
\end{equation}
induced by the diagonal embedding $\iota_{\textrm{Leo}}:\cO_{H_\eta}^\times \longrightarrow \cU_{p,\eta}$. The injectivity of $\iota_\textrm{Leo}$, known as the Leopoldt conjecture, implies the injectivity of the map in  (\ref{eq:loc_p_et_leopoldt}) which should be thought as the ``$\rho\otimes\eta$-isotypic component of the Leopoldt conjecture for $H_\eta$ and $p$''. 
\begin{lemme}\label{lem:existence_p_stabilization}
	$W_p$ admits at least one $\eta$-admissible $p$-stabilization if and only if the $\rho\otimes\eta$-isotypic component of the Leopoldt conjecture for $H_\eta$ and $p$ holds, \textit{i.e.}, if the map in  (\ref{eq:loc_p_et_leopoldt}) is injective.
\end{lemme}
\begin{proof}
	There exists at least one $\eta$-admissible $p$-stabilization of $W_p$ if and only if the linear subspace $W_p^{\circ}$ has dimension $d^-$, that is, if the map $\alpha_\eta$ is injective. But we have seen that its injectivity is equivalent to the one of the map in (\ref{eq:loc_p_et_leopoldt}).
\end{proof}

\subsection{Selmer groups} \label{sec:selmer_group}
A Galois-stable lattice of $W_p$ is a free $\cO_p$-submodule of $W_p$ of rank $d$ which is stable by the action of the global Galois group $G_\bQ$. The pair $(\rho,\rho^+)$ will always refer to the choice of:
\begin{enumerate}
	\item a Galois-stable lattice $T_p$ of $W_p$,
	\item a $p$-stabilization $W_p^+$ of $\rho$,
\end{enumerate}

which we will be fixed henceforth. Let $W_p^-=W_p/W_p^+$, $T_p^+=W_p^+ \bigcap T_p$ and $T_p^-=T_p/T_p^+$. We also define $\cO_p$-divisible Galois modules $D_p=W_p/T_p$ and  $D_p^\pm=W_p^\pm/T_p^\pm$. Once we fix a generator of the different of $\cO_p$ over $\Zp$, one may identify $\cO_p^\vee$ with $E_p/\cO_p$ where $M^\vee:=\Hom_\Zp(M,\Qp/\Zp)$ stands for the Pontryagin dual of a $\Zp$-module $M$. This allows us to identify $D_p^\vee$ with $T_p$. Let $n\in \bN \cup\{\infty\}$ and let $I_\ell$ be the inertia subgroup of $\Gal(\ob{\bQ}/\bQ_n)$ at the place above $\ell$ determined by $\iota_\ell$. 

\begin{definition}\label{def:sel}
	The Selmer group of level $n$ attached to $(\rho,\rho^+)$ is defined to be
\[\begin{array}{rcl} \Sel_n(\rho,\rho^+)
&:=& \ker \left[ \HH^1(\bQ_n,D_p) \longrightarrow \HH^1(I_{p},D_p^-) \times \prod_{\ell\neq p} \HH^1(I_{\ell},D_p)  \right].\end{array}\]
The strict Selmer group $\Sel_n^\str(\rho,\rho^+)$ of level $n$ attached to $(\rho,\rho^+)$ is the sub-$\cO_p$-module of $\Sel_n(\rho,\rho^+)$ whose cohomology classes are trivial on the decomposition subgroup at $p$.

The dual Selmer group is defined as the Pontryagin dual of $\Sel_n(\rho,\rho^+)$, that is,
\[X_n(\rho,\rho^+):= \Sel_n(\rho,\rho^+)^\vee = \Hom_\Zp(\Sel_n(\rho,\rho^+),\Qp/\Zp). \]
We also define the strict dual Selmer group of level $n$ by putting $X_n^\str(\rho,\rho^+)=\Sel_n^\str(\rho,\rho^+)^\vee$. 
\end{definition}

By standard properties of discrete cohomology groups, $X_\infty(\rho,\rho^+)$ can be identified with the inverse limit $\varprojlim_n X_n(\rho,\rho^+)$ and it is a finitely generated module over the Iwasawa algebra $\Lambda=\cO_p[[\Gamma]]$.

\begin{lemme}\label{lem:suite_exacte_zeros_triviaux_GV}
	If $X_\infty(\rho,\rho^+)$ is of $\Lambda$-torsion, then there is a short exact sequence of torsion $\Lambda$-modules
	$$\begin{tikzcd}0 \rar & \HH^0(\Qp,T_p^-) \rar & X_\infty(\rho,\rho^+) \rar & X_\infty^\str(\rho,\rho^+) \rar & 0, \end{tikzcd}
	$$
	where the $\Gamma$-action on the first term is trivial. 
\end{lemme}
\begin{proof}
	The first map is obtained by evaluating cocycles at $\sigma_p$ and by applying Pontryagin duality, and the second one is the dual of the inclusion $\Sel_\infty^\str(\rho,\rho^+) \subseteq \Sel_\infty(\rho,\rho^+)$. The only non-obvious statement is the injectivity of the first map, which will follow from \cite[Prop. (2.1)]{greenbergvatsal2000}, but we must check that $\HH^0(\bQ_\infty,\widecheck{D}_p)$ is finite, where  $\widecheck{D}_p=\Hom(T_p,\mu_{p^\infty})$. Since $\rho$ is unramified at $p$ and since it does not contain the trivial representation, $\HH^0({\bQ(\mu_{p^\infty})},D_p)$ is finite, so $\HH^0(\bQ_\infty,\widecheck{D}_p)\subseteq \HH^0(\bQ(\mu_{p^\infty}),\widecheck{D}_p)=\Hom(\HH^0({\bQ(\mu_{p^\infty})},D_p),\mu_{p^\infty})$ is also finite, as wanted.
\end{proof}
The study of the structure of $X_\infty(\rho,\rho^+)$ was investigated in \cite{maks}, where the unramifiedness assumption was partially released (only the quotient $W_p^-$ was assumed to be unramified). However, $\rho$ was taken irreducible and $\rho^+$ motivic. These last hypotheses imply the $\eta$-admissibility of $\rho^+$ for all characters $\eta$ under the Weak $p$-adic Schanuel conjecture (or when $d^+=1$) as in Lemma \ref{lem:schanuel+motivic_implies_admissible}. We recall the results obtained in \textit{loc. cit.}.

\begin{theorem} 
	Assume that $\rho$ is irreducible and that $\rho^+$ is motivic. If $d^+=1$ or if Conjecture \ref{conj:schanuel} holds, then the following four claims are true.
	\begin{enumerate}
		\item The Selmer groups $X_n(\rho,\rho^+)$ are finite for all $n\in\bN$.
		\item The $\Lambda$-module $X_\infty(\rho,\rho^+)$ is torsion and has no non-trivial finite submodules.
		\item Let $\theta^\alg_{\rho,\rho^+}\in\Lambda$ be a generator of its characteristic ideal. Then $\eta(\theta^\alg_{\rho,\rho^+})$ does not vanish for all non-trivial finite order characters $\eta :\Gamma \longrightarrow \ob{\bQ}^\times$.
		\item Let $e=\dim \HH^0(\bQ_p,W_p^-)$. Then $\theta^\alg_{\rho,\rho^+}$ vanishes at the trivial character $\mathds{1}$ if and only if $e=0$. Moreover, $\theta^\alg_{\rho,\rho^+}$ has a zero of order $\geq e$ at $\mathds{1}$, \textit{i.e.},  
		$$\theta^\alg_{\rho,\rho^+} \in \cA^{e},$$
		where $\cA$ is the augmentation ideal of $\Lambda$.
	\end{enumerate}
\end{theorem}
\begin{proof}
	This follows from \cite[Théorème 2.1.5]{maks}.
\end{proof}
\subsection{Artin $L$-functions and Galois-Gauss sums}\label{sec:Galois_Gauss}
We review some classical results on Artin $L$-functions and on Galois-Gauss sums, and we give equivalent reformulations of Conjecture \ref{conj:IMC}. Our main reference is \cite{martinet}. Let $(V,\pi)$ be an Artin representation of $G_{\bQ}$ of dimension $d$ and of Artin conductor $\ff(\pi)$. Put $d^+=\dim\HH^0(\bR,V)$ and $d^-=d-d^+$. The Artin $L$-function of $\pi$ is the meromorphic continuation to $s\in\bC$ of the infinite product (converging for $\Re(s)>1$) over all rational primes
$$L(\pi,s)= \prod_{\ell} \det(1-\ell^{-s}\sigma_\ell \ \big|\ V^{I_\ell})^{-1},$$
where $\sigma_\ell$ is the Frobenius substitution at $\ell$ and  $I_\ell\subseteq G_{\bQ}$ is any inertia group at $\ell$. It is known to satisfy a functional equation which can compactly be written $\Lambda(\pi,1-s)=W(\pi)\Lambda(\pi^*,s)$, where $W(\pi)$ is Artin's root number and $\Lambda(\pi,s)$ is the ``enlarged'' $L$-function. By definition, $\Lambda(\pi,s)$ is equal to the product $\ff(\pi)^{s/2}\Gamma(\pi,s)L(\pi,s)$, where $\Gamma(\pi,s)=\Gamma_\bR(s)^{d^+}\Gamma_\bR(s+1)^{d^-}$ (and $\Gamma_\bR(s):=\pi^{-s/2}\Gamma(s/2)$) is the $L$-factor at $\infty$ of $\pi$. The Galois-Gauss sum of $\pi$ is defined as
$$\tau(\pi)=i^{d^-}\sqrt{\ff(\pi)}W(\pi^*)=i^{d^-}\dfrac{\sqrt{\ff(\pi)}}{W(\pi)}$$
(see \cite[Chapter II, Definition 7.2]{martinet} and the remark that follows). In particular, when $\HH^0(\bQ,V)=0$, the functional equation yields
\begin{equation}\label{eq:eval_equation_fonctionnelle}
L^*(\pi^*,0):=\lim_{s\rightarrow 0}L(\pi^*,s)/s^{d^+}=\frac{\tau(\pi)}{2^{d^+}}\dfrac{L(\pi,1)}{(-i\pi)^{d^-}}.
\end{equation}
\begin{lemme}\label{lem:formule_interpolation_en_s=1}
	We keep the notations of Conjecture \ref{conj:IMC}. The interpolation property of $\theta_{\rho,\rho^+}$ can be written as follows:
	\begin{description}
	\item[$\textbf{EX}_{\rho,\rho^+}$] for all non-trivial characters $\eta\in\widehat{\Gamma}$ of exact conductor $p^{n}$, one has 
	$$\begin{aligned}
		\eta(\theta_{\rho,\rho^+})
		&=\det(\rho^-)(\sigma_p^{-n})\ {\Reg_{\omega_p^+}(\rho\otimes\eta)} \cdot j\left(\frac{\tau(\eta)^{d^-}}{2^{d^+}}\  \dfrac{L\left(\rho \otimes \eta,1\right)}{(-i\pi)^{d^-}\Reg_{\omega_\infty^+}(\rho\otimes\eta)}\right).
	\end{aligned}$$
	\item[$\textbf{EZC}_{\rho,\rho^+}$] If $W_p^+$ is admissible, then $\theta_{\rho,\rho^+}$ has a trivial zero at the trivial character $\mathds{1}$ of order at least $e$, and one has
	$$\begin{aligned}
		\frac{1}{e!} \dfrac{\diff^e}{\diff \!s^e}\kappa^s(\theta_{\rho,\rho^+})\bigg|_{s=0}
		&=(-1)^e\ \cL(\rho,\rho^+)\ \cE(\rho,\rho^+)\ \Reg_{\omega_p^+}(\rho)\  j\left(\dfrac{L\left(\rho,1\right)}{2^{d^+}(-i\pi)^{d^-}\Reg_{\omega_\infty^+}(\rho)}\right).
	\end{aligned}$$
\end{description}

\end{lemme}
\begin{proof}
This follows from (\ref{eq:eval_equation_fonctionnelle}) applied to $\pi=\rho\otimes\eta$.
\end{proof}

\begin{lemme}\label{lem:properties_galois_gauss_sums}\hspace{2em}
	\begin{enumerate}
		\item If $\chi$ is a Dirichlet character, then $\tau(\chi)$ is the usual Gauss sum of $\chi^{-1}$, \textit{i.e.}, $\tau(\chi)=\g(\chi^{-1})$.
		\item We have $\tau(\pi)\in F^\times$ for any splitting field $F\subseteq\ob{\bQ}$ of $\pi$. 
		\item If $\pi$ is unramified at $p$, then $\tau(\pi)$ is a $p$-adic unit. 
		\item Take $\pi=\rho\otimes\eta$ with $\eta\in\widehat{\Gamma}$ and put $N=\ff(\rho)$, $p^n=\ff(\eta)$. Then $$\tau(\rho\otimes\eta)=\tau(\rho)\cdot\g(\eta^{-1})^{d}\cdot\det(\rho)^{-1}(\sigma_p^n)\cdot\eta^{-1}(N).$$
	\end{enumerate}
\end{lemme}
\begin{proof}
	The first statement follows from the well-known fact that $\g(\chi^{-1})=i^{d^-}W(\chi^{-1})\ff(\chi)^{1/2}$, and the second statement from Fröhlich's theorem \cite[Chapter II, Theorem 7.2]{martinet}. For the third statement, recall first that $\tau(\pi)$ is a product over all primes $\ell$ of local Galois-Gauss sums $\tau(\pi_\ell)\in\ob{\bQ}^\times$ attached to the local representation $\pi_\ell$ over $\bQ_\ell$ associated with $\pi$ (see \textit{loc.cit.}, Chapter II, Proposition 7.1). We claim that $\tau(\pi_\ell)$ is only divisible by primes above $\ell$ and that $\tau(\pi_p)=1$. Local Galois-Gauss sums are defined with the aid of Brauer induction from the case of multiplicative characters $\theta$ of $\Gal(\ob{\bQ}_\ell/M)$ for a finite extension $M/\bQ_\ell$ (see \textit{loc.cit.}, Chapter II, §4. and §2.). It is known that $\tau(\theta)$ is an algebraic integer dividing the norm of the local conductor (which is a power of $\ell$), and moreover that $\tau(\theta)=1$ whenever both $\theta$ and $M/\bQ_\ell$ are unramified. This implies easily our claim and (3) as well. Since $\rho$ and $\eta$ have coprime conductors, the statement (4) follows from \textit{loc. cit.}, Chapter IV, Exercise 3b). 
\end{proof}
\begin{proposition}\label{prop:renormalisation_mesure_p_adique}\hspace{2em}
		\begin{enumerate}
		\item The statement \textbf{EX}$_{\rho,\rho^+}$ in Conjecture \ref{conj:IMC} is equivalent to the existence of an element $\theta_{\rho,\rho^+}'\in\Frac(\Lambda)$ which has at most a pole at $\mathds{1}$ and which satisfies the following interpolation property: for all non-trivial characters $\eta\in\widehat{\Gamma}$ of conductor $p^n$, we have
		$$\eta(\theta_{\rho,\rho^+}')=\frac{\det(\rho^+)(\sigma_p^n)}{\g(\eta^{-1})^{d^+}}\ \Reg_{\omega_p^+}(\rho\otimes\eta)\cdot j\left(\dfrac{L^{*}\left((\rho \otimes \eta)^*,0\right)}{\Reg_{\omega_\infty^+}(\rho\otimes\eta)}\right).$$ 
		Moreover, if \textbf{EX}$_{\rho,\rho^+}$ holds, then $\theta_{\rho,\rho^+}$ and $\theta_{\rho,\rho^+}'$ are equal up to multiplication by a unit of $\Lambda$.
		\item If \textbf{EX}$_{\rho,\rho^+}$ holds and if $W_p^+$ is admissible, then \textbf{EZC}$_{\rho,\rho^+}$ is equivalent to $$\frac{1}{e!} \dfrac{\diff^e}{\diff \!s^e}\kappa^s(\theta_{\rho,\rho^+}')\bigg|_{s=0}=(-1)^e\ \cE(\rho,\rho^+)\ \cL(\rho,\rho^+)\ \Reg_{\omega_p^+}(\rho) \cdot j\left(\dfrac{L^{*}\left(\rho^*,0\right)}{\Reg_{\omega_\infty^+}(\rho)}\right). $$ 
	\end{enumerate}
\end{proposition}
\begin{proof}
By Lemma \ref{lem:properties_galois_gauss_sums} (4), the two quotients of $p$-adic measures $\theta_{\rho,\rho^+}$ and $\theta_{\rho,\rho^+}'$ are related (when they exist) by the formula $\theta_N \theta_{\rho,\rho^+}'=\tau(\rho)\theta_{\rho,\rho^+}$, where we have written $\theta_N=\prod_{\ell|N}\gamma_\ell^{\ord_\ell(N)}\in\Gamma\subseteq\Lambda^\times$ (and where $\gamma_\ell\in\Gamma$ is equal to the restriction to $\bQ_\infty$ of $\sigma_\ell$). By Lemma \ref{lem:properties_galois_gauss_sums} (3), one has $\tau(\rho)\in\cO_p^\times$, so $\theta_{\rho,\rho^+}$ and $\theta_{\rho,\rho^+}'$ are equal up to a unit of $\Lambda$, and the proposition follows easily.
\end{proof}



\subsection{Stark's regulator}
Let $(\pi,V)$ be a linear representation of $G_{\bQ}$ factoring through the Galois group of a finite extension $L/\bQ$ and with coefficients in a number field $F \subset \ob{\bQ}\subset \bC$. We denote by $S_{\infty}(L)$ the set of archimedean places of $L$.

Given any finite set of primes $S$ of $L$ containing $S_\infty(L)$, we let $Y_{L,S}$ be the free abelian group on the set $S$, and we define the subgroup
\begin{equation}\label{eq:definition_X_H,S}
	X_{L,S}:=\left\{\sum_{w} a_w\cdot w\in Y_{L,S}\ \biggm | \  \sum_{w}a_w=0\right\}.
\end{equation} 
Let $\cO_{L,S}$ be the ring of $S$-integers of ${L}$. By Dirichlet's unit theorem, the regulator map 
\begin{equation}\label{eq:dirichlet_reg}
	\lambda_{L,S} : \bR\otimes \cO_{L,S}^\times \overset{\sim}{\longrightarrow}\bR\otimes  X_{L,S}, \qquad a\mapsto - \sum_{w|v\in S}\log|a|_ww,
\end{equation}
is a $\Aut(L)$-equivariant isomorphism (see \cite[§1.1]{rubinstark}).

We now specialize to $S=S_\infty(L)$, $\cO_{L,S}^\times=\cO_{L}^\times$ and we put $X=X_{L,S_\infty(L)}$ and $\lambda=\lambda_{L,S_\infty(L)}$. 
By the Noether-Deuring theorem, there exists an isomorphism of $\bQ[G_{\bQ}]$-modules (which we fix)

\begin{equation}\label{eq:f} f\ \colon \bQ\otimes X \simeq \bQ\otimes \cO^\times_{L}. \end{equation}
The determinant $R(\pi,f)\in \bC^\times$ of the map
\[\begin{array}{rcl}
	\Hom_{G_{\bQ}}(\bC \otimes_{F,\iota_\infty} V, \bC\otimes X) & \longrightarrow &  \Hom_{G_{\bQ}}(\bC \otimes_{F,\iota_\infty} V, \bC\otimes X) \\
	\psi & \mapsto & \lambda \circ f \circ \psi
\end{array}\]
is Stark's regulator attached to $\pi^* $ and $f$ defined in \cite[Chap. I, \S4]{tate}. Note that both $L(\pi^*,s)$ and $R(\pi,f)$ implicitly depend on the embedding $\iota_\infty \colon \ob{\bQ}\hookrightarrow \bC$. Given any automorphism $\alpha$ of $\bC$, the map $\alpha \circ \iota_\infty$ also defines a complex embedding, and we denote by $L((\pi^\alpha)^*,s)$ and $R(\pi^\alpha,f)$ respectively the resulting $L$-series and Stark regulator.
\begin{conjecture}[Stark's conjecture]\label{conj:stark}
	For any automorphism $\alpha$ of $\bC$  we have $A(\pi^\alpha,f)=A(\pi,f)^\alpha$, where 
	\[
	A(\pi^\alpha,f)=\frac{L^*((\pi^\alpha)^*,0)}{R(\pi^\alpha,f)}. 
	\]
\end{conjecture}
In particular, $A(\pi,f)\in F^\times$ by basic field theory. One can check that Conjecture \ref{conj:stark} does not depend on the choice of the isomorphism $f$ in \eqref{eq:f}.

\begin{lemme}\label{lem:stark}
	Assume that $\pi$ does not contain the trivial representation. The determinant of the pairing 
	\begin{equation*}
		B\colon\left\{\begin{array}{rclcc} \bC \otimes \HH^0(\bR,V) & \times& \bC \otimes \HH^1_\ff(\pi^*(1)) &\longrightarrow &\bC \\
			(1 \otimes v &,&1\otimes \psi) &\mapsto & \log_\infty(\psi(v)),
		\end{array}\right.
	\end{equation*}
	computed with respect to $\omega_\infty^+\in \det_F \HH^0(\bR,V)$ and $\omega_{\ff} \in \det_F \HH^1_\ff(\pi^*(1))$ equals $R(\pi,f) \mod F^\times$.
\end{lemme}
\begin{proof}
	 Recall that $X$ sits inside $Y=\bigoplus_{w\in S_\infty(L)} \bZ w$ and note that $Y/X\simeq \bZ$, the $G_{\bQ}$-action on $\bZ$ being trivial. Our hypothesis on $\pi$ ensures that the natural embedding $\Hom_{G_\bQ}(V,F\otimes X) \hookrightarrow \Hom_{G_\bQ}(V,F\otimes Y)$ is an isomorphism. Letting $w_0$ be the archimedean place of $L$ corresponding to $\iota_\infty$, we have $Y\simeq \Ind_{\Gal(\bC/\bR)}^{\Gal(L/\bQ)} \ \bZ w_0$ as a Galois module. By Shapiro's lemma, we have
	\[
	\Hom_{G_\bQ}(V,F\otimes X) = \Hom_{G_\bQ}(V,F\otimes Y) \simeq \Hom_{\Gal(\bC/\bR)}(V,F)=\Hom_F(\HH^0(\bR, V),F),
	\]
	where the isomorphism in the middle is induced by the projection map $p \colon Y \rightarrow \bZ$ given by $p(\sum_w a_w w)=a_{w_0}$. Since $p\circ \lambda=-\log_\infty$ on $\cO^\times_{L}$, one sees that the determinant of $B$ coincides with that of the map $\psi \mapsto p \circ \lambda \circ \psi$ with respect to $\omega_{\ff}$ and $(\omega_\infty^+)^{-1}\in \det_F \Hom_F(\HH^0(\bR, V),F)$. Therefore, this determinant equals $\pm R(\pi,f)$ if one chooses $f,\omega_{\ff},\omega_\infty^+$ so that $f^*(\omega_{\ff})=p_*((\omega_\infty^+)^{-1})$.
\end{proof}

\begin{proposition}\label{prop:dependence_theta_on_bases}
	Assume Conjecture \ref{conj:stark}. The validity of Conjecture \ref{conj:IMC} for $(\rho,\rho^+)$ is independent of the choice of the bases $\omega_\infty^+$, $\omega_p^+$ and $\omega_{\ff,\eta}$ (for ${\eta\in\widehat{\Gamma}}$). More precisely, the measure $\theta_{\rho,\rho^+}\in \cO_p[[\Gamma]]$ is well-defined up to multiplication by an element of $\cO_p^\times$, and does not depend of the choice of the family $(\omega_{\ff,\eta})_{\eta}$.
\end{proposition}
\begin{proof}
	Note that the Galois-Gauss sums and the terms $\det(\rho^-)(\sigma_p^{-n})$, $\cE(\rho,\rho^+)$ and $\cL(\rho,\rho^+)$ appearing in \textbf{EX}$_{\rho,\rho^+}$ and \textbf{EZC}$_{\rho,\rho^+}$ do not depend on any basis. We are reduced to study the dependence on the choice of bases of the quantity 
	\begin{equation}\label{eq:quotient_reg_appendix}
		\frac{\Reg_{\omega_p^+}(\rho\otimes\eta)}{j\left(\Reg_{\omega_\infty^+}(\rho\otimes\eta)\right)} \in \ob{\bQ}_p
	\end{equation}
	for $\eta\in \widehat{\Gamma}$. The quotient \eqref{eq:quotient_reg_appendix} equals the determinant with respect to $\omega_\infty^+$ and $\omega_p^+$ of the composition of maps
	\[
	\bC \otimes \HH^0(\bR,W) \longrightarrow \bC \otimes \HH^1((\rho\otimes\eta)^*(1))^* \simeq \ob{
		\bQ}_p  \otimes \HH^1((\rho\otimes\eta)^*(1))^* \longrightarrow\ob{\bQ}_p  \otimes W_p^+
	\]
	induced by the pairings \eqref{eq:intro_complex_pairing}, \eqref{eq:intro_p_adic_pairing} and $j^{-1}\colon \bC\simeq\ob{\bQ}_p$. In particular, \eqref{eq:quotient_reg_appendix} does not depend on the choice of $\omega_{\ff,\eta}$, and it is multiplied by $ab^{-1}$ if one replaces ${\omega}_\infty^+$ and ${\omega}_p^+$ by $a\omega_\infty^+$ and $b\omega_p^+$ respectively (with $a,b\in E_p$). Therefore, the same conclusion holds for $\theta_{\rho,\rho^+}$, and the proposition follows from the fact that $a$ and $b$ belong to $\cO_p^\times$ by hypothesis of $T_p$-optimality.
\end{proof}


\subsection{Local and global duality}\label{sec:duality}
In order to better describe our Selmer group we first need to introduce and compare the ``unramified condition'' and the ``$\ff$-condition'' of Bloch and Kato for (the dual of) a local Galois representation with finite image. 

Take any $\cO_p$-representation $\textbf{T}$ of the absolute Galois group $G_F$ of a finite extension $F$ of $\bQ_\ell$ (including $\ell=p$). Assume that $\textbf{T}$ is of finite image, \textit{i.e.}, the action of $G_F$ factors through the Galois group $\Delta$ of a finite extension $L/F$. Define $\textbf{D}=\textbf{T}\otimes \Qp/\Zp$ and $\widecheck{\textbf{T}}=\textbf{T}^*(1)$ the arithmetic dual of $\textbf{T}$. Recall that, if $I\subseteq G_F$ is the inertia subgroup of $G_F$ and $M$ is a $G_F$-module, then $\HH^1_\ur(F,M)$ is the kernel of the restriction map $\HH^1(F,M) \longrightarrow\HH^1(I,M)$.
\begin{lemme}\label{lem:comparaison_f_et_ur}\hspace{2em}
	\begin{enumerate}
		\item If $\ell \neq p$, then we have $\HH^1_\ur(F,\widecheck{\textbf{T}}) \subseteq \HH^1_\ff(F,\widecheck{\textbf{T}})$ and $\HH^1_\ff(F,{\textbf{D}}) \subseteq \HH^1_\ur(F,{\textbf{D}})$.
		\item If $\ell=p$ and if $\textbf{T}$ is unramified, then $\HH^1_\ur(F,\textbf{D})=\HH^1_\ff(F,{\textbf{D}})$.
		\item Under the local Tate pairing $\HH^1(F,\widecheck{\textbf{T}}) \times \HH^1(F,\textbf{D}) \longrightarrow \cO_p \otimes \Qp/\Zp$, the $\HH^1_\ff$'s are orthogonal complements for any prime $\ell$ and the $\HH^1_\ur$'s are orthogonal complements for any $\ell\neq p$.
	\end{enumerate}
\end{lemme}
\begin{proof}
The first point follows from \cite[Lem. 3.5]{rubinES}. Let us prove the second statement and assume now that $\ell=p$ and that $\textbf{T}$ is unramified. Recall that $\HH^1_\ff(F,\textbf{D})$ is by definition the image of $\HH^1_\ff(F,\textbf{W})$ under the map $\HH^1(F,\textbf{W})\longrightarrow \HH^1(F,\textbf{D})$, where we have put  $\textbf{W}=\textbf{T}\otimes \Qp$. Since $\textbf{W}$ is unramified, it is easy to see that $\HH^1_\ff(F,\textbf{W})=\HH^1_\ur(F,\textbf{W})$. Moreover, the map $\HH^1_\ur(F,\textbf{W}) \longrightarrow \HH^1_\ur(F,\textbf{D})$ is surjective because it coincides with the projection map $\textbf{W}_{\Delta} \twoheadrightarrow \textbf{D}_\Delta$ by \cite[Lem. 3.2.(i)]{rubinES}, where $(-)_\Delta$ means that we took the $\Delta$-coinvariants. Therefore, we have $\HH^1_\ur(F,\textbf{D})=\HH^1_\ff(F,{\textbf{D}})$. 

The third statement is standard (see \cite[Prop. 3.8]{blochkato} for Bloch-Kato's condition and \cite[Prop. 4.3.(i)]{rubinES} for the unramified one).
\end{proof}

For $F=\bQ,\bQ_\ell$ (for any prime $\ell$) and for any compact $\cO_p[[G_F]]$-module $\textbf{T}$, define the Iwasawa cohomology along the $\Zp$-cyclotomic extension $F_\infty=\cup_n F_n$ of $F$ by letting
\[\begin{array}{ll}
\HH^1_{\Iw,*}(F,\textbf{T})&:= \varprojlim_n \HH^1_*(F_n,\textbf{T}) \qquad \qquad (*\in\{\emptyset,\ff,\ur\}), \\
\HH^1_{\Iw,\ff,p}(\bQ,\textbf{T})&:= \varprojlim_n \HH^1_{\ff,p}(\bQ_n,\textbf{T}),
\end{array}\]
where the subscript $\ff,p$ in the last global cohomology groups means that we relaxed the condition of being crystalline at $p$. We also use the standard notation $\HH^1_{\Iw,/\ff}(\bQ_\ell,\textbf{T})$ for the quotient $\HH^1_\Iw(\bQ_\ell,\textbf{T})/\HH^1_{\Iw,\ff}(\bQ_\ell,\textbf{T})$. All these cohomology groups are finitely generated modules over $\Lambda$ by Shapiro's lemma. 

We keep the notations of \S\ref{sec:selmer_group} and we fix a Galois-stable lattice $T_p$ of $W_p$ and any $p$-stabilization $W_p^+$ of $W_p$. By Maschke's theorem we may identify $W_p^-$ with a $\Gal(\ob{\bQ}_p/\Qp)$-stable complement of $W_p^+$ in $W_p$. Let $\widecheck{T}_p=T_p(1)^*$ (resp. $\widecheck{T}_p^\pm=T_p^\pm(1)^*$) be the arithmetic dual of $T_p$ (resp. of $T_p^\pm$). In particular, $\widecheck{T}_p^\pm$ is a free $\cO_p$-submodule of $\widecheck{T}_p$ of rank $d^\pm$. We will relate $\Sel^\str_\infty(\rho,\rho^+)$ and $\Sel_\infty(\rho,\rho^+)$ to the following three localization maps

\[\begin{array}{ll}
\Loc_+^\str :& \HH^1_{\Iw,\ff,p}(\bQ,\widecheck{T}_p) \longrightarrow \HH^1_{\Iw}(\bQ_p,\widecheck{T}^+_p) \\
\Loc_+ :& \HH^1_{\Iw,\ff}(\bQ,\widecheck{T}_p) \longrightarrow \HH^1_{\Iw,\ff}(\Qp,\widecheck{T}_p^+), \\
\Loc_+' :& \HH^1_{\Iw,\ff,p}(\bQ,\widecheck{T}_p) \longrightarrow \HH^1_\Iw(\Qp,\widecheck{T}_p^+) \bigoplus \HH^1_{\Iw,/\ff}(\Qp,\widecheck{T}_p^-).
\end{array}\]
For $\ell\neq p$, it is known that the quotient of the absolute Galois group of $\bQ_{\ell,\infty}$ by its inertia subgroup $I_\ell$ is of order prime to $p$, so the restriction map $\HH^1(\bQ_{\ell,\infty},D) \longrightarrow \HH^1(I_\ell,D)$ is injective. This, together with Lemma \ref{lem:comparaison_f_et_ur}, implies that 
\begin{equation}\label{eq:computation_local_Iw_coho}
	\varinjlim_n \HH^1_\ff(\bQ_{\ell,n},D_p) = \varinjlim_n \HH^1_\ur(\bQ_{\ell,n},D_p)=0, \quad \textrm{and} \quad \HH^1_{\Iw,\ff}(\bQ_\ell,\widecheck{T}_p) = \HH^1_{\Iw,\ur}(\bQ_\ell,\widecheck{T}_p) = \HH^1_{\Iw}(\bQ_\ell,\widecheck{T}_p).
\end{equation}
Hence, the Selmer group of $(\rho,\rho^+)$ fits into an exact sequence 
$$\begin{tikzcd}
0 \rar & \Sha_\infty^1(D_p) \rar & \Sel_\infty(\rho,\rho^+) \rar & \varinjlim_n \left(\HH^1(\bQ_{p,n},D_p^+) \bigoplus \HH^1_\ff(\bQ_{p,n},D_p^-)\right),
\end{tikzcd}$$
where $\Sha_\infty^1(D_p)=\ker \left[\HH^1(\bQ_\infty,D_p) \longrightarrow \prod_\ell \HH^1(\bQ_{\ell,\infty},D_p)\right]$ is the first Tate-Shafarevitch group. It follows from Poitou-Tate duality \cite[Cor. 7.5]{rubinES} that there is a commutative diagram of $\Lambda$-modules with short exact rows
\begin{equation}\label{eq:gros_diagramme_commutatif_Sel_infty}
\begin{tikzcd}
0 \rar & \Sha_\infty^1(D_p) \rar \dar[hookrightarrow] & \Sel_\infty(\rho,\rho^+) \rar \dar[equal] & \coker\left(\Loc_+'\right)^\vee \rar \dar[two heads] & 0 \\
0 \rar & \HH^1_\ur(\bQ_\infty,D_p) \rar & \Sel_\infty(\rho,\rho^+) \rar & \coker\left(\Loc_+\right)^\vee \rar & 0 \\
0 \rar & \Sha_\infty^1(D_p) \rar \uar[hookrightarrow] & \Sel^\str_\infty(\rho,\rho^+) \rar \uar[hookrightarrow] & \coker\left(\Loc_+^\str\right)^\vee \rar \uar[two heads] & 0,
\end{tikzcd}
\end{equation}
where $\HH^1_\ur(\bQ_\infty,D_p)= \ker\left[\HH^1(\bQ_\infty,D_p) \longrightarrow \prod_\ell \HH^1(I_\ell,D_p)\right]$ (and where $I_\ell$ is the inertia subgroup of $\bQ_{\ell,\infty}$).

We close this section by proving the Weak Leopoldt conjecture for both $W_p$ and $\widecheck{W}_p=W_p^*(1)$. The proof is self-contained and it won't use the running assumption that $\rho$ is unramified at $p$ (but $p$ is still be assumed to be odd). Let $\Sigma$ be a finite set of places of $\bQ$ containing $p$ and all the primes at which $\rho$ is ramified. For $i\in\bN$ and $\textbf{T}=T_p$ or $\widecheck{T}_p$ define the cohomology groups $\HH^i_{\Iw,\Sigma}(\bQ,\textbf{T})=\varprojlim_n \HH^i(\bQ_\Sigma/\bQ_n,\textbf{T})$, where $\bQ_\Sigma/\bQ$ is the maximal extension of $\bQ$ which is unramified outside $\Sigma$ and $\infty$. We also consider the second (compact) Tate-Shafarevich groups  $\Sha_\infty^2(\textbf{T})=\ker \left[\HH^2_\Iw(\bQ,\textbf{T}) \longrightarrow \prod_\ell \HH^2_\Iw(\bQ_{\ell},\textbf{T})\right]$.

\begin{proposition}\label{prop:conjecture_leopoldt_faible}\hspace{2em}
	\begin{enumerate}
		\item  We have $\Sha_\infty^2(\widecheck{T}_p)\simeq \Sha_\infty^1(D_p)^\vee$ as $\Lambda$-modules.
		\item The Weak Leopoldt conjecture along $\bQ_\infty/\bQ$ for $W_p$ and $\widecheck{W}_p$ holds, that is, the $\Lambda$-modules $\HH^2_{\Iw,\Sigma}(\bQ,T_p)$ and $\HH^2_{\Iw,\Sigma}(\bQ,\widecheck{T}_p)$ are torsion.
	\end{enumerate} 
\end{proposition}
\begin{proof}
The first statement follows from Poitou-Tate duality (see, for instance, \cite[Thm. 4.10.(a)]{milneADT}). Let us first prove the Weak Leopoldt conjecture for $W_p$. By \cite[Prop. 1.3.2]{perrin1995fonctions}, it is equivalent to $\HH^2(\bQ_\Sigma/\bQ_\infty,D_p)=0$. As $\Gal(\bQ_\Sigma/\bQ_\infty)$ has cohomological dimension 2, this module is $p$-divisible because $D_p$ is. It is therefore enough to show that $\HH^2(\bQ_\Sigma/\bQ_\infty,D_p)^\vee$ is a torsion $\Lambda$-module, which has already been shown in the proof of \cite[Prop. 2.5.5]{maks}. Finally, since the Weak Leopoldt Conjecture is clearly invariant by Tate twists, its validity for $\widecheck{W}_p$ follows from the one for the contragredient representation $W_p^*$. 
\end{proof}

\subsection{Limits of unit groups}\label{sec:limits_of_unit_groups}
The various cohomology groups introduced in \S\ref{sec:duality} can usefully be described in terms of ideal class groups and of unit groups.

\begin{notation}\label{nota:units_and_class_groups}
	Let $n\geq 0$ be an integer and let $w$ be a $p$-adic place of $H$. We still denote by $w$ the unique place of $H_n$ above $H$, and we let
	\begin{itemize}
		\item $A_{n}$ be the $p$-part of the ideal class group of ${H_n}$, and $A'_{n}$ its quotient by all the classes of $p$-adic primes of ${H_n}$,
		\item $\cO_{H_n}^\times$ (resp. $\cO_{H_n}[\tfrac{1}{p}]^\times$) be the unit group (resp. the group of $p$-units) of ${H_n}$,
		\item $U_{n}$ (resp. $U'_{n}$) be the pro-$p$ completion of the unit group (resp. of the group of $p$-units) of ${H_n}$,
		\item $U_{n,w}$ (resp. $U'_{n,w}$) be the pro-$p$ completion of the unit group (resp. of the group of non-zero elements) of ${H_{n,w}}$.
	\end{itemize}
		 We also let $A_\infty,A'_\infty,U_\infty,U'_\infty,U_{\infty,w},U'_{\infty,w}$ respectively be the projective limits of the preceding groups, where the transition maps are the (global or local) norm maps. All of these groups are $\Zp$-modules but we keep the same notations for the $\cO_p$-modules that are obtained after tensoring with $\cO_p$ by a slight abuse of notation. They are all endowed with a natural action of the Galois group $\Gal(H_\infty/\bQ)\simeq G \otimes \Gamma$, so they have a structure of $\Lambda[G]$-modules.
		 	\end{notation}


Fix as in \S\ref{sec:selmer_group} a Galois-stable $\cO_p$-lattice $T_p$ and a $p$-stabilization $W_p^+$ of $W_p$. We also keep the notations of \S\ref{sec:duality}. 
\begin{lemme}\label{lem:relation_cohomologie_et_unités}
	Let $G_p$ be the decomposition subgroup of $G$ at the place $w$ determined by $\iota_p$ and let $\bullet \in \{\emptyset,+,-\}$. 
	\begin{enumerate}
		\item The restriction maps on cohomology groups induce the following natural isomorphisms: 
	$$ \HH^1_{\Iw,\ff,p}(\bQ,\widecheck{T}_p) \simeq \Hom_G(T_p,U'_\infty), \qquad \HH^1_{\Iw}(\bQ_p,\widecheck{T}^{\bullet}_p) \simeq \Hom_{G_p}(T^{\bullet}_p,U'_{\infty,w}),$$ 
	$$ \HH^1_{\Iw,\ff}(\bQ,\widecheck{T}_p) \simeq \Hom_G(T_p,U_\infty), \qquad \HH^1_{\Iw,\ff}(\bQ_p,\widecheck{T}^{\bullet}_p) \simeq \Hom_{G_p}(T^{\bullet}_p,U_{\infty,w}).$$
	\item The $\Lambda$-modules $\Sha_\infty^1(D_p)^\vee$ and $\Hom_G(T_p,A'_\infty)$ are isomorphic after tensoring with $\Qp$ (as $\Lambda \otimes \Qp$-modules). They are isomorphic as $\Lambda$-modules if we assume that $p$ does not divide the order of $G$. 
	\item The $\Lambda$-modules $\Sha_\infty^1(D_p)$ and $\Hom_G(A_\infty',D_p)$ are pseudo-isomorphic.
\end{enumerate} 
\end{lemme}
\begin{proof}
We will derive the isomorphisms from Kummer theory and from the injectivity (resp. bijectivity) of restriction maps between certain local (resp. global) cohomology groups. Fix an integer $n\geq 0$, a prime number $\ell$ (including $\ell=p$) and let $\lambda|\ell$ be the prime of $H_n$ determined by $\iota_\ell$.  As in \S\ref{sec:duality}, let $\bQ_\Sigma$ be the largest extension of $\bQ$ (or, equivalently, the largest extension of $H_n$) which is unramified outside $\Sigma$ and $\infty$. Consider the finite field extension $F'/F$ and the Galois groups $\cG' \subseteq \cG$ given either by 
\begin{itemize}
	\item  $H_n/\bQ_n$ and $\cG'=\Gal(\bQ_\Sigma/H_n)$, $\cG=\Gal(\bQ_\Sigma/\bQ_n)\qquad$ (first case), 
	\item or by $H_{n,w}^\ur/\bQ_{n,w}^\ur$ and $\cG'=\Gal(\ob{\bQ}_p/H_{n,w}^\ur)$, $\cG=\Gal(\ob{\bQ}_p/\bQ_{n,w}^\ur)\qquad$ (second case), 
	\item or by $H_{n,\lambda}/\bQ_{n,\lambda}$ and $\cG'=\Gal(\ob{\bQ}_\ell/H_{n,\lambda})$, $\cG=\Gal(\ob{\bQ}_\ell/\bQ_{n,\lambda})\qquad$ (third case).
\end{itemize} 
Note that $\cG/\cG'$ can be identified with $G$ in the first case (because $p$ is unramified in $H/\bQ$) and with $G_p$ in the third case when $\ell=p$. Since $F'$ never contains $\bQ(\mu_{p^\infty})$, we have in all three cases $\HH^0(\cG',\widecheck{T}_p)= T_p^* \otimes \HH^0(\cG',\Zp(1))=0$. Moreover, Hochschild-Serre's spectral sequence applies to $\widecheck{T}_p$ in the first case and the third case with $\ell\neq p$ by \cite[Appendix B, Proposition 2.7]{rubinES}, so the restriction map 
$$\HH^i(\cG,\widecheck{T}_p) \longrightarrow \HH^0(\cG/\cG',\HH^i(\cG',\widecheck{T}_p)), \qquad (i=1,2)$$
is bijective when $i=1$, and surjective when $i=2$. It is also injective when $i=1$ in all three cases by the inflation-restriction exact sequence.

Let us first look at the first isomorphism of (1). By taking inverse limits in the first case with $i=1$, the restriction map gives an isomorphism $\HH^1_{\Iw,\Sigma}(\bQ,\widecheck{T}_p) \simeq \HH^1_{\Iw,\Sigma}(H,\widecheck{T}_p)^G$. Moreover, the local computation in  \eqref{eq:computation_local_Iw_coho} implies that  $\HH^1_{\Iw,\Sigma}(\bQ,\widecheck{T}_p)=\HH^1_{\Iw,\ff,p}(\bQ,\widecheck{T}_p)$ and $\HH^1_{\Iw,\Sigma}(H,\widecheck{T}_p)=\HH^1_{\Iw,\ff,p}(H,\widecheck{T}_p)$. We are reduced to study the module $\HH^1_{\Iw,\ff,p}(H,\widecheck{T}_p)^G$ which we compute as follows. Since $G_H$ acts trivially on $T_p$, it is equal to $\Hom_G(T_p, \HH^1_{\Iw,\ff,p}(H,\cO_p(1)))$. But Kummer theory naturally identifies $\HH^1_{\Iw,\ff,p}(H,\cO_p(1))$ with $U'_\infty$, so our claim follows. The three other restriction maps in (1) are shown to be isomorphisms by invoking similar arguments.

We now study $\Sha_\infty^1(D_p)^\vee$, which is known to be isomorphic to $\Sha_\infty^2(\widecheck{T}_p)$ by Proposition \ref{prop:conjecture_leopoldt_faible}. The Hochschild-Serre spectral sequence provides in our setting a commutative diagram with exact rows
{\small
$$\begin{tikzcd}
0 \rar & \HH^1(G,\HH^1_{\Iw,\Sigma}(H,\widecheck{T}_p)) \rar \dar & \HH^2_{\Iw,\Sigma}(\bQ,\widecheck{T}_p)\rar \dar & \HH^2_{\Iw,\Sigma}(H,\widecheck{T}_p)^G \rar \ar[d, "\alpha"] & 0 \\
0 \rar & \prod_{\ell\in \Sigma} \HH^1(G, \prod_{\lambda|\ell}\HH^1_{\Iw}(H_\lambda,\widecheck{T}_p)) \rar & \prod_{\ell\in \Sigma}\HH^2_{\Iw}(\bQ_\ell,\widecheck{T}_p) \rar & \left(\prod_{\lambda|\ell \in \Sigma}\HH^2_{\Iw}(H_\lambda,\widecheck{T}_p)\right)^G \rar & 0.
\end{tikzcd}$$
}
The cohomology groups $\HH^1(G,-)$ on the left are killed by the order $\#G$ of $G$, so they vanish after tensoring with $\Qp$ or whenever $p$ is coprime to $\#G$. Therefore, in order to prove the claim (2) it suffices to check that the module $\Sha_\infty^2(H,\widecheck{T}_p)^G=\ker\alpha$ can be identified with $\Hom_G(T_p,A_\infty')$. But again, $G_H$ acts trivially on $T_p$ so it is enough to see that $\Sha_\infty^2(H,\cO_p(1))\simeq A_\infty'$, which is classical (see for instance \cite[(9.2.2.2)]{sel}).

The proof of last statement also uses inflation-restriction and is nearly identical to the proof of \cite[Lemme 2.2.2]{maks}, once we identify $A_\infty'$ with the Galois group of the maximal abelian pro-$p$ extension of $H_\infty$ which is unramified everywhere and in which all primes above $p$ split completely.
\end{proof}

\begin{lemme}\label{lem:liberté_module_G_équiv}
	Assume that $U$ is a free $\Lambda$-module of finite rank endowed with a $\Lambda$-linear action of $G$. Then the module $Z=\Hom_G(T_p,U)$ is also $\Lambda$-free.
\end{lemme}
\begin{proof}
	First recall that any finitely generated $\Lambda$-module $V$ is free if and only if $V^\Gamma$ vanishes and $V_\Gamma$ is $\cO_p$-free (see for example \cite[Lemme 1.1]{belliard}). We already have $Z^\Gamma=\Hom_G(T_p,U^\Gamma)=0$. It is therefore enough to check that $Z_\Gamma$ has no $\cO_p$-torsion, and this will follow from the fact that $Z_\Gamma$ injects into the torsion-free $\cO_p$-module $\Hom_G(T_p,U_\Gamma)$. Let us first check that $(\gamma-1)\Hom_G(T_p,U)=\Hom_G(T_p,(\gamma-1)U)$, where $\gamma$ is a topological generator of $\Gamma$. The inclusion $\subseteq$ is obvious, so we only consider the reverse inclusion $\supseteq$ and we take $\alpha \in \Hom_G(T_p,(\gamma-1)U)$. As $T_p$ is $\cO_p$-free, one may write $\alpha=(\gamma-1)\beta$ for some $\cO_p$-linear map $\beta : T_p \longrightarrow U$. As the $G$-action is assumed to be $\Lambda$-linear, for all $g\in G$ and $t\in T_p$ we have $\beta(g.t)-g.\beta(t) \in U^\Gamma=0$, so $\beta \in \Hom_G(T_p,U)$ and $\alpha \in (\gamma-1)\Hom_G(T_p,U)$ as claimed. Therefore, we have 
	\begin{multline*}
	Z_\Gamma = \Hom_G(T_p,U)/(\gamma-1)\Hom_G(T_p,U)= \Hom_G(T_p,U)/\Hom_G(T_p,(\gamma-1)U) \\ \hookrightarrow \Hom_G(T_p,U/(\gamma-1)U)=\Hom_G(T_p,U_\Gamma),
	\end{multline*}
	as wanted.
	\end{proof}

\begin{lemme}\label{lem:freeness_Iwasawa_modules}\hspace{2em}
	\begin{enumerate}
		\item The $\Lambda$-modules $\HH^1_{\Iw,\ff,p}(\bQ,\widecheck{T}_p)$ and $\HH^1_{\Iw,\ff}(\bQ,\widecheck{T}_p)$ are both free of rank $d^+$.
		\item If $\coker \left(\Loc_+\right)$ is of $\Lambda$-torsion, then the modules $\coker\left(\Loc_+'\right)$, $\coker\left(\Loc_+'\right)$ and the Selmer groups $X_\infty(\rho,\rho^+)$ and $X^\str_\infty(\rho,\rho^+)$ are all of $\Lambda$-torsion, and moreover $\ker \left(\Loc_+\right)=\ker \left(\Loc_+'\right)=\ker \left(\Loc_+^\str\right)=0$.
	\end{enumerate}
\end{lemme}

\begin{proof}
 Since $H$ is unramified at $p$, we have $\mu_p \not\subseteq H$ so the $\Lambda$-modules $U_\infty$ and $U'_\infty$ are free by \cite[Théorème 1.5 and Corollaire 1.6]{belliard}. Thus, both Iwasawa cohomology groups are $\Lambda$-free by Lemmas \ref{lem:relation_cohomologie_et_unités} and \ref{lem:liberté_module_G_équiv}. Moreover, the validity of the Weak Leopoldt conjecture (Proposition \ref{prop:conjecture_leopoldt_faible} (2)) implies that they are both of rank $d^+$ over $\Lambda$. Let us treat (2) and assume that $\coker \left(\Loc_+\right)$ is of $\Lambda$-torsion. Note that the $\Lambda$-module $\HH^1_\ur(\bQ_\infty,D_p)$ is co-torsion because $A_\infty$ is of $\Lambda$-torsion by Iwasawa's classical result \cite{iwasawa1973zl}. Thus, the claims of torsionness of (2) all follow from the commutative diagram (\ref{eq:gros_diagramme_commutatif_Sel_infty}). Moreover, the source and the target of $\Loc_+$ have the same rank and the source is torsion-free, so $\ker \left(\Loc_+\right)$ must vanish. Since the two other kernels are submodules of $\ker \left(\Loc_+\right)$, they must be trivial as well. 
\end{proof}

\subsection{Torsionness of Selmer groups}\label{sec:torsionness_Sel}
For any character $\eta \in \widehat{\Gamma}$ factoring through ${\Gamma}_{m}$ but not through $\Gamma_{m-1}$ for some $m\geq 0$, let $M[\eta]=M \otimes_{\Lambda,\eta} \ob{\bQ}_p$ be the $\eta$-isotypic component of a $\Lambda$-module $M$, and let
$$|\cdot|_\eta : \left\{\begin{array}{ccc}
U'_\infty &\longrightarrow& U'_m[\eta] \\
(u_n)_{n\geq 0} &\mapsto & e_\eta . u_m,
\end{array}\right.$$
where $e_\eta \in \ob{\bQ}_p[\Gamma_m]$ is the idempotent attached to $\eta$ (see \S\ref{sec:coleman_intro}), and where $U_m$ is seen as a $\Lambda$-module via the projection map $\Lambda \twoheadrightarrow \cO_p[\Gamma_m]$.
\begin{lemme}\label{lem:eta-spécialisation_unités}
	Let $\eta$ be a non-trivial character of $\Gamma$ of conductor $p^n$. Consider the following commutative diagram
	$$\begin{tikzcd}
	\Hom_G(T_p,U_\infty)[\eta] \dar\rar & \Hom_G(T_p,U_{n-1})[\eta] \dar \\ 
	\Hom_G(T_p,U'_\infty)[\eta] \rar & \Hom_G(T_p,U'_{n-1})[\eta],
	\end{tikzcd}$$
	where the horizontal maps are induced by $|\cdot|_\eta$. Then all the four maps are isomorphisms.
\end{lemme}
\begin{proof} Since $\Gamma$ acts trivially on the quotient $U'_\infty/U_\infty$ (resp. $U'_{n-1}/U_{n-1}$) and since $\eta$ is non-trivial by assumption, we have $U_\infty[\eta]=U'_\infty[\eta]$ (resp. $U_{n-1}[\eta]=U'_{n-1}[\eta]$), so the two vertical maps are isomorphisms. Hence, we only have to show that the bottom horizontal map is an isomorphism. By Lemmas \ref{lem:relation_cohomologie_et_unités} (1) and \ref{lem:freeness_Iwasawa_modules} (1), its domain has dimension $d^+$ over $\ob{\bQ}_p$, as well as for its codomain because $\eta \neq \mathds{1}$. Therefore, it is enough to check its injectivity, which easily follows from the fact that the projection map injects $(U'_\infty)_{\Gamma^{p^{n-1}}}$ into $U'_{n-1}$ by \cite[Thm. 7.3]{kuzmin}. 
\end{proof}
\begin{remark}\label{rem:target_eta_proj}
	The $\ob{\bQ}_p$-vector space $\Hom_G(T_p,U_{n-1})[\eta]$ in Lemma \ref{lem:eta-spécialisation_unités} can be identified with $\Hom_{G_{\bQ}}(W_{p,\eta},U_{n-1} \otimes \ob{\bQ}_p)$, and hence with $\ob{\bQ}_p \otimes_E \HH^1_\ff((\rho \otimes \eta)^*(1))$.
\end{remark}

Fix a eigenbasis $\omega_p^+=t_1\wedge \ldots \wedge t_{d^+}$ of $\det_{\cO_p}T_p^+$ for the action of $G_p$ and let $\delta_1,\ldots,\delta_{d^+}:G_p \longrightarrow \cO^\times_p$ be the corresponding characters. The number of $\delta_i$'s which are trivial is $f-e$, where $f=\dim\HH^0(\Qp,W_p)$ and $e=\dim\HH^0(\Qp,W_p^-)$. We define two composite maps 
\begin{equation}\label{eq:wedge_col}\begin{tikzcd}
\sC_{\omega_p^+} : {\bigwedge^{d^+}} \Hom_G(T_p,U_\infty) \ar[r, "\wedge \Loc_+"] & {\bigwedge^{d^+}} \Hom_{G_p} (T^+_p,U_{\infty,w}) \ar[r, "\wedge\textrm{ev}_{\omega_p^+}", "\simeq"'] & {\bigwedge^{d^+}} \left(\bigoplus_{i=1}^{d^+} U^{\delta_i}_{\infty,w}\right) \ar[r, "\wedge_i \Coleman^{\delta_i}", "\simeq"'] & \Lambda, \\
\sC^\str_{\omega_p^+} : {\bigwedge^{d^+}} \Hom_G(T_p,U'_\infty) \ar[r, "\wedge\Loc^\str_+"] & {\bigwedge^{d^+}} \Hom_{G_p} (T^+_p,U'_{\infty,w}) \ar[r, "\wedge\textrm{ev}_{\omega_p^+}", "\simeq"'] & {\bigwedge^{d^+}} \left(\bigoplus_{i=1}^{d^+} (U'_{\infty,w})^{\delta_i}\right) \ar[r, "\wedge_i \widetilde{\Coleman}^{\delta_i}", "\simeq"'] & \cI^{f-e},
\end{tikzcd}
\end{equation}
where the map $\textrm{ev}_{\omega_p^+}$ is the natural map induced by the evaluation at $t_1,\ldots,t_{d^+}$,  the last maps are the one of Definitions \ref{def:coleman_isotypic} and \ref{def:coleman_isotypic_extended} and  $\cI$ is the invertible ideal of $\Lambda$ introduced in Definition \ref{def:coleman_isotypic_extended} for $\cO=\cO_p$. Once we fix a topological generator of $\Gamma$ and thus an isomorphism $\Lambda\simeq\cO_p[[T]]$, $\cI^{f-e}$ is nothing but $T^{-(f-e)}\Lambda \subseteq \Frac(\Lambda)$.
\begin{lemme}\label{lem:interpolation_C+}
	Fix an eigenbasis $\omega_p^+$ of $\det_{\cO_p}T_p^+$ for $G_p$ and let $\omega=\Psi_1\wedge \ldots\wedge\Psi_{d^+}$ be an element of ${\bigwedge^{d^+}} \Hom_G(T_p,U_\infty)$ (resp. of ${\bigwedge^{d^+}}\Hom_G(T_p,U'_\infty)$). Then for all non-trivial characters $\eta\in\widehat{\Gamma}$ of conductor $p^n$, the image $\theta$ of $\omega$ under $\sC_{\omega_p^+}$ (resp. under $\sC^\str_{\omega_p^+}$) satisfies
	$$\eta(\theta)= \left(\frac{p^{n-1}}{\g(\eta^{-1})}\right)^{d^+}\det(\rho^+)(\sigma_p^n)\det \left(\log_p|\Psi_j(t_i)|_\eta\right)_{1\leq i,j \leq d^+}.$$
	\end{lemme}
\begin{proof}
	Let $\beta_i=\delta_i(\sigma_p)$ for all $1\leq i \leq d^+$. Note that $\det(\rho^+)(\sigma_p)=\prod_{i=1}^{d^+}\beta_i$. In the case where $\theta=\sC_{\omega_p^+}(\omega)$, Lemma \ref{lem:val_spé_coleman} shows that
	\begin{align*}
	\eta(\theta) &= \det\left(\frac{p^{n-1}}{\g(\eta^{-1})} \beta_i^n\log_p|{\Psi}_j(t_i)|_\eta\right)_{1\leq i,j \leq d^+} \\
	&=\left(\frac{p^{n-1}}{\g(\eta^{-1})}\right)^{d^+}\det(\rho^+)(\sigma_p^n)\det \left(\log_p|{\Psi}_j(t_i)|_\eta\right)_{1\leq i,j \leq d^+}
	\end{align*}
	for any $\eta\in\widehat{\Gamma}$ of conductor $p^n>1$. Since $\sC^\str_{\omega_p^+}$ extends $\sC_{\omega_p^+}$, the case where $\theta=\sC^\str_{\omega_p^+}(\omega)$ follows from the first one, noting that for non-trivial $\eta\in\widehat{\Gamma}$, the maps $\eta \circ \sC_{\omega_p^+}$ and $\eta \circ \sC^\str_{\omega_p^+}$ both factor through $\Hom_G(T_p,U_\infty)[\eta]=\Hom_G(T_p,U_\infty')[\eta]$ and that they clearly coincide on it.
\end{proof}
\begin{theorem}\label{thm:equivalence_torsionness_Sel}
	Fix a basis $\omega_p^+$ of $\det_{\cO_p}T_p^+$. The following conditions are equivalent:
	\begin{enumerate}
		\item[(i)] $X_\infty(\rho,\rho^+)$ is a torsion $\Lambda$-module,
		\item[(ii)] $\coker\left(\Loc_+\right)$ is a torsion $\Lambda$-module,
		\item[(iii)] there exists a non-trivial character $\eta$ of ${\Gamma}$ of finite order such that $\Reg_{\omega_p^+}(\rho \otimes \eta)\neq 0$,
		\item[(iv)] for all but finitely many characters $\eta$ of $\Gamma$ of finite order, one has $\Reg_{\omega_p^+}(\rho \otimes \eta)\neq 0$.
	\end{enumerate}
Moreover, if these equivalent conditions hold and if $d^+>0$, then there exists linearly independent elements $\Psi_1,\ldots,\Psi_{d^+}\in\Hom_G(T_p,U_\infty)$ which only depend on $T_p$ and there exists a generator $\theta^\alg_{\rho,\rho^+}$ of the characteristic ideal of $X_\infty(\rho,\rho^+)$ such that 
\begin{equation}\label{eq:interpolation_algebraic_p_adic_L_function}
\eta(\theta^\alg_{\rho,\rho^+}) = j\left(\frac{\tau(\eta)^{d^-}}{\tau(\rho\otimes\eta)}\right)\ \frac{p^{(n-1)\cdot d^+}}{\det(\rho^-)(\sigma_p^n)} \cdot\det \left(\log_p|\Psi_j(t_i)|_\eta\right)_{1\leq i,j \leq d^+}, 
\end{equation}
for all non-trivial characters $\eta\in\widehat{\Gamma}$ of conductor $p^n$, where we have written $\omega_p^+=t_1\wedge\ldots\wedge t_{d^+}$.
\end{theorem}
\begin{proof}
	We may assume without loss of generality that the basis $\omega_p^+=t_1\wedge\ldots\wedge t_{d^+}$ of $\det_{\cO_p}T_p^+$ is an eigenbasis for $G_p$. The equivalence of (i) and (ii) follows from Lemma \ref{lem:freeness_Iwasawa_modules}. We now show the equivalence of the three last statements and we will use Lemma \ref{lem:relation_cohomologie_et_unités} to identify the source and the target of $\Loc_+$ with the respective $\Hom$'s, which are known to both be free of rank $d^+$ over $\Lambda$ by Lemma \ref{lem:freeness_Iwasawa_modules}. The statement (ii) is equivalent to the injectivity of $\Loc_+$, which in turn is equivalent to the injectivity of $\sC_{\omega_p^+}$, \textit{i.e.}, to the non-vanishing of $\theta_1^\alg:=\sC_{\omega_p^+}(\widetilde{\omega})$, where $\widetilde{\omega}=\tilde{\Psi}_1\wedge\ldots\wedge\tilde{\Psi}_{d^+}$ is a $\Lambda$-basis of $\Hom_G(T_p,U_\infty)$. Moreover, if (ii) holds, then the characteristic ideal of $\coker(\Loc_+)$ is generated by $\theta_1^\alg$. By Lemma \ref{lem:interpolation_C+}, for any character $\eta\in\widehat{\Gamma}$ of conductor $p^n>1$, one has 
\begin{equation*}
\eta(\theta_1^\alg) =\left(\frac{p^{n-1}}{\g(\eta^{-1})}\right)^{d^+}\det (\rho^+)(\sigma_p^n)\det \left(\log_p|\tilde{\Psi}_j(t_i)|_\eta\right)_{1\leq i,j \leq d^+}.
\end{equation*}
Since the image of the basis $\tilde{\Psi}_1,\ldots,\tilde{\Psi}_{d^+}$ under $|\cdot|_\eta$ is a $\ob{\bQ}_p$-basis of $\ob{\bQ}_p \otimes_E \HH^1_\ff((\rho \otimes \eta)^*(1))$ by Lemma \ref{lem:eta-spécialisation_unités} and Remark \ref{rem:target_eta_proj}, the last $d^+ \times d^+$-sized determinant is a non-zero multiple of $\Reg_{\omega_p^+}(\rho \otimes \eta)$. Therefore, by Weierstrass preparation theorem, one has $\theta_1^\alg \neq 0$ if and only if $\Reg_{\omega_p^+}(\rho \otimes \eta)\neq0$ for some character $\eta\neq \mathds{1}$, if and only if $\Reg_{\omega_p^+}(\rho \otimes \eta)\neq 0$ for all but finitely many characters $\eta\neq \mathds{1}$. This shows the equivalence (ii)$\Leftrightarrow $(iii)$\Leftrightarrow$(iv).

Let us now assume (i)-(iv) and let $\theta_2^\alg$ be a generator of $\char_\Lambda \HH^1_\ur(\bQ_\infty,D_p)^\vee$. By the exactness of the second row of the diagram (\ref{eq:gros_diagramme_commutatif_Sel_infty}) and by multiplicativity of characteristic ideals, the $p$-adic measure $\theta^\alg_{\rho,\rho^+}:=\theta^\alg_{1}\theta^\alg_{2}$ is a generator of $\char_\Lambda X_\infty(\rho,\rho^+)$. If one moreover assumes $d^+>0$, then one can simply set 
$$\Psi_1=\tau(\rho)^{-1}\cdot\theta_N\cdot\theta^\alg_{2}\cdot \tilde{\Psi}_1,\quad \Psi_2=\tilde{\Psi}_2,\quad \ldots,\quad \Psi_{d^+}=\tilde{\Psi}_{d^+},$$ 
where $\theta_N$ is as in the proof of Proposition \ref{prop:renormalisation_mesure_p_adique}, and Formula (\ref{eq:interpolation_algebraic_p_adic_L_function})
follows by linearity from the above expression of $\eta(\theta_1^\alg)$.
 \end{proof}
\begin{corollaire}\label{coro:existence_p_stab_Sel_tors}
	There exists a $p$-stabilization $\rho^+$ of $\rho$ such that $X_\infty(\rho,\rho^+)$ is a torsion $\Lambda$-module.
\end{corollaire}
\begin{proof}
		 It suffices to produce a $p$-stabilization $\rho^+$ of $\rho$ such that the condition (ii) of Theorem \ref{thm:equivalence_torsionness_Sel} holds. Since $\rho$ is unramified at $p$, $W_p$ splits as $G_{\bQ_p}$-module into a direct sum $E_p(\chi_1)\oplus \ldots \oplus E_p(\chi_d)$ for certain one-dimensional characters $\chi_1,\ldots,\chi_d$. Note that any subset $\{\chi_{i_1},\ldots,\chi_{i_{d^+}}\}$ of size $d^+$ of $\{\chi_1,\ldots,\chi_d\}$ yields a $p$-stabilization of $\rho$ by letting $W_p^+=E_p(\chi_{i_1})\oplus \ldots\oplus E_p(\chi_{i_{d^+}})$, and the corresponding map $\Loc_+$ is the composition:
		 \[\HH^1_{\Iw,\ff}(\bQ,\widecheck{T}_p) \longrightarrow \HH^1_{\Iw,\ff}(\Qp,\widecheck{T}_p)=\bigoplus_{1\leq i\leq d} \HH^1_{\Iw,\ff}(\Qp,\widecheck{T}_{p,i}) \twoheadrightarrow \bigoplus_{1\leq k\leq d^+} \HH^1_{\Iw,\ff}(\Qp,\widecheck{T}_{p,i_k}),\]
		 where $T_{p,i}= E_p(\chi_d) \cap T_p$.
		 
		 As the weak Leopoldt conjecture for $\rho$ holds by Proposition \ref{prop:conjecture_leopoldt_faible}, the localization map  $\HH^1_{\Iw,\ff}(\bQ,\widecheck{T}_p) \longrightarrow \HH^1_{\Iw,\ff}(\Qp,\widecheck{T}_p)$ is injective (see \cite[\S1.3]{perrin1995fonctions}). Moreover, $\HH^1_{\Iw,\ff}(\bQ,\widecheck{T}_p)$ and $\HH^1_{\Iw,\ff}(\Qp,\widecheck{T}_{p,i})$ are free $\Lambda$-modules of respective ranks $d^+$ and $1$ (see Lemma \ref{lem:freeness_Iwasawa_modules} and diagram \eqref{eq:wedge_col}). The corollary then follows from the following elementary fact: let $A \hookrightarrow B$ be a map of free modules of respective ranks $d^+$ and $d$ over an integral domain $R$ and assume that $B=B_1\oplus \ldots \oplus B_d$ where the $B_i$'s are all free of rank one. Then there exist $1\leq i_1 <\ldots <i_{d^+}\leq d$ such that the cokernel of the composite map $A \hookrightarrow B \twoheadrightarrow B_{i_1} \oplus \ldots \oplus B_{i_{d^+}}$ is of $R$-torsion.
\end{proof}

\subsection{An $\cL$-invariant for $(\rho,\rho^+)$}\label{sec:L-invariant}
We define in this section an $\cL$-invariant for $\rho$ which depends on the choice of an admissible $p$-stabilization $W_p^+$ of $W_p$. Its definition still makes sense for $\rho$ ramified at $p$ and it generalizes Gross's $\cL$-invariant (see \S\ref{sec:exemple_odd_motives}). We put $\cU = {E_p} \otimes U_0$ and $\cU'= {E_p} \otimes U'_0$, where $U_0$ (resp. $U'_0$) is the formal $p$-adic completion of the group of global units (resp. of $p$-units) of $H$ (see Notation \ref{nota:units_and_class_groups} for $n=0$). Recall that the extension $K$ of $\Qp$ cut out by $\rho_{\big| G_{\Qp}}$ is assumed to be contained in $E_p$. Recall also that $\HH^1_{f,p}(\bQ,\widecheck{W}_p) \simeq \Hom_G(W_{p},\cU')$. The $p$-adic valuation map and the $p$-adic logarithm map define two $G_{\Qp}$-equivariant maps $\ord_p : \cU' \longrightarrow {E_p}$ and $\log_p : \cU' \longrightarrow K \otimes {E_p}$ where the $G_{\Qp}$-action on the target of the first map is trivial. These maps give rise to two linear maps 
$$\begin{aligned}
\ord_p^0 : & \Hom_G(W_{p},\cU') \longrightarrow \Hom_{G_p}(W_{p},{E_p})=\Hom(W^0_{p},{E_p}) \\
\log_p :  & \Hom_G(W_{p},\cU') \longrightarrow \Hom_{G_p}(W_{p},K \otimes{E_p})\simeq\Hom(W_{p},{E_p}),
\end{aligned}$$
where we have put $W^0_{p} = \HH^0(\Qp,W_{p})$ and where the last isomorphism is induced by the internal multiplication $K \otimes {E_p} \rightarrow {E_p}$. Choose any $G_{\Qp}$-stable complement $W^-_{p}$ of $W^+_{p}$ and let $W_{p}^{\pm,0}=\HH^0(\Qp,W^\pm_{p})$. The composition with restriction maps yields the following maps $$\ord_p^{\pm,0}=\res_{W_{p}^{\pm,0}} \circ \ord_p^0, \quad \log_p^+ = \res_{W_{p}^{+}} \circ \log_p, \quad \log^{-,0} = \res_{W_{p}^{-,0}} \circ \log_p, $$
which can be combined in order to obtain two linear maps
$$\begin{tikzcd}[column sep = 17ex]\Hom_G(W_{p}, \cU') \ar[r, "\ord_p^0 \oplus \log_p^+", "\ord_p^{+,0} \oplus \log_p^{-,0}\oplus \log_p^+"'] & \Hom(W_{p}^0,{E_p}) \bigoplus \Hom(W_{p}^+,{E_p})=:Z. \end{tikzcd}$$
Note that $W_p^+$ is admissible if and only if the restriction of $\log_p^+$ to $\HH^1_\ff(\bQ,\widecheck{W}_p) =\Hom_G(W_{p},\cU)$ is an isomorphism onto $\Hom(W^+_{p}, {E_p})$. Therefore, the map $\ord_p^0 \oplus \log_p^+$ is an isomorphism.
\begin{definition}
	Let $W^+_p \subseteq W_p$ be an admissible $p$-stabilization of $\rho$. We define the $\cL$-invariant attached to $(\rho,\rho^+)$ to be
	$$\cL(\rho,\rho^+)=\det\left(\left(\ord_p^{+,0} \oplus \log_p^{-,0}\oplus \log_p^+\right) \circ \left(\ord_p^0 \oplus \log_p^+ \right)^{-1} \; \bigg| \; Z\;\right)\in {E_p} .$$ 
\end{definition}
We now give equivalent and more useful definitions of $\cL(\rho,\rho^+)$. Fix bases $\omega_p^+=t_1^+\wedge\ldots\wedge t_{d^+}^+$ and $\omega_p^{-,0}=t^-_1\wedge \ldots \wedge t^-_e$ of $\det_{E_p} W_{p}^+$ and of $\det_{E_p} W^{-,0}_{p}$ respectively. Then one may identify $Z$ with $d^++f$ copies of $E_p$ and $\cL(\rho,\rho^+)$ satisfies 
\begin{equation*}
(\wedge \ord_p^{+,0}) \wedge (\wedge \log_p^{-,0}) \wedge(\wedge \log_p^{+}) = \cL(\rho,\rho^+)\cdot (\wedge \ord_p^{0}) \wedge(\wedge \log_p^{+}) 
\end{equation*}
in ${\bigwedge^{d^++f}}\Hom(\Hom(W_p,\cU'),E_p)= \det_{E_p}\Hom(\HH^1_{\ff,p}(\bQ,\widecheck{W}_p),E_p)$.

The kernel of $\ord_p^{+,0}$ contains $\Hom_G(W_{p},\cU)$ and is of dimension $d^++e$ because $W_p^+$ is admissible. Choose any basis $\psi_1 \wedge \ldots\wedge\psi_{d^+}\wedge\psi'_1\wedge\ldots\wedge\psi'_e$ of $\det_{E_p} \ker \left(\ord_p^{+,0}\right)$ such that $\psi_1 \wedge \ldots\wedge\psi_{d^+}$ is a basis of $\det_{E_p} \Hom_G(W_{p},\cU)$, and define the following matrices with coefficients in ${E_p}$:
\begin{equation}\label{eq:def_mat_A_B_O}
	A^\pm =\left[\log_p\left(\psi_j(t^\pm_i)\right)\right]_{i,j}, \quad B^\pm =\left[\log_p\left(\psi'_j(t^\pm_i)\right)\right]_{i,j}, \quad  O^- =\left[\ord_p\left(\psi'_j(t^-_i)\right)\right]_{i,j}.
\end{equation}
The square matrices $A^+$, $B^-$ and $O^-$ have respective sizes $d^+$, $e$ and $e$. The determinant of $A^+$ is non-zero and it is equal to $\Reg_{\omega_p^+}(\rho)$ modulo $E^\times$ if the basis $\psi_1 \wedge \ldots\wedge\psi_{d^+}$ is taken $E$-rational (see (\ref{eq:expression_reg_p})). Also, $O^-$ is invertible.
\begin{lemme}\label{lem:formule_L_invariant_quot_det}
 $\cL(\rho,\rho^+)$ can be expressed as a quotient of determinants as follows: 
	$$\cL(\rho,\rho^+)= \dfrac{\det \begin{pmatrix}A^+ & B^+ \\ A^- & B^- \end{pmatrix}}{\det(A^+)\cdot\det(O^-)}.$$
\end{lemme}
\begin{proof}It is a simple computation of linear algebra.
\end{proof}

\subsection{Relation to Perrin-Riou's and Benois' theory}\label{sec:perrin-riou}
Let $M$ be the Artin motive attached to $\rho$ and let $\widecheck{M}=M^*(1)$ be its dual. Then $\widecheck{M}$ is a pure motive of weight $-2$ over $\bQ$ which is crystalline at $p$ and whose $p$-adic realization is the arithmetic dual $\widecheck{W_p}=W_p^*(1)$ of $W_p$. We relate our Selmer group defined in \S\ref{sec:selmer_group} to Perrin-Riou's definition of the module of $p$-adic $L$-functions attached to $\widecheck{M}$ given in \cite{perrin1995fonctions} by using Benois' interpretation in terms of Selmer complexes. It depends on the choice of a Galois stable $\cO_p$-lattice $T_p$ of $W_p$ and of a regular subspace $D$ of $\textbf{D}_\crys(\widecheck{W}_p)$ whose definition is first recalled.

 Let $G_p=\Gal(K/\Qp)$ and assume as before that $K \subseteq E_p$. Let $t\in \textbf{B}_\crys$ be Fontaine's $p$-adic period. The Dieudonné module $\textbf{D}_\crys(\widecheck{W}_p)=\left(\widecheck{W}_p\otimes\textbf{B}_\crys\right)^{G_{\bQ_p}}$ can be described as 
\begin{equation}\label{eq:identification_dieudonné_module}
\textbf{D}_\crys(\widecheck{W}_p) \simeq \Hom_{G_p}\left(W_p, t^{-1} K \otimes E_p  \right) \simeq \Hom\left(W_p, E_p\right),
\end{equation}
where $G_p$ acts on $t^{-1} K \otimes E_p$ via $g\cdot (t^{-1}x\otimes y)=t^{-1}g(x)\otimes y$ and the second isomorphism is induced by the internal multiplication $t^{-1} K \otimes E_p \simeq K \otimes E_p \longrightarrow E_p$. The crystalline Frobenius $\varphi$ acts on $f\in \Hom\left(W_p, E_p\right)$ by $\varphi(f)(w)=p^{-1}f(\sigma_p^{-1}.w)$, where $\sigma_p\in G_p$ is the arithmetic Frobenius at $p$.

Any $\varphi$-submodule $D$ of $\textbf{D}_\crys(\widecheck{W}_p)$ gives rise to a regulator map given by the composition
\begin{equation}\label{eq:def_reg_r_D}
	r_D :\begin{tikzcd}
\HH^1_\ff(\bQ,\widecheck{W}_p) \ar[r, "\loc_p"] &  \HH^1_\ff(\Qp,\widecheck{W}_p) \ar[r, "\log_\BK"] & 
\textbf{D}_\crys(\widecheck{W}_p) \ar[r, two heads] &
\textbf{D}_\crys(\widecheck{W}_p)/D,
\end{tikzcd}
\end{equation}
where $\loc_p$ is the localization at $p$, and $\log_\BK$ is Bloch-Kato's logarithm. The definition of $r_D$ for a more general $p$-adic representation $V$ involves the tangent space $\textbf{D}_{\crys}(V)/\Fil^0$ instead of $\textbf{D}_{\crys}(V)$, but in our case, we have $\Fil^0=\{0\}$. The $\varphi$-module $D$ is called regular whenever $r_D$ is an isomorphism (see \cite[§4.1.3]{benoiscrys}).

\begin{lemme}\label{lem:equivalence_admissible_régulier}
	Under the identification (\ref{eq:identification_dieudonné_module}) any $\varphi$-submodule $D$ of $\textbf{D}_\crys(\widecheck{W}_p)$ of $E_p$-dimension $d^-$ can be uniquely written as $D=\Hom\left(W_p/W^+_p,E_p\right)$ where $W_p^+$ is a $p$-stabilization of $W_p$, and any $p$-stabilization $W_p^+$ of $W_p$ defines a $\varphi$-submodule in this way. It is moreover regular if and only if $W_p^+$ is admissible. 
\end{lemme}
\begin{proof}
	The first claim is obvious. Let us prove the second claim and put $D=\Hom(W_p/W_p^+,E_p)$, where $W_p^+$ is a $p$-stabilization of $W_p$. Under the identification (\ref{eq:identification_dieudonné_module}) the composite map $\log_\BK \circ \loc_p$ coincides with the composite map given in (\ref{eq:loc_p_bloch_kato}). Therefore, $r_D$ coincides with the map $\HH^1_\ff(\bQ,\widecheck{W}_p) \longrightarrow \Hom(W_p^+,E_p)$ induced by the $p$-adic pairing (\ref{eq:intro_p_adic_pairing}), which, by definition, is an isomorphism if and only if $W_p^+$ is admissible.
\end{proof}
Given a pure motive of weight $-2$ whose $p$-adic realization $V$ satisfies conditions \textbf{(C1-C5)} of \cite[§4.1.2]{benoiscrys} and given a regular submodule $D$ of $\textbf{D}_\crys(V)$, Benois has defined an $\cL$-invariant $\cL(V,D)$ \cite[§4.1.4]{benoiscrys}. It is not hard to see that $V=\widecheck{W}_p$ satisfies the above-mentioned conditions: the first one follows from the finiteness of the ideal class group of $H$, the second one from the running assumption $\HH^0(\bQ,W)=0$, the third and fourth ones from the unramifiedness assumption at $p$ and from the semi-simplicity of $\rho(\sigma_p)$, and the last one is true whenever there exists at least one regular submodule $D$ of $\textbf{D}_\crys(\widecheck{W}_p)$.
\begin{lemme}\label{lem:comparaison_L_invariant_Benois}
		Let $W_p^+$ be an admissible $p$-stabilization of $W_p$, let $V=\widecheck{W}_p$ be the $p$-adic realization of $\widecheck{M}$ and let $D$ be the regular submodule of $\textbf{D}_\crys(V)$ defined as in (\ref{eq:identification_dieudonné_module}). Then 
		$$\cL(\rho,\rho^+)=(-1)^e \cdot  \cL(V,D),$$ 
		where $\cL(V,D)$ is Benois' $\cL$-invariant for $V$ and $D$ as defined in \cite[§4.1.4]{benoiscrys}. 
\end{lemme}
\begin{proof}
It is shown in \cite[Thm. 2.15]{maksouddimitrov} that $\cL(V,D)=(-1)^e \cdot \dfrac{\det \begin{pmatrix}A^+ & B^+ \\ A^- & B^- \end{pmatrix} }{\det(A^+)\det(O^-)}$, where $A^\pm$, $B^\pm$, $O^-$ are the matrices defined in \S\ref{sec:L-invariant}. Therefore, one has $\cL(\rho,\rho^+)=(-1)^e \cdot  \cL(V,D)$ by Lemma \ref{lem:formule_L_invariant_quot_det}.
\end{proof}

The main algebraic object in Perrin-Riou's formulation of Iwasawa theory for a motive that is crystalline at $p$ is the module of $p$-adic $L$-functions, introduced and studied in \cite[Chapter 2]{perrin1995fonctions} and later interpreted (and generalized) in \cite[§6.2.3]{benoiscrys} in terms of Selmer complexes. Its definition only makes sense when $\cL(V,D)\neq 0$ and under the Weak Leopoldt conjecture for $V$ and for $\widecheck{V}=V^*(1)$ together with conditions \textbf{(C1-C5)} of \textit{loc. cit.}. It is denoted by $\textbf{L}_{\Iw,h}^{(\eta_0)}(N,T)$ in \textit{loc. cit.} and it depends on the choice of a $G_{\bQ}$-stable lattice $T$ of $V$, on the choice of a $\cO_p$-lattice $N$ of a regular submodule $D$ of $\textbf{D}_\crys(V)$ and on a parameter $h>0$. 

We consider here the case of the dual motive $\widecheck{M}$ of $\rho$. More precisely, let $T_p$ be a $G_{\bQ}$-stable lattice of $W_p$ and let $W_p^+$ be an admissible $p$-stabilization of $W_p$. We put $V=\widecheck{W}_p$, $T=\widecheck{T}_p$, $V^-=\widecheck{W}_p^-$ and $T^-=\widecheck{T}_p^-$ in this paragraph. Under the identification (\ref{eq:identification_dieudonné_module}), we define a regular submodule of $\textbf{D}_\crys(V)$ by letting $D=\textbf{D}_\crys(V^-)=\Hom_{G_p}(W_p^-,t^{-1}K \otimes E_p)$ (see Lemma \ref{lem:equivalence_admissible_régulier}). Explicitly, Bloch-Kato's logarithm map for $V^-$ is the isomorphism
\begin{equation}\label{eq:log_BK_explicite}
\HH^1_\ff(\Qp,V^-)=\Hom_{G_p}(W_p^-,\cO_{K}^{\times,1} \otimes E_p)\overset{\sim}{\longrightarrow}\Hom_{G_p}(W_p^-,t^{-1} K\otimes E_p)=D
\end{equation}
induced by the $p$-adic logarithm $\log_p : \cO_{K}^{\times,1} \hookrightarrow K\simeq t^{-1}K$, where $\cO_{K}^{\times,1}$ is the group of principal units of $K$. We define an $\cO_p$-lattice of $D$ by letting
$$N=\textbf{D}_\crys(T^-)=\Hom_{G_p}(T_p^-,t^{-1}\cO_K \otimes \cO_p).$$
We also may set $h=1$ in the definition of  $\textbf{L}_{\Iw,h}^{(\eta_0)}(N,T)$ since the $G_{\Qp}$-representation $V$ is the Tate twist of an unramified representation.

\begin{proposition}\label{prop:comparaison_idéaux_caractéristiques_Benois}
	Assume that $W_p^+$ is admissible and that $\cL(\rho,\rho^+)\neq 0$. Then $X_\infty(\rho,\rho^+)$ is of $\Lambda$-torsion, and we have $\textbf{L}_{\Iw,h}^{(\eta_0)}(N,T)=\char_\Lambda X_\infty(\rho,\rho^+)$.
\end{proposition}
\begin{proof}
	Fix a topological generator $\gamma$ of $\Gamma$ and let 
	$$\cH=\left\{\: f(\gamma-1) \ \big| \ f(X)\in E_p[[X]] \textrm{ is holomorphic on the open unit disc\:}\right\}$$ 
	be the large Iwasawa algebra. Consider the complex of $\cH$-modules $\textbf{R}\Gamma^{(\eta_0)}_{\Iw,h}(D,V)$ defined in \cite[§6.1.2]{benoiscrys}. Note that it is possible because we already checked conditions \textbf{(C1-5)} and because the Weak Leopoldt conjecture for $V$ and $V^*(1)$ holds by Proposition \ref{prop:conjecture_leopoldt_faible}. It is a Selmer complex in the sense of \cite[(6.1)]{sel} given by the following local conditions: at finite primes $\ell\neq p$ we take the unramified condition, and at $p$ we take the derived version of Perrin-Riou's exponential map 
	$$\textrm{Exp}_{V,h}: (N\otimes \Lambda) \otimes_{\Lambda}\cH \longrightarrow\HH^1_\Iw(\Qp,\widecheck{T}_p^-)\otimes_\Lambda \cH\subseteq \HH^1_\Iw(\Qp,T)\otimes_\Lambda \cH.$$ 
	As explained in \cite[§4.1.3-5]{perrinriou1994}, the map $\textrm{Exp}_{V,h}$ is induced by the inverse of Coleman's isomorphism $\Coleman_N: \HH^1_{\Iw,\ff}(\Qp,\widecheck{T}_p^-) \overset{\sim}{\longrightarrow}N \otimes \Lambda$. Since $\cH$ is flat over $\Lambda$, the complex $\textbf{R}\Gamma^{(\eta_0)}_{\Iw,h}(D,V)$ is a base change to $\cH$ of a Selmer complex $\textbf{R}\Gamma_{\Iw}(\rho,\rho^+)$ over $\Lambda$ given by the unramified condition at $\ell\neq p$, and at $p$ by the morphism of complexes $N\otimes\Lambda[-1] \longrightarrow \textbf{R}\Gamma_{\Iw}(\Qp,T)$ induced by $(\Coleman_N)^{-1}$ in degree 1. By \cite[Thm. 4]{benoiscrys}, the $\Lambda$-module $\textbf{R}^i\Gamma_{\Iw}(\rho,\rho^+)$ vanishes when $i\neq 2$ and it is of $\Lambda$-torsion for $i=2$. Moreover, as in \cite[§6.1.3.3]{benoiscrys} we have a short exact sequence 
	\[\begin{tikzcd} 0 \rar & \coker(\Loc_+') \rar &\textbf{R}^2\Gamma_{\Iw}(\rho,\rho^+) \rar & \Sha_\infty^2(\widecheck{T}_p)\rar & 0,
	\end{tikzcd}\]
	where $\Loc_+'$ is the localization map introduced in \S\ref{sec:duality}. It follows from the exactness of the first row of (\ref{eq:gros_diagramme_commutatif_Sel_infty}) and from Proposition \ref{prop:conjecture_leopoldt_faible} that $X_\infty(\rho,\rho^+)$ is also of $\Lambda$-torsion, and that it shares the same characteristic ideal with $\textbf{R}^2\Gamma_{\Iw}(\rho,\rho^+)$, the latter being equal to $\textbf{L}_{\Iw,h}^{(\eta_0)}(N,T)$ by construction.
\end{proof}

\begin{theorem}\label{coro:formule_BK}
	Assume that $W_p^+$ is admissible and that $\cL(\rho,\rho^+)\neq 0$. If Conjecture \ref{conj:IMC} holds, then the $p$-part of Bloch-Kato's conjecture (in the formulation of Fontaine and Perrin-Riou, \cite[III, 4.5.2]{fontaineperrinriou}) holds for the Artin motive associated with $\rho$, that is,
	\begin{equation}\label{eq:formule_TNC}
		\dfrac{L^{*}\left(\rho^*,0\right)}{\Reg_{\omega_\infty^+}(\rho)} \, \sim_p \, \dfrac{\#\Sha({T}_p) \cdot \prod_{\ell\neq p}\#\left(\HH^1(I_\ell,\widecheck{T}_p)^{G_{\bQ_\ell}}\right)_\textrm{tors}}{\# \HH^0(\bQ,D_p)},
	\end{equation}
	where $a\sim_p b$ means that $a$ and $b$ are equal up to a $p$-adic unit, $\Reg_{\omega_\infty^+}(\rho)$ is computed with respect to $T_p$-optimal bases $\omega_\infty^+$ and $\omega_\ff$ of $\det_E \HH^0(\bR,W)$ and $\det_{E_\eta} \HH^1_\ff(\rho^*(1))$ respectively, $\Sha(T_p)$ is the Tate-Shafarevitch group of $T_p$ \cite[II,5.3.4]{fontaineperrinriou} and $I_\ell$ is the absolute inertia group at $\ell$. In particular, when $p$ does not divide the order of the image of $\rho$, one has 
	$$\dfrac{L^{*}\left(\rho^*,0\right)}{\Reg_{\omega_\infty^+}(\rho)}  \sim_p \#\Hom_{\cO_p[G]}(T_p,\cO_p \otimes_\bZ\mathscr{C}\!\ell(H)),$$
	where $\mathscr{C}\!\ell(H)$ is the ideal class group of $H$, the field out by $\rho$.
\end{theorem}
\begin{proof}
	Let $V=\widecheck{W}_p$ and $T=\widecheck{T}_p$ as before.
	Assume Conjecture \ref{conj:IMC} for $(\rho,\rho^+)$ and consider the $p$-adic measure $\theta_{\rho,\rho^+}'$ of Proposition \ref{prop:renormalisation_mesure_p_adique}. Fix $T_p$-optimal bases $\omega_p^+$ and $\omega_{\ff}$ of $\det_{E_p} W_p^+$ and $\HH^1(\bQ,V)$ respectively. Proposition \ref{prop:comparaison_idéaux_caractéristiques_Benois} tells us that, under \textbf{IMC}$_{\rho,\rho^+}$, the $p$-adic analytic function $s\mapsto \kappa^s(\theta_{\rho,\rho^+}')$ is equal (up to a unit) to the one denoted $L_{\Iw,h}(T,N,s)$ in \cite[§6.2.3]{benoiscrys}. Therefore, \cite[Cor. 2]{benoiscrys} reads, with the notations of \textit{loc. cit.},
	\begin{equation}\label{eq:formule_intermediaire}
		\frac{1}{e!} \dfrac{\diff^e}{\diff \!s^e}\kappa^s(\theta_{\rho,\rho^+}')\bigg|_{s=0} \,\sim_p\, R_{V,D}(\omega_{T,N})\cdot \cL(V,D)\cdot \cE^+(V,D) \cdot \dfrac{\#\Sha({T}_p) \cdot \textrm{Tam}^0_{\omega_T}(T)}{\# \HH^0(\bQ,D_p)\cdot \# \HH^0(\bQ,\widecheck{D}_p)}. 
	\end{equation}
	Since $T_p$ is unramified, note that $\HH^0(\bQ,\widecheck{D}_p)=\{0\}$, as  $\HH^0(I_p,\widecheck{D}_p)=\Hom_{\cO_p}(T_p,\mu_{p^\infty}^{I_p})=\{0\}$. From Lemma \ref{lem:comparaison_L_invariant_Benois}, we know that $\cL(V,D) \sim_p \cL(\rho,\rho^+)$, and from the description in  \eqref{eq:identification_dieudonné_module} of the $\varphi$-module $\textbf{D}_\crys(V)$, it is not hard to see that $\cE^+(V,D)=\cE(\rho,\rho^+)$.
	
	By definition, $\textrm{Tam}^0_{\omega_T}(T)$ is a product of local Tamagawa factors $\textrm{Tam}^0_{p,\omega_T}(T) \cdot \prod_{\ell \neq p} \textrm{Tam}^0_{\ell}(T)$ whose recipe is given in \cite[Chapitre I, \S4.1]{fontaineperrinriou}. It is shown in \textit{loc. cit.} that $\textrm{Tam}^0_{\ell}(T)=\#\left(\HH^1(I_\ell,\widecheck{T}_p)^{G_{\bQ_\ell}}\right)_\textrm{tors}$ for every $\ell \neq p$. 
	
	It remains to determine the product $R_{V,D}(\omega_{T,N})\cdot \textrm{Tam}^0_{p,\omega_T}(T)$. Each term of this product depends on the choice of a basis of $\textbf{D}_{\crys}(V)$, but the product does not. The term $R_{V,D}(\omega_{T,N})$ also depends on the choice of $\cO_p$-bases $\omega_\ff$ and $\omega_N$ of $\det_{\cO_p}\HH_\ff^1(\bQ,T)$ and $\det_{\cO_p}N$ respectively. Since $\omega_p^+$ is $T_p$-optimal, it yields a basis $(\omega_p^+)^*$ of $\det_{\cO_p}\textbf{D}_\crys(T^+)=\det_{\cO_p}(T^+_p,t^{-1}\cO_K \otimes \cO_p)$. We take $\omega_T$ to be the image of $\omega_N \otimes (\omega_p^+)^*$ under the natural map 
	\[\det \textbf{D}_\crys(T^-) \otimes \det  \textbf{D}_\crys(T^+) \longrightarrow \det  \textbf{D}_\crys(T) \subseteq \det  \textbf{D}_\crys(V).\]
	
	By definition, $R_{V,D}(\omega_{T,N})$ is equal to the determinant of the map $r_D$ of \eqref{eq:def_reg_r_D} with respect to $\omega_\ff$, $\omega_T$ and $\omega_N$. Lemma \ref{lem:equivalence_admissible_régulier} and its proof then show that $R_{V,D}(\omega_{T,N})= \Reg_{\omega_p^+}(\rho)$. We also obtain from the definition that $\textrm{Tam}^0_{p,\omega_T}(T)\sim_p 1$, using that $\omega_T$ is a basis of $\det_{\cO_p}\textbf{D}_\crys(T)$.
	 
	 We obtain \eqref{eq:formule_TNC} from \textbf{EZC}$_{\rho,\rho^+}$ and \eqref{eq:formule_intermediaire} after simplification by $\Reg_{\omega_p^+}(\rho)\cdot \cL(\rho,\rho^+)\cdot \cE(\rho,\rho^+)\neq 0$.
	
	We now explain how to simplify \eqref{eq:formule_TNC} in the case where $p$ does not divide the order of $G$. Note that $\HH^0(G,D_p) \hookrightarrow \HH^1(G,T_p)=0$, as this last $\cO_p$-module is killed by $\#G$. To see that the local Tamagawa numbers are all trivial, let us fix any prime $\ell\neq p$. By \cite[Lem. 3.2 (ii) and Lemma 3.5 (ii-iii)]{rubinES}, one has $\left(\HH^1(I_\ell,\widecheck{T}_p)^{G_{\bQ_\ell}}\right)_\textrm{tors}\simeq \cW^{\sigma_\ell=1}$, where $\cW$ is the quotient of $\widecheck{D}_p^{I_\ell}$ by its divisible part. But the action of $I_\ell$ on $\widecheck{D}_p$ factors through a finite group of prime-to-$p$ order, so we must have $\cW=0$. Finally, the description of $\Sha(T_p)$ in terms of class groups directly follows from the inflation-restriction exact sequence. 
\end{proof}
\begin{remark}
	The proof of Theorem \ref{coro:formule_BK} used Benois' computation of the leading term of the characteristic series of a Selmer group, which is a generalization of Greenberg's computation \cite[Prop. 4]{greenberg1994trivial} to the case of non-critical and non-ordinary motives. A proof of \eqref{eq:formule_intermediaire} could possibly be adapted from Greenberg's arguments, but extra work would be needed in order to relax the assumptions in \textit{loc. cit.}, which force all the terms $\cE^+(V,D)$, $\# \HH^0(\bQ,D_p)$, $\# \HH^0(\bQ,\widecheck{D}_p)$, $\textrm{Tam}^0_{\ell}(T)$ ($\ell \neq p$) to be $p$-adic units.
\end{remark}

Recall that $A_\infty'=\varprojlim_n A_n'$ is the inverse limit over $n$ of the $p$-split ideal class group of $H_n$ and it has a structure of $\Lambda[G]$-module (see Notation \ref{nota:units_and_class_groups}). We end this section with some applications of our results to the Gross-Kuz'min conjecture \cite{kuzmin,gross1981padic}, also called ``Gross' finiteness conjecture'' in \cite[Thm. 1.1]{BKSANT}. \begin{conjecture}[Gross-Kuz'min conjecture]\label{conj:gross_kuzmin_general}
	The module of $\Gamma$-coinvariants of $A'_\infty$ is finite.
\end{conjecture} 
Our main contribution concerns the $\rho$-isotypic part of the Gross-Kuz'min conjecture which can be stated as follows.
\begin{conjecture}[$\rho$-part of the Gross-Kuz'min conjecture]\label{conj:gross_finiteness_conj}
 The module of $\Gamma$-coinvariants of $\Hom_G(T_p,A_\infty')$ is finite.
\end{conjecture}
Conjecture \ref{conj:gross_finiteness_conj} is equivalent to the vanishing of $\Hom_G(W_p,A_\infty' \otimes \Qp)_\Gamma$, so it does not depend on the choice of the lattice $T_p$. In particular, Conjecture \ref{conj:gross_kuzmin_general} is true if and only if Conjecture \ref{conj:gross_finiteness_conj} is true for all $p$-adic representations $\rho$ of $\Gal(H/\bQ)$.
\begin{theorem}\label{thm:exact_order_of_vanishing}
	Let $f=\dim\HH^0(\Qp,W_p)$.
	\begin{enumerate} 
		\item Let $W_p^+$ be any admissible $p$-stabilization of $W_p$ such that $\cL(\rho,\rho^+)\neq 0$. Any generator of the characteristic ideal of $X_\infty(\rho,\rho^+)$ belongs to $\cA^e\backslash \cA^{e+1}$, where $\cA\subseteq \Lambda$ is the augmentation ideal and  $e=\dim \HH^0(\Qp,W_p^-)$. 
		\item If there exists at least one admissible $p$-stabilization $W_p^+$ of $W_p$ such that $\cL(\rho,\rho^+)\neq 0$, then Conjecture \ref{conj:gross_finiteness_conj} holds.
		\item If $f=0$ and if the $\rho$-isotypic component of Leopoldt's conjecture for $H$ and $p$ holds (see (\ref{eq:loc_p_et_leopoldt})), then Conjecture \ref{conj:gross_finiteness_conj} holds as well.
		\item If $f\leq 1$ and $d^+\leq 1$, then Conjecture \ref{conj:gross_finiteness_conj} holds.  
	\end{enumerate}
\end{theorem}
\begin{proof}
	The first statement follows from Proposition \ref{prop:comparaison_idéaux_caractéristiques_Benois} and from \cite[Thm. 5 (i)]{benoiscrys}. For the three other statements, first note that, by the exactness of the third row of (\ref{eq:gros_diagramme_commutatif_Sel_infty}), by Proposition \ref{prop:conjecture_leopoldt_faible} (1) and by Lemma \ref{lem:relation_cohomologie_et_unités} (ii), the existence of a $p$-stabilization $\rho^+$ such that $X_\infty^\str(\rho,\rho^+)$ has finite $\Gamma$-coinvariants immediately implies Conjecture \ref{conj:gross_finiteness_conj}. Therefore, claim (2) follows from (1) and from Lemma \ref{lem:suite_exacte_zeros_triviaux_GV}. 
	
	Consider (3) and assume that $f=0$ and that the map in (\ref{eq:loc_p_et_leopoldt}) is injective for $\eta=\mathds{1}$. By Lemma \ref{lem:existence_p_stabilization}, there exists an admissible $p$-stabilization $\rho^+$ of $\rho$, and since $f=0$, one must have $e=0$ as well. Therefore, $\cL(\rho,\rho^+)=1$ by Lemma \ref{lem:formule_L_invariant_quot_det}, so (3) follows from (2). 
	
	Let us prove (4), and assume that $f,d^+\leq 1$. We may assume $\rho$ irreducible. When $d^+=1$, it is easy to produce a motivic $p$-stabilization $\rho^+$ such that $e=0$, so $\cL(\rho,\rho^+)=1$ (take $W_p^+$ containing $\HH^0(\Qp,W_p)$). Since every motivic $p$-stabilization is automatically admissible by Lemma \ref{lem:schanuel+motivic_implies_admissible}, Conjecture \ref{conj:gross_finiteness_conj} follows in this case from (2). The case where $d^+=0$ follows from \cite[Prop. 2.13]{gross1981padic}, once we have checked that $\cL(\rho,\rho^+)$ generalizes Gross's regulator for $\rho^+=0$ (see \S\ref{sec:exemple_odd_motives} for details).
\end{proof}

\subsection{Changing the $p$-stabilization}\label{sec:changement_p_stabilisation}
Let $t_1,\ldots,t_d$ be an eigenbasis of $T_p$ for $\sigma_p$. We may define a basis $\{\omega_{p,\alpha}\ | \ \alpha\in I\}$ of ${\bigwedge}^{d^+}W_p$ by letting $\omega_{p,\alpha}^+=t_{i_1}\wedge\ldots\wedge t_{i_{d^+}}$, where $\alpha=(1\leq i_{1}< \ldots< i_{d^+}\leq d)$ runs over the set $I$ of strictly increasing sequences of $d^+$ integers between $1$ and $d$. For each $\alpha\in I$, $\omega_{p,\alpha}^+$ defines a $T_p$-optimal basis of a $p$-stabilization $(\rho_\alpha^+,W_{p,\alpha}^+)$ of $W_p$. 

Let $\omega_p^+\in {\bigwedge}^{d^+}W_p$ be a $T_p$-optimal eigenbasis of a given $p$-stabilization $(\rho^+,W_p^+)$ of $W_p$. Write $\omega_p^+$ as $\sum_{\alpha\in I} c_\alpha \cdot\omega_{p,\alpha}^+$ for $c_\alpha\in \cO_p$. Writing $\omega_{p}^+$ as a pure tensor and expanding in the eigenbasis $t_1,\ldots,t_d$ shows that, for any $\alpha\in I$, we have $c_\alpha=0$ unless $\rho^+(\sigma_p)$ and $\rho^+_\alpha(\sigma_p)$ share the same list of eigenvalues. Thus, we have in particular $\cE(\rho,\rho^+)=\cE(\rho,\rho^+_\alpha)$,  $\det(\rho^\pm)(\sigma_p)=\det(\rho_\alpha^\pm)(\sigma_p)$ and $e:=\dim \HH^0(\Qp,W_p^-)=\dim \HH^0(\Qp,W_{p,\alpha}^-)$ for all $\alpha\in I_{\rho^+}=\{\alpha\in I \, /\, c_\alpha\neq 0\}$.
\begin{proposition}\label{prop:changement_de_base}
	If \textbf{EX}$_{\rho,\rho^+_\alpha}$ holds for all $\alpha\in I_{\rho^+}$, then \textbf{EX}$_{\rho,\rho^+}$ holds as well, and $\theta_{\rho,\rho^+} = \sum_{\alpha \in I_{\rho^+}}c_\alpha\cdot \theta_{\rho,\rho^+_\alpha}$.
\end{proposition}
\begin{proof}
	Assume that \textbf{EX}$_{\rho,\rho^+_\alpha}$ holds for all $\alpha\in I_{\rho^+}$. For every character $\eta\in\widehat{\Gamma}$, the rule $\omega_p^+ \mapsto \Reg_{\omega_p^+}(\rho\otimes\eta)$ (where the $p$-adic regulator is computed in a fixed basis $\omega_{\ff,\eta}$ of $\det_E \HH^1_\ff((\rho\otimes\eta)^*(1))$) defines a $E_{p}$-linear map 
	${\bigwedge}^{d^+} W_p \longrightarrow E_{p,\eta},$ so we have $$\Reg_{\omega_p^+}(\rho\otimes\eta)=\sum_{\alpha \in I_{\rho^+}} c_\alpha\cdot \Reg_{\omega^+_{p,\alpha}}(\rho\otimes\eta).$$ Therefore, the element $\theta_{\rho,\rho^+}\in\Frac(\Lambda)$ defined as $\sum_{\alpha \in I_{\rho^+}}c_\alpha\cdot \theta_{\rho,\rho^+_\alpha}$ satisfies
	\[\eta(\theta_{\rho,\rho^+}) = \sum_{\alpha \in I_{\rho^+}}c_\alpha\cdot \eta(\theta_{\rho,\rho^+_\alpha}) =M_{\rho,\eta} \cdot \sum_{\alpha \in I_{\rho^+}}c_\alpha\cdot \frac{\Reg_{\omega^+_{p,\alpha}}(\rho\otimes\eta)}{\det(\rho_\alpha^-)(\sigma_p^n))} =M_{\rho,\eta} \cdot \frac{\Reg_{\omega^+_{p}}(\rho\otimes\eta)}{\det(\rho^-)(\sigma_p^n))}\]
	for all non-trivial characters $\eta\in\widehat{\Gamma}$ of conductor $p^n$, where we have put $M_{\rho,\eta}=\frac{\tau(\eta)^{d^-}}{\tau(\rho\otimes\eta)}\ \frac{L^{*}\left((\rho \otimes \eta)^*,0\right)}{\Reg_{\omega_\infty^+}(\rho\otimes\eta)}$. Therefore, $\theta_{\rho,\rho^+}$ satisfies the interpolation property of \textbf{EX}$_{\rho,\rho^+}$. Since it has no pole except maybe at $\mathds{1}$ by construction, we have shown that \textbf{EX}$_{\rho,\rho^+}$ is valid.
\end{proof}

\begin{remark}\label{rem:p_adic_artin_formalism}
	Conjecture \ref{conj:IMC} satisfies the following ``$p$-adic Artin formalism'': if $\rho=\rho_1\bigoplus\rho_2$, and if $\rho^+$ is a $p$-stabilization of $\rho$ which splits into a sum of two $p$-stabilizations $\rho_1^+$ and $\rho_2^+$ of $\rho_1$ and $\rho_2$ respectively, then the validity of \textbf{IMC} for any two pairs in $\left\{(\rho,\rho^+),(\rho_1,\rho_1^+),(\rho_2,\rho_2^+)\right\}$ implies the validity of \textbf{IMC} for the third pair, in which case $\theta_{\rho,\rho^+}=\theta_{\rho_1,\rho_1^+}\cdot\theta_{\rho_2,\rho_2^+}$. Also, either \textbf{EX} or \textbf{EZC} for both $(\rho_1,\rho_1^+)$ and $(\rho_2,\rho_2^+)$ implies the same statement for $(\rho,\rho^+)$. However, $\rho^+$ needs not split in general even if $\rho$ is reducible. Therefore, Conjecture \ref{conj:IMC} for $\rho$ (and varying $\rho^+$) appears to be stronger than Conjecture \ref{conj:IMC} for $\rho_1$ and $\rho_2$ taken together.
\end{remark}

\section{Conjectures on Rubin-Stark elements}\label{sec:conjectures_on_RS_elts}
\subsection{The Rubin-Stark conjecture}\label{sec:RS_conjecture}
Let $H/\bQ$ be a Galois extension which is unramified at $p$ and let $\chi$ be a non-trivial $E$-valued character of $\Gal(H/k)$, where $k/\bQ$ is an intermediate extension of $H/\bQ$. Denote by $L=H^{\ker\chi}$ be the field cut out by $\chi$ and by $\Delta$ the Galois group of the abelian extension $L/k$. We fix for the moment an integer $n\geq 0$, and we put $L_n=L\bQ_n$ and $\Delta_n=\Gal(L_n/k)\simeq \Delta \times \Gamma_n$. Consider the following finite sets of places of $k$:
\begin{align*}
S&=S_\infty(k)\cup S_\ram(L/k),\\
S'&=S \cup S_p(k),\\
 V'&= \{v \in S'\ \big|\  \chi(\Delta_v)=1\}=\{v_{\infty,1},\ldots,v_{\infty,d^+},v_{p,1},\ldots,v_{p,f}\}, \\
 V&= V' \backslash S_p(k)=\{v_{\infty,1},\ldots,v_{\infty,d^+}\}.
\end{align*}
Fix once and for all a place $w_{\infty,i}$ (resp. $w_{p,i}$) of $L_n$ above $v_{\infty,i}$ (resp. above $v_{p,i}$) for all indexes $i$. Recall from \eqref{eq:definition_X_H,S} and \eqref{eq:dirichlet_reg} that we have a $\Delta_n$-equivariant isomorphism
\begin{equation*}
\lambda_{{L_n},S'} : \bR\cO_{{L_n},S'}^\times \overset{\sim}{\longrightarrow}\bR X_{{L_n},S'}, \qquad a\mapsto - \sum_{w|v\in S'}\log|a|_ww,
\end{equation*}
where $\cO^\times_{{L_n},S'}$ is the ring of $S'$-units of $L_n$. For any character $\eta\in\widehat{\Gamma}_n=\Hom(\Gamma_n,\ob{\bQ}^\times)$, the order of vanishing of the $S'$-truncated $L$-function of $(\chi \otimes \eta)^{-1}$ is, by \cite[Chapter I, Proposition 3.4]{tate},
$$r:=\ord_{s=0} L_{S'}((\chi \otimes \eta)^{-1} ,s) = \dim_{\bC}(e_{\chi\otimes \eta}\bC \cO^\times_{L_n,{S'}})= \dim_{\bC}(e_{\chi\otimes\eta}\bC X_{L_n,S'})= \left\{\begin{array}{lcc}
d^+ & \textrm{if}& \eta\neq 1 \\ d^++f & \textrm{if} &\eta= 1,
\end{array}\right. $$
where $e_{\chi\otimes \eta}= (\#\Delta_n)^{-1} \sum_{\delta\in\Delta_n} (\chi \otimes \eta)^{-1}(\delta)\delta=e_\chi \cdot e_\eta$ denotes the idempotent associated with $\chi \otimes \eta$. Thus, the limit $L^{*}_{S'}({(\chi\otimes\eta)^{-1}},0):=\lim_{s\rightarrow 0} L_{S'}(\chi,s)/s^{r} \in \bC$ is well-defined and non-zero. 

\begin{definition}
The $\chi$-part of the Rubin-Stark elements
$$\varepsilon^\chi_n \in {\bigwedge}^{d^+}_{\bC[\Gamma_n]} e_{\chi} \bC\cO^\times_{L_n,S'} \quad (n\geq 1), \qquad \left(\textrm{resp. } \ u^\chi \in {\bigwedge}^{d^++f}_{\bC} e_\chi \bC\cO^\times_{L,S'} \quad (n=0)\right), $$
 is defined to be the inverse image under $\lambda_{{L_n},S'}$ of 
 $$\left(\sum_{\eta \in \widehat{\Gamma}_n}L^{*}_{S'}({(\chi\otimes\eta)^{-1}},0) e_{\chi\otimes \eta}\right)\cdot \bigwedge_w w \in {\bigwedge}^{r}_{\bC[\Gamma_n]}e_{\chi}\bC Y_{L_n,S'}={\bigwedge}^{r}_{\bC[\Gamma_n]}e_{\chi}\bC X_{L_n,S'},$$
 where $w$ runs through $\{w_{\infty,1},\ldots,w_{\infty,d^+}\}$ (resp. through $\{w_{\infty,1},\ldots,w_{p,f}\}$).
 Note that the last equality follows from our assumption that $\chi$ is non-trivial.
\end{definition}
\begin{remark}\label{rem:RS_elts_are_p_units}
It will be convenient to see the $\chi$-part of the Rubin-Stark elements as $p$-units of $H_n$ via the equality $e_\chi\bC\cO_{L_n,S'}^\times = e_\chi\bC\cO_{H_n}[\tfrac{1}{p}]^\times$. On the other hand, the $L$-series $L_{S'}((\chi\otimes\eta)^{-1},s)$ coincides with $L_{\{p\}}((\chi\otimes\eta)^{-1},s)$ for $\chi=\mathds{1}$ and with $L((\chi\otimes\eta)^{-1},s)$ for $\chi\neq\mathds{1}$.
\end{remark} 
The Rubin-Stark conjecture over $\bQ$ \cite[Conjecture A']{rubinstark} implies the following conjecture.

\begin{conjecture}[Rubin-Stark conjecture for $\chi$: algebraicity statement]\label{conj:rubinstark_alg}
One has 
$$\varepsilon_n^\chi \in {\bigwedge}^{d^+}_{E[\Gamma_n]} e_{\chi} E\cO_{H_n}[\tfrac{1}{p}]^\times \quad (n\geq 1), \qquad \textrm{resp. } \ u^\chi \in {\bigwedge}^{d^++f}_{E} e_\chi E\cO_{H}[\tfrac{1}{p}]^\times \quad (n=0).
$$
\end{conjecture}

By means of the isomorphism $j : \bC \simeq \ob{\bQ}_p$, one may see the $\chi$-part of Rubin-Stark elements as living in the top exterior algebra of $e_\chi \ob{\bQ}_p\cO_{H_n}[\tfrac{1}{p}]^\times$. For $R=\cO_p[\Gamma_n]$ or $R=\Lambda$ and for any finitely generated $\cO_p$-free $R$-module $M$, let 
$${\bigcap}^r_{R}M := \left({\bigwedge}^r_{R}M^*\right)^* \hookrightarrow {\bigwedge}^r_{\ob{\bQ}_p\otimes R}\ob{\bQ}_p \otimes M $$
be the ($r$-th order) exterior bi-dual of $M$, where we have put $(-)^*=\Hom_{R}(-,R)$ (see \cite[§4]{BKSDocumenta} and \cite[Appendix B]{buyukboduksakamoto} for its basic properties). Note that the canonical map ${\bigwedge}^r_{R}M \longrightarrow{\bigcap}^r_{R}M$ is an isomorphism when $M$ is $R$-projective.

Recall that we denoted by $U_n$ (resp. $U_n'$) the $\cO_p$-span of the pro-$p$ completion of the group of units (resp. of $p$-units) of $H_n$ (Notation \ref{nota:units_and_class_groups}). We omit the index when $n=0$. Note that $U_n'$ is torsion-free because $p$ is unramified in $H$. The Rubin-Stark conjecture over $\bZ$ \cite[Conjecture B']{rubinstark} implies the following conjecture.

\begin{conjecture}[Rubin-Stark conjecture for $\chi$: $p$-integrality statement]\label{conj:rubinstark_p_int}
	One has 
	$$\varepsilon_n^\chi \in {\bigcap}^{d^+}_{\cO_p[\Gamma_n]} U_n' \quad (n\geq 1), \qquad \textrm{resp. } \ u^\chi \in {\bigwedge}^{d^++f}_{\cO_p} U' \quad (n=0).
	$$
\end{conjecture}
Recall that, if $\varphi$ is a linear form on a $R$-module $M$ (for a commutative ring $R$), and if $t\geq 1$ is an integer, then $\varphi$ induces a $R$-linear map $\varphi : {\bigwedge}^t_R M \longrightarrow {\bigwedge}^{t-1}_R M$ which sends $m_1\wedge\ldots\wedge m_t$ to $\sum_{i=1}^t(-1)^{i-1}m_1\wedge\ldots\wedge m_{i-1}\wedge m_{i+1}\wedge\ldots\wedge m_t$. More generally, $s$ linear forms $\varphi_1,\ldots\varphi_s$ on $M$ with $s\leq t$ induce a $R$-linear map 
$$\bigwedge_{1\leq i\leq s} \varphi_i : {\bigwedge}^t_R M \longrightarrow {\bigwedge}^{t-s}_R M$$ 
given by $m\mapsto \varphi_s \circ \ldots \circ \varphi_1 (m)$.

We take $n=0$ for the rest of this section and for $1\leq i \leq f$ we consider the $p$-adic valuation $\ord_{w_{p,i}} : \bC\cO_H[\tfrac{1}{p}]^\times \longrightarrow \bC$ induced by the place ${w_{p,i}}$. By \cite[Prop. 3.6]{sanocompositio}, the induced map
$$\bigwedge_{1 \leq i \leq f} \ord_{w_{p,i}} : {\bigwedge}_\bC^{d^++f} \bC\cO_H[\tfrac{1}{p}]^\times \longrightarrow {\bigwedge}_\bC^{d^+} \bC\cO_H[\tfrac{1}{p}]^\times$$
sends $u^\chi$ on the Rubin-Stark element 
\begin{equation}\label{eq:xi^chi}
	\xi^\chi\in {\bigwedge}_\bC^{d^+} \bC\cO_H[\tfrac{1}{p}]^\times
\end{equation}
 defined as the inverse image under $\lambda_{L,S}$ of $L^*(\chi^{-1},0)e_\chi \cdot w_{\infty,1}\wedge \ldots \wedge w_{\infty,d^+}$. If Conjecture \ref{conj:rubinstark_alg} or Conjecture \ref{conj:rubinstark_p_int} holds for $u^\chi$, then the corresponding statement for $\xi^\chi$ is also true.
\subsection{Iwasawa-theoretic conjectures}
We assume in this section that $\chi$ is non-trivial and of prime-to-$p$ order. The idempotent $e_\chi$ has coefficients in $\cO_p$ and the $\chi$-part $M^\chi$ and the $\chi$-quotient $M_\chi$ of an $\cO_p$-module $M$ (see \S\ref{sec:coleman_isotypic}) both coincide with $e_\chi M$. Moreover, $\chi$ is residually non-trivial, a fact which is used in \eqref{eq:family_RS} and, therefore, in the formulation of Conjectures \ref{conj:IMC_chi}, \ref{conj:EZC_for_RS_elements} and \ref{conj:MRS_chi}.

We let the integer $n\geq 0$ of last section vary and we assume Conjecture \ref{conj:rubinstark_p_int} for every $n$. As explained in \cite[3B2]{BKSANT} the family $(\varepsilon_n^\chi)_{n\geq 1}$ is norm-compatible, so it defines an element 
\begin{equation}\label{eq:family_RS}
	\varepsilon_\infty^\chi \in \varprojlim_n {\bigcap}^{d^+}_{\cO_p[\Gamma_n]} (U_n')^\chi = {\bigcap}^{d^+}_\Lambda (U_\infty')^\chi = {\bigwedge}^{d^+}_\Lambda (U_\infty')^\chi.
\end{equation}
Here, the first identification follows from \cite[Cor. B.5]{buyukboduksakamoto} and the second one from the fact that $(U_\infty')^\chi = \HH^1_{\Iw,\ff,p}(\bQ,\widecheck{T}_p)$ which is free (of rank $d^+$) over $\Lambda$ by the results of \S\ref{sec:limits_of_unit_groups}. The following conjecture is taken from \cite[Conj. 3.14]{BKSANT} and should be thought as a cyclotomic Iwasawa main conjecture for $\chi$. Let us mention that this conjecture may also be formulated for other $\Zp$-extensions of $k$ than the cyclotomic one.
\begin{conjecture}[\textbf{IMC}$_\chi$]\label{conj:IMC_chi}
	We have 
	$$\char_\Lambda \left({\bigwedge}^{d^+} (U_\infty')^\chi\right) \big/\left(\Lambda\cdot \varepsilon_\infty^\chi\right) = \cA^f \cdot \char_\Lambda (A_\infty')_\chi,$$
	where $\cA$ is the augmentation ideal of $\Lambda$ and  $A_\infty'$ is the inverse limit over $n\geq 0$ of the $p$-split ideal class groups of $H_n$ (see Notation \ref{nota:units_and_class_groups}).
\end{conjecture}
Since ${\bigwedge}^{d^+} (U_\infty')^\chi$ is free of rank one over $\Lambda$, Conjecture \ref{conj:IMC_chi} implies immediately the non-vanishing of $\varepsilon_\infty^\chi$, as well as the following conjecture:
\begin{conjecture}[\textbf{wEZC}$_\chi$]\label{conj:EZC_for_RS_elements}
	We have $\varepsilon_\infty^\chi \in \cA^f \cdot {\bigwedge}^{d^+} (U_\infty')^\chi$.
\end{conjecture}
This is \cite[Conj. 2.7]{buyukboduksakamoto} for the cyclotomic extension (and a more general number field $k$), where it is referred to as the Exceptional Zero Conjecture for Rubin-Stark elements. Assume Conjecture \ref{conj:EZC_for_RS_elements} and fix $\gamma$ a topological generator of $\Gamma$. Following \textit{loc. cit.} we now reformulate the (cyclotomic) Iwasawa-theoretic Mazur-Rubin-Sano Conjecture for $(\chi,S,V')$ in terms of the element $\kappa_{\infty,\gamma}\in{\bigwedge}^{d^+} (U_\infty')^\chi$ which satisfies 
$$\varepsilon_\infty^\chi = (\gamma-1)^f \cdot \kappa_{\infty,\gamma}.$$
For all $1\leq i \leq f$, let $\rec_{w_{p,i}} : L^\times \longrightarrow \Gal((\bQ_\infty L)_{w_{p,i}}/L_{w_{p,i}})\simeq \Gamma$ be the local reciprocity map for $L$ at ${w_{p,i}}$. We still denote by $\rec_{w_{p,i}}$ the induced $\cO_p$-homomorphism 
$$\rec_{w_{p,i}} : (U')^\chi=e_\chi (\cO_p\otimes \cO_{L,S'}^\times) \longrightarrow \cO_p \otimes \Gamma \simeq \cA/\cA^2.$$
\begin{conjecture}[\textbf{MRS}$_\chi$]\label{conj:MRS_chi}
	Conjecture \ref{conj:EZC_for_RS_elements} holds true, and if we let $\kappa_\gamma \in {\bigwedge}^{d^+}_{\cO_p} (U')^\chi$ be the bottom layer of $\kappa_{\infty,\gamma}$, then the map 
	$$\bigwedge_{1 \leq i \leq f} \rec_{w_{p,i}} : {\bigwedge}^{d^++f}_{\cO_p} (U')^\chi \longrightarrow \cA^f/\cA^{f+1} \otimes_{\cO_p} {\bigwedge}^{d^+}_{\cO_p} (U')^\chi$$
	sends $u^\chi$ to $(-1)^{d^+\cdot f} \cdot(\#\Delta)^{-f} \cdot (\gamma-1)^f \otimes \kappa_\gamma$.
\end{conjecture}
\begin{remark}
	As first noted in \cite[Rem. 2.10 (i)]{buyukboduksakamoto}, one may prove as in \cite[Prop. 3.13]{bullachhofer} that Conjecture \ref{conj:MRS_chi} is equivalent to \cite[Conj. 4.2, MRS($H_\infty/k$,$S$,$\emptyset$,$\chi$,$V'$)]{BKSANT}. As for Conjecture \ref{conj:rubinstark_p_int}, taking $T=\emptyset$ (in the notations of \cite{rubinstark,BKSANT}) is allowed because the $\Zp$-module $U'_n$ is torsion-free for all $n\geq 0$. Lastly, note that, while the definition of  Rubin-Stark elements depends on how we ordered the places $v_{p,1},\ldots,v_{p,f}$ of $V'\backslash V$ and on the choice of $w_{p,i}$ above $v_{p,i}$, the validity of all the conjectures of this section does not depend on these choices.
\end{remark}

\section{Monomial representations} \label{sec:monomial_representations}
\subsection{Induced representations}
Let $\rho$ be a monomial representation and fix an isomorphism $\rho \simeq \Ind_k^{\bQ}\chi$ over $E$, where $\chi : \Gal(H/k) \longrightarrow E^\times$ is a non-trivial character. We do not assume yet that $\chi$ has order prime to $p$. The underlying space of $\rho$ is then equal to $W=E[G] \otimes_{E[\Gal(H/k)]} E(\chi)$, where $E(\chi)$ is a $E$-line on which $\Gal(H/k)$ acts via $\chi$ and  the tensor product follows the rule $gh\otimes 1=g\otimes\chi(h)$ for all $g\in G,h\in\Gal(H/k)$. The (left) $G$-action on $W$ is given by $g\cdot(g'\otimes 1)=gg'\otimes 1$ for all $g,g'\in G$. By Frobenius reciprocity we have $\rho \otimes \eta \simeq \Ind_k^{\bQ}(\chi \otimes \eta)$ for any $\eta\in \widehat{\Gamma}$, where we still denoted by $\eta$ its restriction to $G_k$. We assume throughout \S\ref{sec:monomial_representations} that the $\cO_p$-lattice $T_p$ of $W_{p}$ is 
$$T_p=\cO[G]\otimes_{\cO[\Gal(H/k)]}\cO(\chi),$$ 
so that the family $(g\otimes 1)_{g\in G}$ generates $T_p$ over $\cO_p$.

\begin{lemme}\label{lem:identification_Hom(T_p,M)_et_Mchi}\hspace{2em}
\begin{enumerate}
	\item 	Given any $\cO_p[G]$-module $M$, there is a canonical isomorphism 
	$$\Hom_G(T_p,M) \overset{\sim}{\longrightarrow} {M}^\chi, \qquad \psi \mapsto \psi(1\otimes 1),$$
	where ${M}^\chi$ denotes the $\chi$-part of $M$, seen as a $\Gal(H/k)$-module. In the same fashion, for any $\eta\in\widehat{\Gamma}$ and for any $E_\eta[G_{\bQ}]$-module $M$, there is a canonical isomorphism $\Hom_{G_{\bQ}}(W_\eta,M)\simeq e_{\chi \otimes \eta} M = {M}^{\chi\otimes\eta}$.
	\item Given any $\cO_p[G]$-module $M$, the module $\Hom_G(M,D_p)$ is canonically isomorphic to the Pontryagin dual $\left(M_\chi\right)^\vee$ of the $\chi$-quotient of $M$.
\end{enumerate}
\end{lemme}
\begin{proof}
	The first statement is straightforward to check, so we only prove the second one. Once we have fixed a generator of the different of $\cO_p$ over $\Zp$, the $\cO_p$-module $\left(M_\chi\right)^\vee$ can be identified with $\Hom_{\cO_p}(M_\chi,E_p/\cO_p)$ as in \S\ref{sec:selmer_group}. On the other hand, since induction and co-induction functors over finite groups coincide, $D_p$ can be described as the $\cO_p$-module of maps $f : G \longrightarrow E_p/\cO_p \otimes \cO_p(\chi)$ satisfying $f(hg)=h\cdot f(g)$ for all $h\in\Gal(H/k)$ and $g\in G$, the left $G$-action being $(g\cdot f)(g')=f(g'g)$. Therefore, the map $F \mapsto (m \mapsto F(m)(1\otimes 1))$ identifies $\Hom_G(M,D_p)$ with $\Hom_{\cO_p}(M_\chi,E_p/\cO_p)$.
\end{proof}
\begin{notation}\label{nota:m_et_psi_m}
	For any $\eta\in\widehat{\Gamma}$, for any module $M$ as in Lemma \ref{lem:identification_Hom(T_p,M)_et_Mchi} (1) and for any $m\in M$, we  let $\psi_m$ be the element of $\Hom_G(T_p,M)$ (resp. of $\Hom_{G_\bQ}(W_\eta,M)$) which satisfies $\psi_m(1\otimes 1)=m$. More generally, for any $\omega\in {\bigwedge}^r M$ (and $r\geq 0$) we let $\psi_\omega$ be the element of $ {\bigwedge}_{\cO_p}^r \Hom_G(T_p,M)$ (resp. of ${\bigwedge}_{E_\eta}^r \Hom_{G_{\bQ}}(W_\eta,M)$) corresponding to $\omega$ under the induced isomorphism on exterior products.
\end{notation}

\subsection{Complex regulators}
We first define a natural basis $\omega_\infty^+$ of $\det_E\HH^0(\bR,W)$ in which we will compute all the complex regulators. The embedding $\iota_\infty:\ob{\bQ}\subseteq \bC$ defines a place $w_\infty$ (resp. $v_\infty$) of $\ob{\bQ}$ (resp. of $k$) as well as a complex conjugation which is denoted by $\sigma_\infty$. As in \S\ref{sec:RS_conjecture}, we denote by $V=\{v_{\infty,1},\ldots,v_{\infty,d^+}\}$ the set of archimedean places of $k$ which split completely in $L=H^{\ker \chi}$. We choose for each $i=1,\ldots,d^+$ an automorphism $\tau_{\infty,i}\in G_{\bQ}$ which sends $v_{\infty,i}$ onto $v_\infty$ and we put $w_{\infty,i}=\tau^{-1}_{\infty,i}(w_\infty)$. For simplicity we still write $\tau_{\infty,i}$ and $w_{\infty,i}$ for their restrictions to finite extensions of $L$. We obtain a basis $\omega_\infty^+=\{t_{\infty,1},\ldots,t_{\infty,d^+}\}$ of $\det_E H^0(\bR,W)=\det_E W^{\sigma_\infty=1}$ by letting 
$$t_{\infty,i}= \left\{\begin{array}{ll}
\tau_{\infty,i}\otimes 1 &  \textrm{if } v_i \textrm{ is real,}\\ 
\tau_{\infty,i}\otimes 1 + \sigma_\infty\cdot\tau_{\infty,i}\otimes 1 &  \textrm{if } v_i \textrm{ is complex.}
\end{array}\right.$$
Note that it is moreover $T_p$-optimal for our fixed choice of $T_p$. 	

\begin{lemme}\label{lem:calcul_reg_complexe_RS}
	Assume Conjecture \ref{conj:rubinstark_alg}. Let $\eta \in \widehat{\Gamma}$ be a character of order $p^{n}$. Put $\omega_{\ff,\eta}=\psi_{e_\eta\cdot \varepsilon_n^\chi}$ if $\eta\neq\mathds{1}$ and $\omega_{\ff,\eta}=\psi_{\xi^\chi}$ if $\eta=\mathds{1}$ (see Notation \ref{nota:m_et_psi_m}). The complex regulator of $\rho\otimes\eta$ computed in the bases $\omega_\infty^+$ and $\omega_{\ff,\eta}$ is equal to
	$$\Reg_{\omega_\infty^+}(\rho\otimes \eta)= p^{-n\cdot d^+}\cdot L^*((\rho\otimes \eta)^*,0).$$
\end{lemme}
\begin{proof}
 For $a \in L_n$ which is seen in $L_{n,w_{\infty,i}}$, put 
$$|a|_{w_i}=\left\{\begin{array}{ll}
\textrm{sgn}(a)a &  \textrm{if } v_i \textrm{ is real,}\\ 
a\cdot \ob{a} &  \textrm{if } v_i \textrm{ is complex,}
\end{array}\right.$$
where sgn is the sign function when $L_{n,w_{\infty,i}}=\bR$ and $a\mapsto \ob{a}$ is the complex conjugation when $L_{n,w_{\infty,i}}=\bC$. Write $e_\eta \cdot \varepsilon_n^{\chi}$ (or $\xi^\chi$ if $\eta$ is trivial) as $\mu_1\wedge\ldots\wedge\mu_{d^+}$ and write $\omega_{\ff,\eta}=\psi_{\mu_1}\wedge\ldots\wedge\psi_{\mu_{d^+}}$ accordingly. Then by construction of the $t_{\infty,i}$'s, one has $1\otimes (\iota_\infty)(\psi_{\mu_{j}}(t_{\infty,i}))=|\mu_j|_{w_{\infty,i}}\in E_\eta \otimes \bR^\times$ for all $1\leq i,j\leq d^+$, so we have
$$\Reg_{\omega_\infty^+}(\rho\otimes \eta)= \det\left(\log_\infty|\mu_j|_{w_{\infty,i}}\right)_{1\leq i,j\leq d^+}.$$
Since $L(\rho\otimes\eta)^*,s)=L((\chi\otimes\eta)^{-1},s)$ and since $e_\eta \cdot \varepsilon_n^{\chi}=p^{-n} \sum_{g\in\Gamma_n}\eta^{-1}(g)g(\varepsilon_n^{\chi})$ by definition, the result follows directly from \cite[Lem. 2.2]{rubinstark} and Remark \ref{rem:RS_elts_are_p_units}.
\end{proof}

\subsection{Iwasawa main conjectures}
In the two last sections we explore the relation between Conjecture \ref{conj:IMC} and the various conjectures on Rubin-Stark elements of \S\ref{sec:RS_conjecture}. We henceforth assume that $\chi$ is of prime-to-$p$ order. In what follows, the basis of a given $p$-stabilization $W_p^+$ of $W_p$ is always assumed to be a $T_p$-optimal eigenbasis for $\sigma_p$ as in Lemma \ref{lem:interpolation_C+}.

\begin{theorem}\label{thm:comparaison_mesure_p_adique_RS_elts}
	Assume Conjectures \ref{conj:rubinstark_alg} and \ref{conj:rubinstark_p_int} and pick any $p$-stabilization $W_p^+$ of $W_p$.
	 \begin{enumerate}
		\item The statement \textbf{EX}$_{\rho,\rho^+}$ in Conjecture \ref{conj:IMC} is true, and the element $\theta_{\rho,\rho^+}'$ of Proposition \ref{prop:renormalisation_mesure_p_adique} coincides with $\sC^\str_{\omega_p^+}(\psi_{\varepsilon_\infty^\chi})$, where $\sC_{\omega_p^+}^\str$ is the operator introduced in \S\ref{sec:torsionness_Sel}.
		\item Conjecture \ref{conj:EZC_for_RS_elements} implies that $\theta_{\rho,\rho^+}$ has an order of vanishing at $\mathds{1}$ greater than or equal to $e$. The converse implication also holds if we moreover assume that $W_p^+$ is admissible and that $\cL(\rho,\rho^+)$ does not vanish.
	\end{enumerate}
\end{theorem}
\begin{proof}
    Put $\theta=\sC^\str_{\omega_p^+}(\psi_{\varepsilon_\infty^\chi})$ and fix a non-trivial character $\eta\in\widehat{\Gamma}$ of conductor $p^n$. By Lemma \ref{lem:interpolation_C+} and Lemma \ref{lem:calcul_reg_complexe_RS}, we have 
    \begin{align*}
    \eta(\theta) &= p^{(1-n)\cdot d^+}\cdot\dfrac{\det(\rho^+)(\sigma_p^n)}{\g(\eta^{-1})^{d^+}}\Reg_{\omega_p^+}(\rho\otimes\eta)\\
    &=\dfrac{\det(\rho^+)(\sigma_p^n)}{\g(\eta^{-1})^{d^+}} \Reg_{\omega_p^+}(\rho\otimes\eta)\cdot j\left(\dfrac{L^*((\rho\otimes\eta)^*,0)}{\Reg_{\omega_\infty^+}(\rho\otimes\eta)}\right),
    \end{align*}
    where the regulators are computed with respect to the basis $\omega_{\ff_\eta}$ defined in Lemma \ref{lem:calcul_reg_complexe_RS}. A comparison with the interpolation property of $\theta_{\rho,\rho^+}'$ and Weierstrass' preparation theorem then shows that $\theta_{\rho,\rho^+}'=\theta$; hence, \textbf{EX}$_{\rho,\rho^+}$ is true by Proposition \ref{prop:renormalisation_mesure_p_adique}, and (1) follows. 
    
    The first implication of (2) is obvious, since $\sC^\str_{\omega_p^+}$ is $\Lambda$-linear and since $\theta$ has at most a pole of order $f-e$ at $\mathds{1}$ by construction. For the converse implication, assume that $\theta_{\rho,\rho^+}'\in\cA^e$ (so $\theta$ is also in $\cA^e$), that $W_p^+$ is admissible and that $\cL(\rho,\rho^+)\neq 0$. Then, we know by Theorem \ref{thm:exact_order_of_vanishing} (1), by Lemma \ref{lem:suite_exacte_zeros_triviaux_GV} and by the exactness of the last row of (\ref{eq:gros_diagramme_commutatif_Sel_infty}) that the cokernel of $\sC^\str_{\omega_p^+}$ has finite $\Gamma$-coinvariants, so its image is generated over $\Lambda$ by an element in $\cI^{f-e}\backslash\cI^{f-e-1}$. We may then write $\theta$ as $(\gamma-1)^f\cdot \theta_\gamma$ for some topological generator $\gamma$ of $\Gamma$ and some $\theta_\gamma\in \im(\sC^\str_{\omega_p^+})$; hence, Conjecture \ref{conj:EZC_for_RS_elements} follows from the injectivity of $\sC^\str_{\omega_p^+}$.
\end{proof}

\begin{theorem}\label{thm:equivalence_IMCs}
    Assume Conjectures \ref{conj:rubinstark_alg} and \ref{conj:rubinstark_p_int}. Let $W_p^+$ be a $p$-stabilization of $W_p$ such that $X_\infty(\rho,\rho^+)$ is of $\Lambda$-torsion.
    \begin{enumerate}
    	\item If either \textbf{IMC}$_{\rho,\rho^+}$ or \textbf{IMC}$_\chi$ are true, then Conjecture \ref{conj:EZC_for_RS_elements} is also true.
    	\item \textbf{IMC}$_{\rho,\rho^+}$ and \textbf{IMC}$_\chi$ are equivalent.
    \end{enumerate}
\end{theorem}
\begin{proof}
    We have already seen that \textbf{IMC}$_\chi$ implies Conjecture \ref{conj:EZC_for_RS_elements} because ${\bigwedge}^{d^+}(U_\infty')^\chi$ is free of rank one over $\Lambda$. Thus, the first claim is implied by the second one, which we prove now. By Theorem \ref{thm:equivalence_torsionness_Sel} and by Lemma \ref{lem:freeness_Iwasawa_modules} we know that the map $\Loc_+^\str$ of \S\ref{sec:duality} is injective. Since its domain and codomain are free, $\coker(\Loc_+^\str)$ and $\coker(\wedge^{d^+}\Loc_+^\str)$ have the same characteristic ideal. On the other hand, recall that we may identify $(U_\infty')^\chi$ with $\Hom_G(T_p,U'_\infty)$ via Lemma \ref{lem:identification_Hom(T_p,M)_et_Mchi} and $\Lambda \cdot \sC_{\omega_p^+}^\str({\varepsilon_\infty^\chi})$ with $\Lambda \cdot \theta_{\rho,\rho^+}$ via Theorem \ref{thm:comparaison_mesure_p_adique_RS_elts}. By Lemma \ref{lem:suite_exacte_zeros_triviaux_GV} and by (\ref{eq:gros_diagramme_commutatif_Sel_infty}) we have exact three short exact sequences: 
    \begin{equation*}
    \begin{tikzcd}
    0 \rar & \underbrace{\left({\bigwedge}^{d^+}(U'_\infty)^\chi\right) \big/\left(\Lambda\cdot \varepsilon_\infty^\chi\right)}_B \ar[r, "\sC_{\omega_p^+}^\str"] & \underbrace{\cI^{e-f} \big/\left(\Lambda\cdot \theta_{\rho,\rho^+}\right)}_C \rar & \underbrace{\coker(\wedge^{d^+}\Loc_+^\str)}_D \rar & 0, \\
    0 \rar & \HH^0(\Qp,T^-_p) \rar & X_\infty(\rho,\rho^+) \rar & X_\infty^\str(\rho,\rho^+) \rar & 0,  \\
    0 \rar & \coker(\Loc_+^\str) \rar & X_\infty^\str(\rho,\rho^+) \rar &  \Sha_\infty^1(D_p)^\vee \rar &  0.
    \end{tikzcd}
    \end{equation*}
    The $\Gamma$-action on $\HH^0(\Qp,T^-_p)$ being trivial, its characteristic ideal is $\cA^e$. Moreover, $\Sha_\infty^1(D_p)^\vee$ and $(A'_\infty)_\chi$ are pseudo-isomorphic $\Lambda$-modules by Lemmas \ref{lem:relation_cohomologie_et_unités} (3) and \ref{lem:identification_Hom(T_p,M)_et_Mchi} (2). Therefore, by multiplicativity of characteristic ideals, we may conclude that
    \begin{align*}
    \textbf{IMC}_\chi &\Longleftrightarrow \char_\Lambda(B)=\cA^f \cdot \char_\Lambda (A'_\infty)_\chi\\
    &\Longleftrightarrow \char_\Lambda(C) = \cA^f \cdot \char_\Lambda X_\infty^\str(\rho,\rho^+)\\
    &\Longleftrightarrow \cA^{f-e}\cdot (\Lambda\cdot \theta_{\rho,\rho^+}) = \cA^{f-e}\cdot \char_\Lambda X_\infty(\rho,\rho^+)\\
    &\Longleftrightarrow (\Lambda\cdot \theta_{\rho,\rho^+}) = \char_\Lambda X_\infty(\rho,\rho^+)\\
    &\Longleftrightarrow \textbf{IMC}_{\rho,\rho^+}.
    \end{align*}\end{proof}

\subsection{Extra zeros at the trivial character}
We first construct an $\cO_p$-basis of $\HH^0(\Qp,T_p)$ as follows. Let $w_p$ be the $p$-adic place of $H$ defined by $\iota_p$, and denote by $v_{p,1}, \ldots, v_{p,f}$ the $p$-adic places of $k$ which totally split in $L$ as in \S\ref{sec:RS_conjecture}. Fix also a place $w_{p,i}$ of $H$ above $v_{p,i}$, and let 
\begin{equation}\label{eq:def_t_p,i}
t_{p,i}= [H_{w_{p,i}} : k_{v_{p,i}}]^{-1} \cdot \sum_{\mathclap{\substack{g\in G, \\ g(w_{p,i})=w_p}}} g \otimes 1 \quad \in T_p, \qquad (1\leq i \leq f).
\end{equation}
 This defines an $\cO_p$-basis $t_{p,1},\ldots,t_{p,f}$ of $\HH^0(\Qp,T_p)$. For any $t\in T_p$, consider the following two composite maps
 \begin{equation*}
 \begin{tikzcd}[column sep = 12ex]
 	\ord_p^{(t)} \ :\ \Hom_G(T_p,U') \ar[r, "\iota_p \circ \textrm{ev}_{t}"] & \cO_p \otimes \widehat{K}^\times \ar[r, "\ord_p"] & \cO_p, 
 	\end{tikzcd}  \end{equation*}
 	 \begin{equation*}\begin{tikzcd}[column sep = 12ex]
 	\log_p^{(t)} \ :\ \Hom_G(T_p,U') \ar[r, "\iota_p \circ \textrm{ev}_{t}"] & \cO_p \otimes \widehat{K}^\times \ar[r, "\log_p"] & \cO_p,
 \end{tikzcd}
 \end{equation*}
 where $\textrm{ev}_{t}$ is the evaluation map at $t$ and  $K=H_{w_p}$ as in \S\ref{sec:iwasawa}.
 
 \begin{lemme}\label{lem:calcul_ord_p_et_log_p^i}
 	For all $1\leq i \leq f$ and $u\in (U')^\chi$, we have 
 	$$\ord_p^{(t_{p,i})}(\psi_u) = [k_{v_{p,i}} : \Qp] \cdot \ord_{w_{p,i}}(u), \qquad \log_p^{(t_{p,i})}(\psi_u) = -\log_p \circ \chi_{\cyc} \circ \rec_{w_{p,i}}(u). $$
 \end{lemme}
\begin{proof}
	Fix $1\leq i \leq f$ and $u\in (U')^\chi=(\cO_p \otimes \cO_L[\tfrac{1}{p}]^\times)^\chi$. Since $\psi_u$ is $G$-equivariant, we have 
	\begin{align*}
	\iota_p(\psi_u(t_{p,i})) &= [H_{w_{p,i}} : k_{v_{p,i}}]^{-1} \cdot \iota_p \left( \prod_{{\substack{g\in G, \\ g(w_{p,i})=w_p}}} g(u)\right) \\
	&= N_i(\iota_{p,i}(u)),
	\end{align*}
	where $N_i$ is the norm map of the extension $L_{w_{p,i}}=k_{v_{p,i}}$ over $\Qp$, and  $\iota_{p,i}$ is the $p$-adic embedding of $L$ defined by $w_{p,i}$. Since $\ord_p(N_i(\iota_{p,i}(u)))= [k_{v_{p,i}} : \Qp] \cdot \ord_{w_{p,i}}(u)$ and since $\chi_{\cyc} \circ \rec_{w_{p,i}}(u)=N_i(\iota_{p,i}(u))^{-1}$, the lemma follows easily.
\end{proof}

\begin{proposition}\label{prop:equivalence_MRS}\hspace{2em}
    \begin{enumerate}
    	\item Let $W_p^0=\HH^0(\Qp,W_p)$. Then, 
    	$$\left(\bigwedge_{1\leq i \leq f} \ord_p^{(t_{p,i})}\right)(\psi_{u^\chi}) = (-1)^{d^+\cdot f}\cdot \frac{\det(1-\sigma_p^{-1}\,\big| \, W_p/W_p^0)}{(\#\Delta)^f} \cdot \psi_{\xi^\chi}$$ 
    	in $\bigwedge^{d^+}_{\ob{\bQ}_p}  \Hom_G(W_p,\ob{\bQ}_p\otimes U')$.
    	\item Assume Conjectures \ref{conj:rubinstark_alg} and \ref{conj:rubinstark_p_int} and \ref{conj:EZC_for_RS_elements}. Fix a topological generator $\gamma$ of $\Gamma$ and put $\varpi_\gamma=\log_p\circ\chi_{\cyc}(\gamma)\in p\Zp$. Then, \textbf{MRS}$_\chi$ is equivalent to the equality 
    	$$\left(\bigwedge_{1\leq i \leq f} \log_p^{(t_{p,i})}\right)(\psi_{u^\chi}) = (-1)^{d^+\cdot f}\cdot \left(-\frac{\varpi_\gamma}{\#\Delta}\right)^f \cdot \psi_{\kappa_\gamma}$$
    	in ${\bigwedge}^{d^+}_{\cO_p} \Hom_G(T_p,U')$.
    \end{enumerate}
\end{proposition}
\begin{proof}
	Let us prove (1). First of all, a direct computation gives $$\det(1-\sigma_p^{-1}\,\big| \, W_p/W_p^0)=\prod_{v\in S^{00}_p(k)} (1-\chi^{-1}(v)) \cdot\prod_{1\leq i \leq f} [k_{v_{p,i}}:\Qp],$$
	where $S^{00}_p(k)=S_p(k)-\{v_{p,1},\ldots,v_{p,f}\}$ and where we see $\chi$ as a Hecke character over $k$. Consider the operator $\Phi_{V',V}$ of \cite[Prop. 3.6]{sanocompositio} (extended by linearity to a $\bC[\Delta]$-linear map), where $V$ and $V'$ are as in \S\ref{sec:RS_conjecture}. Its $\chi$-part is given by
	$$e_\chi\cdot \Phi_{V',V} = (-1)^{d^+\cdot f} \cdot{(\#\Delta)^f} \cdot \bigwedge_{1 \leq i \leq f}\ord_{w_{p,i}}= (-1)^{d^+\cdot f} \frac{(\#\Delta)^f}{\prod_{1\leq i \leq f} [k_{v_{p,i}}:\Qp]} \cdot \bigwedge_{1 \leq i \leq f}\ord_p^{(t_{p,i})},$$
	the last equality being a consequence of Lemma \ref{lem:calcul_ord_p_et_log_p^i} (1). On the other hand, by Propositions 3.5 and 3.6 of \textit{loc. cit.}, we know that
	$$\left(e_\chi \cdot \Phi_{V',V}\right)(\psi_{u^\chi}) = \prod_{v\in S_p^{00}(k)}(1-\chi^{-1}(v))\cdot \psi_{\xi^\chi}.$$
	The claim (1) then follows from the above three equations. As for the second claim, it follows immediately from the formula for $\log_p^{(t_{p,i})}$ of Lemma \ref{lem:calcul_ord_p_et_log_p^i}.
\end{proof}

We are now in a position to state and prove a theorem comparing \textbf{MRS}$_\chi$ and \textbf{EZC}$_{\rho,\rho^+}$.

\begin{theorem}\label{thm:equivalence_MRS_EZC}
	Assume Conjectures \ref{conj:rubinstark_alg} and \ref{conj:rubinstark_p_int}.
	\begin{enumerate}
		\item \textbf{MRS}$_\chi$ implies \textbf{EZC}$_{\rho,{\rho}^+}$ for all admissible $p$-stabilizations $(\rho^+,W_p^+)$ of $W_p$.
		\item Conversely, choose any integer $0\leq e \leq d^-$ and any eigenbasis $\{\widetilde{t}_1,\ldots,\widetilde{t}_d\}$ of $W_p$ for $\sigma_p$ such that $\{\widetilde{t}_1,\ldots,\widetilde{t}_{d^++e}\}$ contains a basis of $\HH^0(\Qp,W_p)$. Let $\cB$ be the (finite) set of increasing sequences $\beta=(1\leq i_1<\ldots<i_{d^+}\leq d^++e)$ of integers between $1$ and $d^++e$ and of length $d^+$, and set $W_{p,\beta}^+=<\widetilde{t}_{i_1},\ldots,\widetilde{t}_{i_{d^+}}>_{E_p}$ for any such $\beta$. If $W_{p,\beta}^+$ is admissible and if \textbf{EZC}$_{\rho,{\rho_\beta^+}}$ holds for all $\beta\in \cB$, and if there exists at least one $\beta\in\cB$ such that $\cL(\rho,\rho^+_\beta)\neq 0$, then \textbf{MRS}$_\chi$ holds true.
	\end{enumerate}
\end{theorem}
\begin{proof}
	We know by the first part of Theorem \ref{thm:comparaison_mesure_p_adique_RS_elts} that \textbf{EX}$_{\rho,\rho^+}$ holds for any $p$-stabilization $(\rho^+,W_p^+)$ of $W_p$. By its second part, we may also assume without loss of generality that Conjecture \ref{conj:EZC_for_RS_elements} (=\textbf{wEZC}$_\chi$) holds and that $\theta_{\rho,\rho^+}$ vanishes at $\mathds{1}$ with multiplicity $\geq \dim \HH^0(\Qp,W_p^-)$ for all $\rho^+$. By Proposition \ref{prop:equivalence_MRS}, \textbf{MRS}$_\chi$ is equivalent to 
	\begin{equation}\label{eq:MRS_equivalence}
	\left(\bigwedge_{1\leq i \leq f} \log_p^{(t_{p,i})}\right)(\psi_{u^\chi}) = (-1)^{d^+\cdot f}\cdot \left(-\frac{\varpi_\gamma}{\#\Delta}\right)^f \cdot \psi_{\kappa_\gamma},
	\end{equation}
	 where $\gamma$ is a fixed generator of $\Gamma$, $\varpi_\gamma=\log_p\circ\chi_{\cyc}(\gamma)\in p\Zp$ and $\kappa_\gamma$ is the bottom layer of the element $\kappa_{\infty,\gamma}\in{\bigwedge}^{d^+}(U_\infty')^\chi$ satisfying $\varepsilon_\infty^\chi=(\gamma-1)^f\cdot\kappa_{\infty,\gamma}$. 
	
    Given any $p$-stabilization $(\rho^+,W_p^+)$ of $W_p$, we introduce the following notation: choose any $T_p$-optimal eigenbasis $\omega_p^+=t_1\wedge\ldots\wedge t_{d^+}$ of $\det_{E_p}W_p^+$ for $\sigma_p$ such that $\{t_1,\ldots,t_{f-e}\}$ generates $\HH^0(\Qp,W_p^+)$, where $e=\dim \HH^0(\Qp,W_p^-)$. Let $I^{+,0}=\{1,\ldots,f-e\}$, $I^{+,00}=\{f-e+1,\ldots,d^+\}$ and $I^{-,0}=\{d^++1,\ldots,d^++e\}$ for simplicity, and let 
	$$\nu_{\rho^+}=\bigwedge_{i\in I^{+,0}}\ord_p^{(t_i)}\wedge \bigwedge_{i\in I^{+,00}} \log_p^{(t_{i})} \in {\bigwedge^{d^+}}_{E_p}\Hom_G(W_p,\cU')^*,$$
	where $\cU'=E_p \otimes U'$ as in \S\ref{sec:L-invariant} and  $X^*=\Hom_{E_p}(X,E_p)$ for any $E_p$-vector space $X$. Note that $\nu_{\rho^+}$ depends on the choice of $\omega_p^+$ and on the ordering of the $t_i$'s, as the construction of $\theta'_{\rho,\rho^+}$ given in Theorem \ref{thm:comparaison_mesure_p_adique_RS_elts} does. Moreover, it makes sense to apply $\nu_{\rho^+}$ to both sides of the equality (\ref{eq:MRS_equivalence}), and we call for convenience $\textrm{LHS}_{\rho,\rho^+}$ and $\textrm{RHS}_{\rho,\rho^+}$ the left-hand side and the right-hand side of the resulting equality in $E_p$. We claim that, if $\rho^+$ is admissible, then this equality is equivalent to $\textbf{EZC}_{\rho,\rho^+}$. In other words:
	\begin{lemme}\label{lem:LHS_RHS}
		If $\rho^+$ is admissible, then $\textrm{LHS}_{\rho,\rho^+}=\textrm{RHS}_{\rho,\rho^+}$ if and only if the $p$-adic measure $\theta_{\rho,\rho^+}'=\sC^\str_{\omega_p^+}(\psi_{\varepsilon_\infty^\chi})$ satisfies the formula stated in Proposition \ref{prop:renormalisation_mesure_p_adique} (2).
	\end{lemme}
	\begin{proof}[Proof of the lemma]
		Assume that $\rho^+$ is admissible. We first compute $\textrm{LHS}_{\rho,\rho^+}$ as follows. Let $t_{d^++1},\ldots,t_{d^++e}$ be any elements of $\HH^0(\Qp,T_p)$ such that $\bigwedge_{i\in I^{0}} t_i=\bigwedge_{1\leq j\leq f}t_{p,j}$, where $I^0=I^{+,0}\bigcup I^{-,0}$. 
	 We have, by definition of $\cL(\rho,{\rho}^+)$, by Proposition \ref{prop:equivalence_MRS} (1), by Lemma \ref{lem:calcul_reg_complexe_RS} and taking into account the sign rule:
	\begin{align*}
	\textrm{LHS}_{\rho,\rho^+}&= \nu_{\rho^+}\wedge \left(\bigwedge_{i\in I^0} \log_p^{(t_{i})}\right)(\psi_{u^\chi})\\
	&=(-1)^{f-e}\cdot \cL(\rho,{\rho}^+)\cdot \left(\bigwedge_{i\in I^0} \ord_p^{(t_{i})}\right)\wedge \left(\bigwedge_{i\in I^+} \log_p^{(t_i)}\right)(\psi_{u^\chi}) \\
	&=(-1)^{f-e}\cdot \cL(\rho,{\rho}^+)\cdot \left(\bigwedge_{i\in I^+} \log_p^{(t_i)}\right)\left((-1)^{d^+\cdot f}\cdot  \frac{\det(1-\sigma_p^{-1}\,\big| \, W_p/W_p^0)}{(\#\Delta)^f}\cdot \psi_{\xi^\chi}\right) \\
	&=(-1)^{f-e+d^+\cdot f}\cdot \cL(\rho,{\rho}^+)\cdot \frac{\det(1-\sigma_p^{-1}\,\big| \, W_p/W_p^0)}{(\#\Delta)^f} \cdot \Reg_{{\omega}^+_{p}}(\rho) \cdot j\left(\frac{L^*(\rho^*,0)}{\Reg_{\omega_{\infty}^+}(\rho)}\right).
	\end{align*}
	Let us explicit the relation between $\textrm{RHS}_{\rho,\rho^+}$ and the $e$-th derivative of $L_p(s):=\kappa^s(\theta_{\rho,\rho^+}')$ at $s=0$. We may write $L_p(s)$ as $(\kappa(\gamma)^s-1)^e\cdot \kappa^s(\theta_\gamma)$, where $\theta_\gamma = \sC^\str_{{\omega}_p^+}((\gamma-1)^{f-e}\cdot \kappa_{\infty,\gamma})\in\Lambda$. Therefore, $\frac{1}{e!}L_p^{(e)}(0)=\varpi_\gamma^e \cdot \mathds{1}(\theta_\gamma)$. On the other hand, if we write $(\gamma-1)^{f-e}\cdot \kappa_{\infty,\gamma}$ as a wedge product of the form $(\gamma-1)\cdot v_1 \wedge \ldots (\gamma-1)\cdot v_{f-e} \wedge v_{f-e+1} \wedge \ldots \wedge v_{d^+}$, then Lemmas \ref{lem:val_spé_coleman} and \ref{lem:coleman_ord} together show that
	$$\mathds{1}(\theta_\gamma)=\varpi_\gamma^{f-e} \cdot(1-p^{-1})^{f-e}\cdot \prod_{i\in I^{+,00}}\frac{1-p^{-1}\beta_i}{1-\beta_i^{-1}} \cdot \nu_{\rho^+}(\kappa_\gamma),$$
	where $\beta_i$ is the eigenvalue of $\sigma_p$ acting on $t_i$. Since $\cE(\rho,{\rho}^+) = \prod_{i\in I^+}(1-p^{-1}\beta_i) \cdot \prod_{i\in I^{-,00}}(1-\beta_i^{-1})$, we have
	\begin{align*}
	\frac{1}{e!} L_p^{(e)}(0) &= \varpi_\gamma^{f} \cdot(1-p^{-1})^{f-e}\cdot \prod_{i\in I^{+,00}}\frac{1-p^{-1}\beta_i}{1-\beta_i^{-1}} \cdot (-1)^{d^+\cdot f}\cdot \left(-\frac{\#\Delta}{\varpi_\gamma}\right)^f \cdot \textrm{RHS}_{\rho,\rho^+} \\
	&=(-1)^{f+d^+\cdot f}\cdot \frac{\cE(\rho,{\rho}^+)}{\det(1-\sigma_p^{-1}\, \big|\, W_p/W_p^0)} \cdot (\#\Delta)^f \cdot \textrm{RHS}_{\rho,\rho^+}.
	\end{align*}
	Hence, we may infer that 
	$$ \textrm{LHS}_{\rho,\rho^+}=\textrm{RHS}_{\rho,\rho^+} \Longleftrightarrow \frac{1}{e!} L_p^{(e)}(0)= (-1)^e \cdot \cL(\rho,{\rho}^+)\cdot \cE(\rho,{\rho}^+)\cdot  \frac{\Reg_{{\omega}^+_{p}}(\rho)}{\Reg_{\omega_{\infty}^+}(\rho)}\cdot L^*(\rho^*,0),$$
	as claimed. \end{proof}
	 We end the proof of Theorem \ref{thm:equivalence_MRS_EZC}. If \textbf{MRS}$_\chi$ holds, then $\textrm{LHS}_{\rho,\rho^+}=\textrm{RHS}_{\rho,\rho^+}$ obviously holds for any admissible $p$-stabilization $\rho^+$ of $\rho$, so \textbf{EZC}$_{\rho,\rho^+}$ also holds by Lemma \ref{lem:LHS_RHS}. This shows the first assertion. 
	 
	 Let us prove (2). We may assume without loss of generality that the given family $\{\widetilde{t}_1,\ldots,\widetilde{t}_d\}$ is a basis of $T_p$ and that $\{\widetilde{t}_i\ |\ i\in I^0\}$ is a basis of $\HH^0(\Qp,T_p)$. By hypothesis, we know that $\{\widetilde{t}_i\ |\ i\in I^+\}$ generates an admissible $p$-stabilization of $W_p$, so the space of linear forms on $\Hom_G(W_p,\cU')$ is generated by $\{\ord_p^{(\widetilde{t}_i)}, \ |\ i\in I^0\}\bigcup\{\log_p^{(\widetilde{t}_i)}\ |\ i\in I^+\}$. Therefore, the space $\bigwedge^{d^+}_{E_p}\Hom_G(W_p,\cU')^*$ is generated by the set of all elements of the form $$\nu_{J^0,J^+}=\bigwedge_{j\in J^0}\ord_p^{(\widetilde{t}_j)}\wedge \bigwedge_{j\in J^+} \log_p^{(\widetilde{t}_{j})},$$
	 where $J^0$ (resp. $J^+$) runs over all subsets of $I^0$ (resp. of $I^+$) such that $(\#J^0)+(\#J^+)=d^+$. In order to show \textbf{MRS}$_\chi$, it is enough to prove that the equality (\ref{eq:MRS_equivalence}) holds after applying $\nu_{J^0,J^+}$ for all such couples of sets $(J^0,J^+)$. 
	 
	 Let $J^0,J^+$ as above. Assume first that $J^+\cap I^0\neq\emptyset$ and let $i\in J^+ \cap I^0$. We claim that applying $\nu_{J^0,J^+}$ to (\ref{eq:MRS_equivalence}) gives $0=0$, so it is true. The left-hand side is indeed zero because $\log_p^{(\widetilde{t}_i)}$ is a linear combination of elements in $\{\log_p^{(t_{p,j})}\ |\ 1\leq j \leq f \}$. To see that the right-hand side also gives zero, it is enough to prove that the map which sends any norm-compatible sequence $(v_n)_{n\geq 1} \varprojlim_n \widehat{\bQ}^\times_{p,n-1}$ to $\log_p(v_1)$ is identically zero. This follows directly from the exactness of the sequence (\ref{eq:exact_seq_local_norms}). 
	 
	 To complete the proof of (2), we have to prove that the equality obtained after having applied $\nu_{J^0,J^+}$ to (\ref{eq:MRS_equivalence}) holds true, where $J^0$ (resp. $J^+$) is a subset of $I^0$ (resp. of $I^+$) such that $(\#J^0)+(\#J^+)=d^+$ and that $J^+\cap I^0=\emptyset$. Note that $J^0$ and $J^+$ are disjoint, so their union defines an element $\beta\in\cB$. The key observation here is that $\nu_{J^0,J^+}$ coincides in this case with $\nu_{\rho_\beta^+}$ up to multiplication by a unit. Since we have assumed that $\rho^+_\beta$ is admissible and that \textbf{EZC}$_{\rho,\rho^+_\beta}$ holds, it follows immediately from Lemma \ref{lem:LHS_RHS} that (\ref{eq:MRS_equivalence}) holds true after applying $\nu_{J^0,J^+}$. This completes the proof of (2) and the proof of the theorem.
\end{proof}

\begin{corollaire}\label{coro:EZC_implique_MRS_cas_particulier}
	Let $W_p^+$ be any admissible $p$-stabilization of $W_p$ containing $\HH^0(\Qp,W_p)$. If \textbf{EZC}$_{\rho,\rho^+}$ holds, then \textbf{MRS}$_\chi$ holds as well.
\end{corollaire}
\begin{proof}
	The assumption $W_p^+ \supseteq \HH^0(\Qp,W_p)$ forces the inequality $f\leq d^+$, so we can take $e=0$ in Theorem \ref{thm:equivalence_MRS_EZC} (2) and assume that $\{\widetilde{t}_1,\ldots,\widetilde{t}_{d^+}\}$ is a basis of $W_p^+$. Moreover, $\cL(\rho,\rho^+)=1$ because $\HH^0(\Qp,W_p^-)=0$. The conclusion of the corollary then follows from Theorem \ref{thm:equivalence_MRS_EZC} (2).
\end{proof}

\section{Examples and applications}\label{sec:examples} 
\subsection{Deligne-critical motives}\label{sec:critical_motives}
The complex $L$-functions of motives and the $p$-adic interpolation of their special values are better understood for motives that admit critical points in the sense of Deligne. With the notations of the introduction, the Artin motives which have this property are precisely those which satisfy $d^+=d$ (the even ones) or $d^+=0$ (the odd ones). In particular, all Dirichlet motives fall in this category. There exists nowadays an extensive literature on the construction and on the properties of $p$-adic $L$-functions in this context, and we will simply recall what is needed.

Keep the notations of the introduction and assume that $\rho$ is even, that $\HH^0(\bQ,\rho)=0$ and that $\rho$ is of type $S$ at $p> 2$, which means that the extension $H/\bQ$ cut out by $\rho$ is linearly disjoint to the $\Zp$-cyclotomic extension. Note this last condition is weaker than being unramified at $p$, and if we denote by $\omega$ the Teichmüller character, then the odd representation $\rho\otimes\omega^{-1}$ is still of type $S$. There exists a $p$-adic measure $\theta_\rho^\DR\in\Lambda$ satisfying the following interpolation property:
\begin{equation}\label{eq:interpolation_deligne_ribet}
\forall\ n\geq2,\quad\forall\ \eta\in\widehat{\Gamma}, \qquad \eta\cdot\kappa^n(\theta^\DR_\rho) = L_{\{p\}}(\rho\otimes\eta\omega^{-n},1-n),
\end{equation}
where the subscript $\{p\}$ means that we removed the Euler factor at $p$. The interpolation property is first proven by the work of Deligne-Ribet \cite{deligneribet} for monomial representations and includes $n=1$, and it is proven in general by a simple application of Brauer's induction theorem (see \cite{greenbergartinII}). The fact that $\theta^\DR_\rho\in\Lambda$ and not only in $\Frac(\Lambda)$, the so-called ``$p$-adic Artin conjecture'' follows from \cite[Thm. 1.1]{wiles1990iwasawa} as a consequence of the classical Iwasawa Main conjecture. The $p$-adic $L$-function attached to $\rho$ is the $p$-adic analytic function given by $$L_p(\rho,s)=\kappa^{1-s}(\theta^\DR_\rho) \qquad\qquad (s\in\Zp).$$ 
Note that since $\rho$ is of type S, the above-mentioned Euler factor at $p$ is in fact equal to $1$ at $n=1$, provided that $\eta\neq \mathds{1}$, and the corresponding $L$-value is non-zero. Therefore, the formula (\ref{eq:interpolation_deligne_ribet}) still holds when $n=1$ and $\eta\neq \mathds{1}$.

\subsection{Even Artin motives}\label{sec:exemple_even_motives}
We expand here upon the relation between Conjecture \ref{conj:IMC}, the  classical Iwasawa Main Conjecture over totally real fields and ``$p$-adic Stark conjectures at $s=1$'', when $d^+=d$. The only choice of $p$-stabilization is then $W_p^+=W_p$, and there is no extra zeros, so $e=0$ and $\cL(\rho,\rho^+)=1$. On the algebraic side, the corresponding Selmer groups $X_\infty(\rho,\rho^+)$ coincides with the one studied by Greenberg in \cite{greenbergartinII}. It is shown in \textit{loc. cit.} that $X_\infty(\rho,\rho^+)$ is of $\Lambda$-torsion by using the validity of the weak Leopoldt conjecture, and that its characteristic ideal does not depend on the choice of $T_p$. It is also shown that $\char_\Lambda X_\infty(\rho,\rho^+)$ is generated by $\theta^\DR_\rho$ by invoking Wiles' theorem \cite[Thm. 1.3]{wiles1990iwasawa}. To make the connection with Conjecture \ref{conj:IMC}, one needs to compare $\theta^\DR_\rho$ with $\theta_{\rho,\rho^+}$, that is, to understand the values $\eta(\theta^\DR_\rho)$ for $\eta\in\widehat{\Gamma}$. It is precisely the content of ``$p$-adic Stark conjectures at $s=1$'' attributed to Serre by Tate in \cite[Chapitre VI, §5]{tate} and we refer the reader to \cite[§4.4]{johnstonnickelpadicstark} for more detail. With the notations of Conjecture 4.9 of \textit{loc. cit.}, it asserts for the representation $\rho\otimes\eta$ that 
\begin{equation}\label{eq:conj_stark_p_adique_s=1}
\eta(\theta^\DR_\rho) \overset{?}{=} \Omega_j(\rho\otimes\eta)\cdot j\left(L_{\{p\}}(\rho\otimes\eta,1)\right)=\left\{\begin{array}{lr}\Omega_j(\rho\otimes\eta)\cdot j\left(L(\rho\otimes\eta,1)\right) & \mbox{if }\eta\neq\mathds{1}\\ \Omega_j(\rho)\cdot\cE(\rho,\rho^+)\cdot j\left(L(\rho,1)\right) & \mbox{if }\eta=\mathds{1}. \end{array}\right.
\end{equation}
Moreover, \cite[Lem. 4.20]{johnstonnickelpadicstark} implies that the quantity $(-1)^d\cdot \Omega_j(\rho\otimes\eta)$ coincides with the quotient of regulators (\ref{eq:quotient_reg_appendix}) for any choice of bases $\omega_\infty^+$ and $\omega_p^+$ of $\det\HH^0(\bR,W_p)=\det W_p^+=\det W_p$ such that $\omega_\infty^+=\omega_p^+$. Therefore, under the validity of (\ref{eq:conj_stark_p_adique_s=1}) for all $\eta\in\widehat{\Gamma}$, the $p$-adic measure $\theta_{\rho,\rho^+}$ exists, it satisfies 
$$\theta_{\rho,\rho^+}=(-2)^{-d}\theta_{\rho}^\DR,$$ 
and the full Conjecture \ref{conj:IMC} holds for any choice of $T_p$. As an example, when $\rho$ is a one-dimensional even character whose associated Dirichlet character $\chi$ has prime-to-$p$ conductor $d\neq1$, for all characters $\eta\in\widehat{\Gamma}$ of conductor $p^n$ we have 
$$\Omega_j(\rho\otimes\eta) = \frac{\log_p(\varepsilon^{\chi\otimes\eta}_{\cyc})}{-j\left(\log_\infty(\varepsilon^{\chi\otimes\eta}_{\cyc})\right)},$$
where we have put 
$$\varepsilon^{\chi\otimes\eta}_{\cyc} = \prod_{a \mod dp^n} \left(e^{\frac{2i\pi a}{dp^n}}-1\right)^{(\chi\otimes\eta)^{-1}(a)} \in \left(\bZ[\mu_{dp^n}]^\times \otimes_\bZ \bQ(\chi\otimes\eta)\right)^{\chi\otimes\eta}. $$
The conjectural equality (\ref{eq:conj_stark_p_adique_s=1}) follows in this case from the well-known formula $L(\chi\otimes\eta,1)=\g(\chi^{-1}\otimes\eta^{-1})^{-1}\cdot \log_\infty(\varepsilon^{\chi\otimes\eta}_{\cyc})$ and its $p$-adic analogue (Leopoldt's formula), see for instance \cite[§1.3]{colmezfonctions}.

\subsection{Odd Artin motives}\label{sec:exemple_odd_motives}
We focus here on the connection between Conjecture \ref{conj:IMC} and the ``$p$-adic Stark conjecture at $s=0$'' (that is, the Gross-Stark conjecture) when $d^+=0$. The only choice of $p$-stabilization is then $W_p^+=0$, and both regulators are equal to $1$. The representation $\widetilde{\rho}=\rho^*\otimes\omega$ is even and of type S, and we may consider $\theta=\textrm{Tw}_{-1}(\theta_{\widetilde{\rho}}^{\DR,\iota})$, where $\iota$ is the involution of $\Lambda$ induced by $\gamma\mapsto \gamma^{-1}$ and  Tw$_{-1}$ is the twist by $\kappa^{-1}$. We thus have
$$\kappa^s(\theta)=L_p(\rho^*\otimes\omega,s),$$
for all $s\in\Zp$ and moreover, 
$$\eta(\theta)= L((\rho\otimes\eta)^*,0)$$
holds for all non-trivial characters $\eta\in\widehat{\Gamma}$ by the discussion of \S\ref{sec:critical_motives}. Therefore, one already can conclude to the existence of the $p$-adic measure $\theta_{\rho,\rho^+}$ and to the equality $\theta=\theta'_{\rho,\rho^+}$, where $\theta'_{\rho,\rho^+}$ is the renormalization of $\theta_{\rho,\rho^+}$ appearing in Proposition \ref{prop:renormalisation_mesure_p_adique}. On the algebraic side, a duality theorem for Selmer groups \cite[Thm. 2]{greenberg1989iwasawa} and Wiles' theorem for $\widetilde{\rho}$ proves that $\theta$ is a generator of the characteristic ideal of the Selmer group $X_\infty(\rho,\rho^+)$; hence the first part of Conjecture \ref{conj:IMC} is valid. Consider now the extra zeros conjecture for $(\rho,\rho^+)$. The number $e=\dim\HH^0(\Qp,W^-)=\dim\HH^0(\Qp,W)$ is nothing but the order of vanishing of the $\{p\}$-truncated Artin $L$-function $L_{\{p\}}(\rho^*,s)$ at $s=0$. The ``Weak $p$-adic Gross-Stark Conjecture'' for $\rho^*$ as formulated by Gross in \cite[Conjecture 2.12b)]{gross1981padic} states that
$$\frac{1}{e!}L_p^{(e)}(\rho^*\otimes\omega,0)\overset{?}{=}R_p(W_p^*)\cdot j\left( \frac{L^*_{\{p\}}(\rho^*,0)}{(-\log(p))^e}\right)=(-1)^e\cdot R_p(W_p^*)\cdot\cE(\rho,\rho^+)\cdot j\left(L(\rho^*,0)\right),$$
where $R_p(W_p^*)$ is Gross's $p$-adic regulator defined in \textit{loc. cit.}, (2.10) and computed with respect to the set of places $\{p,\infty\}$ of $\bQ$. By \cite[\S3.1]{maksouddimitrov}, we have $R_p(W_p^*)=\cL(\rho,\rho^+)$, so \textbf{EZC}$_{\rho,\rho^+}$ is here equivalent to the Weak $p$-adic Gross-Stark Conjecture. It has already been proven true by Dasgupta, Kakde and Ventullo in \cite{dasguptakakdeventullo} for monomial representations. For a general $\rho$, it is proven in \cite[Thm. 2.6]{burnsOnderivativesofpadicLseriesats0} that Gross-Stark's conjecture holds true under Gross' ``Order of Vanishing Conjecture'' \cite[Conj. 2.12a)]{gross1981padic} for all monomial representations cutting out the same field extension as $\rho$. This last conjecture states that $\ord_{s=0}L_p(\rho^*\otimes\omega,s)=e$, and it is known to be equivalent to the non-vanishing of $R_p(W_p^*)$ by \cite[Thm. 3.1 (i)-(iii)]{burnsOnderivativesofpadicLseriesats0}.

\subsection{Restriction to the cyclotomic line of Katz's $p$-adic $L$-function}\label{sec:exemple_katz}
It is shown in \cite{buyukboduksakamoto} that the Iwasawa-theoretic properties of Rubin-Stark elements imply $p$-adic Beilinson type formulae for Katz $p$-adic $L$-functions. We follow the notations of \textit{loc. cit.}. 

Let $k$ be a CM field of degree $2g$ and let $k^+$ be its maximal totally real subfield. Let $p$ be an odd prime which splits totally in $k^+$. Assume as in \S\ref{sec:monomial_representations} that $\rho$ is induced from a non-trivial character $\chi : \Gal(H/k) \longrightarrow E^\times$ of prime-to-$p$ order and that $T_p=\cO_p[G] \otimes_{\cO_p[\Gal(H/k)]}\cO_p(\chi)$. Hence, $d=2g$ and $d^+=g$. Under the assumption that every prime of $k^+$ above $p$ splits in $k$ (known as Katz's $p$-ordinarity condition), one may pick a $p$-adic CM type of $k$. This amounts to choosing a subset $\Sigma\subseteq S_p(k)$ such that $\Sigma$ and its complex conjugate form a partition of $S_p(k)$. To any $p$-adic CM type $\Sigma$ one can attach a (motivic) $p$-stabilization $(\rho^+_\Sigma,W_{p,\Sigma}^+)$ of $W_p$ by taking the linear span of $\{g\otimes 1\ \big|\ g\in G,\ g^{-1}(v_p)\in \Sigma\}$, where $v_p$ is the place of $k$ defined by $\iota_p$. By Shapiro's lemma, one can show that $X_\infty(\rho,\rho^+)$ is pseudo-isomorphic over $\Lambda$ to the $\chi$-isotypic component of the Galois group $\Gal(M_{\Sigma}/H\cdot \bQ_\infty)$ of the maximal pro-$p$ abelian extension of $H\cdot\bQ_\infty$ which is $\Sigma$-ramified (as in \cite[\S4.2]{maks}). In particular, it is of $\Lambda$-torsion by \cite[Thm. 1.2.2 (iii)]{hidatilouine94}. 

Consider now the Galois group $\Gamma_\infty$ of the compositum of all $\Zp$-extensions $k_\infty$ of $k$. The Iwasawa main conjecture for CM fields \cite{hidatilouine94} asserts that 
\begin{equation}\label{eq:IMC_CM_fields}
\char_{\hat{\cO}_p[[\Gamma_\infty]]} (\hat{\cO}_p \otimes\Gal(M_{\infty,\Sigma}/H\cdot k_\infty))^\chi\overset{?}{=}(L^{\iota}_{p,\chi,\Sigma}),
\end{equation}
where $\hat{\cO}_p$ is the completion of the maximal unramified extension of $\cO_p$, $M_{\infty,\Sigma}$ is the maximal pro-$p$ abelian extension of $H\cdot k_\infty$ which is $\Sigma$-ramified, and $ L_{p,\chi,\Sigma}$ is Katz's $p$-adic $L$-function for $\chi$ and $\Sigma$ \cite{katzpadicLfunctionsCMfields,hidatilouine93}. This latter is a $p$-adic measure which $p$-adically interpolates the algebraic part of critical Hecke $L$-values for $\chi^{-1}$ twisted by certain characters of $\Gamma_\infty$. The $L$-values $L((\chi\otimes\eta)^{-1},0)$ do not belong to its range of interpolation, because they are non-critical. Assuming the equality (\ref{eq:IMC_CM_fields}), one can deduce by a descent argument that
$$\char_\Lambda X_\infty(\rho,\rho_\Sigma^+)\overset{?}{=}(L^{\cyc,\iota}_{p,\chi,\Sigma}),$$
where $L^{\cyc}_{p,\chi,\Sigma}=L_{p,\chi,\Sigma}\big|_\Gamma$ is the cyclotomic restriction of $L_{p,\chi,\Sigma}$. More precisely, we apply \cite[Cor. C.10]{buyukboduksakamoto} and we claim that $X_\infty(\rho,\rho_\Sigma^+)$ has the same characteristic ideal as the extended Selmer group $\widetilde{\HH}^2_\ff(G_{k,S},\bT_\Gamma,\Delta_\Sigma)$ of \textit{loc. cit.} as follows. By Shapiro's lemma, the exact sequence (26) of \textit{loc. cit.} can be written (with the notations of \S\ref{sec:duality}) as 
\small{$$\begin{tikzcd}
	0 \rar & \coker(\Loc_+^\str) \rar & \widetilde{\HH}^2_\ff(G_{k,S},\bT_\Gamma,\Delta_\Sigma) \rar & \ker\left[\HH^2_{\Iw,\ff,p}(\bQ,\widecheck{T}_p) \longrightarrow \bigoplus_{v\in\Sigma}\HH^2_\Iw(k_v,\widecheck{T}_p)\right] \rar & 0.
	\end{tikzcd}
$$}\normalsize
Moreover, the characteristic ideal of the last term of the sequence can be computed as $\cA^e \cdot \char_\Lambda \Sha^2_\Iw(\widecheck{T}_p)$, where $\cA\subset \Lambda$ is the augmentation ideal and  $e=\#\{v|p,\ v\notin\Sigma, \ |\ \chi(v)=1\}=\dim(\Qp,W_{p,\Sigma}^-)$. Our claim then easily follows from (\ref{eq:gros_diagramme_commutatif_Sel_infty}) and Lemma \ref{lem:suite_exacte_zeros_triviaux_GV}.

The (cyclotomic part of the) explicit reciprocity conjecture for Rubin-Stark elements states that 
\begin{equation}\label{eq:explicit_rec_conj_BS}
\sC_{\omega_p^+}^\str(\varepsilon_\infty^\chi)\overset{?}{=}L^{\cyc,\iota}_{p,\Sigma},
\end{equation}
where $\sC_{\omega_p^+}^\str$ is the operator introduced in \S\ref{sec:torsionness_Sel} (see \cite[Conjecture 2]{buyukbodukmainconjecture} and \cite[Conjecture 4.8 and (21)]{buyukboduksakamoto}).
\begin{proposition}\label{prop:IMC_Katz}
	Assume as above that $p$ is an odd prime which totally splits in a CM field $k$ and that $\rho\simeq \Ind_k^\bQ \chi$ for some non-trivial character $\chi$ of prime-to-$p$ order and conductor which satisfies Conjectures \ref{conj:rubinstark_alg} and \ref{conj:rubinstark_p_int}. If (\ref{eq:IMC_CM_fields}) and (\ref{eq:explicit_rec_conj_BS}) both hold for some $p$-adic CM type $\Sigma$, then \textbf{IMC}$_{\chi}$ and \textbf{IMC}$_{\rho,\rho^+}$ hold for any choice of $p$-stabilization $\rho^+$ of $\rho$. 
	
	In particular, when $k$ is an imaginary quadratic field whose class number is prime to $p$, \textbf{IMC}$_\chi$ and \textbf{IMC}$_{\rho,\rho^+}$ are valid for any $p$-stabilization $\rho^+$ of $\rho$. 
	\end{proposition}
\begin{proof}
 We know that $\sC_{\omega_p^+}^\str(\varepsilon_\infty^\chi)=\theta_{\rho,\rho_\Sigma^+}'=U\cdot \theta_{\rho,\rho_\Sigma^+}$ for some unit $U\in\Lambda^\times$ by Theorem \ref{thm:comparaison_mesure_p_adique_RS_elts} and Proposition \ref{prop:renormalisation_mesure_p_adique}. Therefore, \textbf{IMC}$_{\rho,\rho_\Sigma^+}$ is a consequence of (\ref{eq:IMC_CM_fields}) and (\ref{eq:explicit_rec_conj_BS}). As $X_\infty(\rho,\rho_\Sigma^+)$ is of $\Lambda$-torsion, we deduce from Theorem \ref{thm:intro_C} (2) that \textbf{IMC}$_{\chi}$ and \textbf{IMC}$_{\rho,\rho^+}$ hold for any choice of $\rho^+$. 
 
  When $k$ is an imaginary quadratic field whose class number is prime to $p$, both (\ref{eq:IMC_CM_fields}) and (\ref{eq:explicit_rec_conj_BS}) hold for any choice of $\Sigma$ by the work of Rubin \cite{rubin1991main} and Yager \cite{yager}. Furthermore, Conjectures \ref{conj:rubinstark_alg} and \ref{conj:rubinstark_p_int} hold for $\chi$, as explained in the proof of Proposition \ref{prop:iqf_inert_and_split} below. So the last claim follows from the first one.
\end{proof}
\begin{remark}
	The admissibility of $W_{p,\Sigma}^+$ is equivalent to the $\Sigma$-Leopoldt conjecture which appear for instance in Hida-Tilouine's work. Moreover,  $\cL(\rho,\rho^+_{\Sigma})$ is equal to the cyclotomic $\cL$-invariant defined in \cite[Def. 5.2]{buyukboduksakamoto}; accordingly, the implication \textbf{MRS}$_\chi \Longrightarrow$ \textbf{EZC}$_{\rho,\rho^+_\Sigma}$  of Theorem \ref{thm:intro_C} can be seen as a reformulation of \cite[Thm. 1.5]{buyukboduksakamoto} in the cyclotomic direction. A similar result was obtained by Betina and Dimitrov for anticyclotomic Katz $p$-adic $L$-functions \cite{betinadimitrovKatz}. Finally, it is worth pointing out that, when $[k:\bQ]=2g>2$, there are $2^g$ distinct $p$-adic CM types, whereas ${\bigwedge}^{d^+}W_p$ has dimension $(2g)!(g!)^{-2}>2^g$, so Conjecture \ref{conj:IMC} predicts in particular the existence of more $p$-adic $L$-functions for $\chi$ than the ones constructed by Katz. 
\end{remark}

\subsection{Two-dimensional odd monomial representations}\label{sec:example_odd_monomial}
Fix an imaginary quadratic field $k$ and an odd prime $p$ which is unramified in $k$. The prime $p$ is allowed to be split or inert in the following proposition.
\begin{proposition}\label{prop:iqf_inert_and_split}
Let $\chi:G_k\longrightarrow E^\times$ be a non-trivial character of prime-to-$p$ order and conductor, and consider $\rho=\Ind_k^\bQ \chi$. Then \textbf{EX}$_{\rho,\rho^+}$ and \textbf{EZC}$_{\rho,\rho^+}$ are valid for all $p$-stabilizations $\rho^+$ of $\rho$.
\end{proposition}
\begin{proof}
We first recall why Conjectures \ref{conj:rubinstark_alg} and \ref{conj:rubinstark_p_int} hold for $\chi$. The Rubin-Stark conjecture over $\bZ$ for abelian extensions of $k$ in the rank $1$ scenario is a consequence of Stark's work (see, for instance, \cite[Chap. IV, Prop. 3.9]{tate}). In particular, the statements concerning $\varepsilon_n^\chi$ for $n\geq 1$ in Conjectures \ref{conj:rubinstark_alg} and \ref{conj:rubinstark_p_int} hold. One can draw the same conclusion for $\xi^\chi$ defined in \eqref{eq:xi^chi} but \emph{a priori} not for $u^\chi$, as we may have $d^++f=1+f>1$. 

Let $L/k$ be the abelian extension cut out by $\chi$ and let $g=[L:k]$. By \cite[Coro. 5.4]{rubinstark}, Rubin-Stark's conjecture over $\bZ[\frac{1}{g}]$ for $u^\chi$ follows from that for $\xi^\chi$ and from a Gras-type conjecture, which is proven in \cite[Thm. 1]{rubinmore} (the missing case $(k,p)=(\bQ(\sqrt{-3}),3)$ where $p|\#\cO_k^\times$ is ruled out by our assumptions). Since we assumed $p\nmid g$, we may infer that the statements for $u^\chi$ in Conjectures \ref{conj:rubinstark_alg} and \ref{conj:rubinstark_p_int} hold as well.
	
We can now apply Theorem \ref{thm:intro_C} and deduce that \textbf{EX}$_{\rho,\rho^+}$ holds for every $\rho^+$. \textbf{EZC}$_{\rho,\rho^+}$ follows likewise from \textbf{MRS}$_\chi$, which is proven in \cite[Thm. A]{bullachhofer} (and whose proof is built on earlier works \cite{bleyeTNC,bleyhofer}).
\end{proof}

We now give a concrete description of the $p$-stabilizations, $p$-adic regulators and $\cL$-invariants attached to $\rho$ when $\rho\simeq \Ind_k^\bQ \chi$ for some non-trivial character $\chi$ which we fix once and for all. We don't necessarily assume that $\chi$ has order prime to $p$. Write $p\cO_k$ as $\gp$ when $p$ remains inert in $k$, and as $\gp\ob{\gp}$ (where $\gp$ is the prime determined by the embedding $\iota_p$) when $p$ splits in $k$. Note that $\sigma_p\in G_k$ if and only if $p$ splits in $k$. We denote by $\tau$ both the complex conjugation and the non-trivial element of $\Gal(k/\bQ)$, and we let $\chi^\tau$ be the character $\chi(\tau \cdot \tau)$ defined over $k$. We assume for simplicity that $\chi\neq\chi^\tau$, so that $\rho$ is irreducible. This assumption is not restrictive, as the case $\rho=\chi^+\oplus\chi^-$ (for some Dirichlet characters $\chi^\pm$) can be derived from \S\ref{sec:exemple_even_motives}-\ref{sec:exemple_odd_motives} and the Artin formalism (Remark \ref{rem:p_adic_artin_formalism}). 

There exists a unique basis $\{w_1,w_2\}$ of $W$ (up to scaling) in which $\rho$ is of the form:
$$\rho(g)\sim \begin{pmatrix}\chi(g) & 0 \\ 0 & \chi^\tau(g)\end{pmatrix},\qquad \rho(\tau)\sim  \begin{pmatrix}0 & 1 \\ 1 & 0\end{pmatrix},$$
where $g\in G_k$. The eigenvalues of $\sigma_p$ are $\chi(\gp)$ and $\chi(\ob{\gp})$ in the split case, and $\pm \sqrt{\chi(\gp)}$ in the inert case respectively. A natural eigenbasis for $\sigma_p$ is $\{w_1,w_2\}$ when $p$ splits in $k$, and $\{-aw_1+ w_2,aw_1+w_2\}$ (where $a=\pm \sqrt{\chi(\gp)}\chi(\sigma_p^{-1}\tau)$) when $p$ is inert in $k$. Therefore, there are exactly two (motivic) choices of $p$-stabilization $\rho^+$ of $\rho$, except in the split case when $\chi(\gp)=\chi(\ob{\gp})$, where any line is a $p$-stabilization. 

Extending some of the ideas present in \cite[§1]{betinadimitrovKatz}, we now express $p$-adic regulators and $\cL$-invariants in terms of Minkowski units. Given a character $\eta\in\widehat{\Gamma}$, let $H_\eta/\bQ$ be the Galois extension cut out by $\rho\otimes\eta$, let $E_\eta$ be its coefficient field and put $C_\eta=\Gal(H_\eta/k)$. We drop the subscript in $H_\eta,E_\eta,C_\eta$ when $\eta=\mathds{1}$. By Minkowski's theorem, there exists a unit $u\in \cO_{H_\eta}^\times$ such that $\{g(u)\ |\ 1\neq g\in C_\eta\}$ generates a sublattice of $\cO_{H_\eta}^\times$ of finite index. Moreover, one can impose $u$ to be fixed by $\tau$, so that the only relation between elements of $\{g(u)\ |\ g\in C_\eta\}$ is given by the kernel of the norm map of the extension ${H_\eta/k}$. 

 Shapiro's lemma provide us with an isomorphism
\begin{equation}\label{eq:isom_shapiro_CM}
	\begin{array}{rcl}
	\HH^1_\ff((\rho\otimes\eta)^*(1)) & \overset{\simeq}{\longrightarrow} & (E_\eta\otimes \cO_{H_\eta}^\times)^{\chi\eta} \\
	\psi & \mapsto & \psi(w_1),
\end{array}
\end{equation}
 where  $({E_\eta}\otimes \cO_{H_\eta}^\times)^{\chi\eta}$ is the $\chi\eta$-isotypic component of the group of units of $H_\eta$. This latter space is generated by the element (written in additive notation)
\begin{equation}\label{eq:u_chi}
	u_{\chi\eta}:= \sum_{g\in C_\eta} \chi\eta(g) \otimes g^{-1}(u),
\end{equation}
and its complex conjugate $\tau(u_{\chi\eta})$ generates $(E_\eta\otimes \cO_{H_\eta}^\times)^{\chi^\tau\eta}$.

We let $\omega_{\ff,\eta}$ be the basis of $\HH^1_\ff((\rho\otimes\eta)^*(1))$ corresponding to $u_{\chi\eta}$ under \eqref{eq:isom_shapiro_CM}. Given a $p$-stabilization $\rho^+$ generated by an element $\omega_p^+=w_1+sw_2$ (where $s\in \bP^1(\ob{\bQ}_p)$), we then have 
$$\Reg_{{\omega}^+_{p}}(\rho\otimes\eta)=\log_p(u_{\chi\eta})+s\log_p(\tau(u_{\chi\eta})).$$ 
As $\rho$ is assumed to be irreducible, we have $\chi\eta\neq\chi^\tau\eta$, so $\log_p(u_{\chi\eta})$ and $\log_p(\tau(u_{\chi\eta}))$ are $\ob{\bQ}$-linearly independent by the Baker-Brumer theorem. In particular, $\rho^+$ is always $\eta$-admissible, except when $\chi(\gp)=\chi(\ob{\gp})$ and $s$ equals 
\[\sS(\chi\eta):=-\frac{\log_p(u_{\chi\eta})}{\log_p(\tau(u_{\chi\eta}))} \in \ob{\bQ}_p\setminus\ob{\bQ}. \]

Let us describe the $\cL$-invariant attached to the admissible $p$-stabilization $\rho^+$ given by the line $\ob{\bQ}_p\cdot w_1$ in the split case, and by $\ob{\bQ}_p\cdot (-aw_1+w_2)$ in the inert case. Extra zeros for $\rho^+$ only arise when $\chi(\ob{\gp})=1$ in the split case (resp. when $\chi({\gp})=1$ and $a=\chi(\sigma_p^{-1}\tau)$ in the inert case), which we assume. Denote by $v_0$ the place of $H$ determined by $\iota_p$, let $y_0$ be any element of $\cO_H[\tfrac{1}{v_0}]^\times$ with $\ord_p(y_0)\neq 0$ and put
\begin{equation}\label{eq:u_p_chi}
	u_{\gp,\chi}= \sum_{g\in C} \chi(g) \otimes g^{-1}(y_0),  \quad \textrm{and }\  u_{\ob{\gp},\chi}=\sum_{g\in C} \chi(g) \otimes g^{-1}\cdot\tau(y_0)\quad \textrm{ (only when $p$ splits in $k$).}
\end{equation}
Under our hypotheses, \eqref{eq:u_p_chi} defines a non-trivial $\gp$-unit (resp. a $\ob{\gp}$-unit) in $H$. Moreover, a simple computation based on Lemma \ref{lem:formule_L_invariant_quot_det} gives
\begin{align*}
\ord_p(\tau(u_{\ob{\gp},\chi}))\cdot\Reg_{{\omega}^+_{p}}(\rho)\cdot\cL(\rho,\rho^+) =\det\begin{pmatrix}\log_p(u_\chi) & \log_p(u_{\ob{\gp},\chi}) \\ \log_p(\tau(u_{\chi})) & \log_p(\tau(u_{\ob{\gp},\chi}))\end{pmatrix} \qquad & \textrm{when $p$ splits in $k$,} \\ \ord_p(u_{\gp,\chi})\cdot\Reg_{{\omega}^+_{p}}(\rho)\cdot\cL(\rho,\rho^+) =-2\cdot \det\begin{pmatrix}\log_p(u_\chi) & \log_p(u_{{\gp},\chi}) \\ \log_p(\tau(u_{\chi})) & \log_p(\tau(u_{{\gp},\chi}))\end{pmatrix} \qquad & \textrm{when $p$ in inert in $k$.}
\end{align*}
Note that replacing $\rho^+$ by $\ob{\bQ}_p\cdot w_2$ in the split case and assuming $\chi(\gp)=1$ would give the same formula as above, with a minus sign and with $u_{\ob{\gp},\chi}$ replaced by $u_{{\gp},\chi}$ in the determinant.

We record in the next proposition the interpolation formulae satisfied by the two $p$-adic $L$-functions attached to $\chi$ when $p$ is inert in $k$. 

\begin{corollaire}\label{coro:formule_theta_p_inerte}
	Let $\chi:G_k\longrightarrow E^\times$ be a character of prime-to-$p$ order and conductor of an imaginary quadratic field $k$ in which $p$ is inert. Assume $\chi\neq \chi^\tau$. Let $\alpha\in E^\times$ be a square root of $\chi(\gp)$ and put $a=\alpha\cdot\chi(\sigma_p^{-1}\tau)$. There exists a measure $\theta^\alpha_\chi\in\Lambda$ with the following interpolation property. For all $\eta\in\widehat{\Gamma}$ of conductor $p^n$, we have
	\[\eta(\theta^\alpha_\chi)= \left(1+\alpha\frac{\eta(p)}{p}\right)\left(1-\frac{\eta(p)}{\alpha}\right)\frac{(-\alpha)^n}{\g(\eta^{-1})}(-a\log_p(u_{\chi\eta})+\log_p(\tau(u_{\chi\eta}))\cdot j\left(\frac{L'((\chi\eta)^{-1},0)}{\log_\infty(u_{\chi\eta}\cdot \tau(u_{\chi\eta}))}\right), \]
	where $u_{\chi\eta}$ is defined as in \eqref{eq:u_chi}. Moreover, if $\alpha=1$, then $\mathds{1}(\theta^\alpha_\chi)=0$ and
	\[\lim_{s\rightarrow 0}\frac{\kappa^s(\theta^\alpha_\chi)}{s}=\left(1+\frac{1}{p}\right)\frac{2}{\ord_p(u_{\gp,\chi})}\det\begin{pmatrix}\log_p(u_\chi) & \log_p(u_{{\gp},\chi}) \\ \log_p(\tau(u_{\chi})) & \log_p(\tau(u_{{\gp},\chi}))\end{pmatrix}\cdot j\left(\frac{L'(\chi^{-1},0)}{\log_\infty(u_{\chi}\cdot \tau(u_{\chi}))}\right), \]
	where $u_{{\gp},\chi}$ is the $\gp$-unit defined in \eqref{eq:u_p_chi}.
\end{corollaire}
\begin{proof}
	We apply Proposition \ref{prop:iqf_inert_and_split} to $\rho^+$ given by the line generated by $-aw_1+w_2$ and we put $\theta^\alpha_\chi=\theta'_{\rho,\rho^+}$ (cf. Proposition \ref{prop:renormalisation_mesure_p_adique}). Then $\sigma_p$ acts on $W_p^+$ by $-\alpha$ and on $W_p^-$ by $\alpha$. In particular, we have $e=0$ and $\cE(\rho,\rho^+)=(1+\alpha p^{-1})(1-\alpha^{-1})$, unless $\alpha=1$, in which case $e=1$ and $\cE(\rho,\rho^+)=(1+\alpha p^{-1})$. To compute $\Reg_{\omega_{\infty}^+}(\rho\otimes\eta)$ we use the basis $w_1+w_2$ of $\HH^0(\bR,W)$, and the basis $u_{\chi\eta}$ of $(E_\eta\otimes \cO_{H_\eta}^\times)^{\chi\eta}$, which yields $\Reg_{\omega_{\infty}^+}(\rho\otimes\eta)=\log_\infty(u_{\chi\eta}\cdot \tau(u_{\chi\eta}))$. The other terms in the interpolation formula for $\theta^\alpha_\chi$ have already been computed. 
\end{proof}
 
\begin{proposition}\label{prop:existence_non_vanishing_L_invariant}
	Let $\rho=\Ind_k^\bQ \chi$ with $\chi\neq\mathds{1}$. There exists at least one admissible $p$-stabilization $\rho^+$ of $\rho$ such that $\cL(\rho,\rho^+)\neq 0$. Moreover, if $\chi$ is of prime-to-$p$ order, then the weak exceptional zero conjecture \textbf{wEZC}$_\chi$ (Conjecture \ref{conj:EZC_for_RS_elements}) for $\chi$ is true.
\end{proposition}
 \begin{proof}
 The second statement follows from the first one and from Theorem \ref{thm:comparaison_mesure_p_adique_RS_elts} (2). Assume first that $\chi=\chi^\tau$. Then $\rho$ is the sum of an even and an odd Dirichlet character $\chi^+$ and $\chi^-$, and $\rho^+=\chi^+$ defines an admissible $p$-stabilization, the $p$-adic regulator essentially being $\log_p(\varepsilon_{\cyc}^{\chi^+})$ with the notations of \S\ref{sec:exemple_even_motives}. It then follows from the $p$-adic Artin formalism (Remark \ref{rem:p_adic_artin_formalism}) and from \S\ref{sec:exemple_odd_motives} that $\cL(\rho,\rho^+)$ is Gross's regulator attached to $\chi^-$, so it does not vanish. 
 
 Assume now that $\chi\neq \chi^\tau$. If $f=\dim \HH^0(\Qp,\rho)=\#\{\mathfrak{q}|p\ | \ \chi(\mathfrak{q})=1\}\leq 1$, then there is always a choice of admissible $\rho^+$ for which $e=\dim \HH^0(\Qp,\rho^-)=0$. For such a $\rho^+$, there is no extra zero and $\cL(\rho,\rho^+)=1$. Hence, we may assume that $f=2$, \textit{i.e.}, $p$ splits in $k$ and $\chi(\gp)=\chi(\ob{\gp})=1$. Since $H$ is the Galois closure of the field cut out by $\chi$, $p$ must split completely in $H$. We show that at least one of the $p$-stabilizations $\rho^+=\ob{\bQ}_p\cdot w_1$ or $\rho^+=\ob{\bQ}_p\cdot w_2$ has a non-zero $\cL$-invariant. With the notations introduced above, it is enough to show that the matrix 
 $$ M=\begin{pmatrix} \log_p(u_\chi) & \log_p(u_{\gp,\chi}) & \log_p(u_{\ob{\gp},\chi}) \\ \log_p(\tau(u_{\chi})) & \log_p(\tau(u_{{\gp},\chi})) & \log_p(\tau(u_{\ob{\gp},\chi})) \end{pmatrix}$$
 has rank 2. This follows from a straightforward adaptation of the proof of \cite[Prop. 1.11]{betinadimitrovKatz} to the cyclotomic setting, which we explain for the sake of completeness. 
 
 By Roy's Six Exponentials Theorem \cite[Cor. 2]{roy}, we only have to show that $M$ has $\bQ$-linearly independent rows and $\bQ$-linearly independent columns. The two rows are linearly independent because $\sS(\chi)$ is transcendental. Since $\log_p(\textrm{N}_{H/k}(u))=\log_p(\textrm{N}_{H/\bQ}(y_0))=0$, any linear equation $a \cdot \log_p(u_\chi) + b\cdot \log_p(u_{{\gp},\chi}) +c \cdot \log_p(u_{\ob{\gp},\chi})$ with $a,b,c\in \bZ$ can be written as
 $$\sum_{1\neq g\in C}(\chi(g)-1)(a\cdot \log_p(g^{-1}(u))+b\cdot \log_p(y_0)) + \sum_{g\in C}(c\cdot\chi(g)-b)\log_p(g^{-1}\tau(y_0))=0.$$
In \textit{loc. cit.}, the same equation appears, with $+b$ instead of the $-b$ in the second sum. We then may repeat \textit{verbatim} their arguments to show that $a=b=c=0$, as wanted.
 \end{proof}
\begin{remark}
	Betina and Dimitrov proved in \cite[Proposition 1.11]{betinadimitrovKatz} an analogue of Proposition \ref{prop:existence_non_vanishing_L_invariant} where $\cL(\rho,\rho^+)$ is replaced by an \emph{anticyclotomic} $\cL$-invariant.
\end{remark}
 
 \begin{theorem}\label{thm:gross_kuzmin}
 	The Gross-Kuz'min conjecture holds for abelian extensions $L$ of $k$ and all odd primes $p$, that is, the module of $\Gamma$-coinvariants of 
 	$$A'_\infty=\varprojlim_{L\subseteq M \subseteq L\cdot \bQ_\infty}A'(M)$$
 	is finite, where $A'(M)$ is the $p$-Sylow of the $p$-split class group of $M$ and  the transition maps are the norm maps.
 \end{theorem}
\begin{proof}
Let $L$ be an abelian extension of $k$ of Galois group $C$. By \cite[p.3, Remarque]{jaulent2003}, we may replace $L$ by its Galois closure over $\bQ$ without loss of generality. Then $A'_\infty$ is a module over $\Zp[[\Gamma]]$ which carries a linear action of $C$, so it is enough to show that the $\chi$-part $A'_\infty$ has finite $\Gamma$-coinvariants for all multiplicative characters $\chi$ of $C$. When $\chi=\mathds{1}$, this follows from the validity of the Gross-Kuz'min conjecture for $k$ which is implicitly proven in \cite{greenberg1973}. When $\chi\neq \mathds{1}$ and when $p$ is unramified in $L$, this is precisely the content of Conjecture \ref{conj:gross_finiteness_conj} for $\rho=\Ind_k^\bQ \chi$, which is valid by Theorem \ref{thm:exact_order_of_vanishing} (2) and by Proposition \ref{prop:existence_non_vanishing_L_invariant}. 

We now indicate how to proceed when $p$ is ramified in $L$. This case is easier than the previous one because of the absence of trivial zeros. Let $\{1\}\neq G_p \subseteq G$ be the decomposition group at the place above $p$ defined by $\iota_p$. By \cite[Preuve du corollaire]{jaulent2003}, it is enough to prove the Gross-Kuz'min conjecture for a well-chosen subfield $L'$ of $L$ such that all the $p$-adic primes of $L'$ are non-split in $L$. In particular, we may reduce the proof to the case where $G_p$ is abelian. Indeed, if there is a unique $p$-adic place in $k$, then it suffices to consider $L'=L^{C\cap G_p}$. Take a faithful ($2$-dimensional) $p$-adic representation $(\rho,W_p)$ of $G$ and fix $T_p\subseteq W_p$ a $G$-stable lattice. Since $G_p\neq \{1\}$ is abelian, we can choose a $G_p$-stable line $W_p^+$ of $W_p$ such that $G_p$ acts \emph{non-trivially} on the quotient $W_p^-=W_p/W_p^+$ and we can form a Selmer group $\Sel^\str_n(\rho,\rho^+)$ for any $n\geq 0$ as in Definition \ref{def:sel}. Since $\HH^0(\Qp,W_p^-)=0$, the module $\HH^1(\Gamma,\HH^0(\bQ_{p,\infty},D_p^-))$ is finite, and it is not hard to see that the restriction map $\Sel^\str_0(\rho,\rho^+) \longrightarrow \Sel^\str_\infty(\rho,\rho^+)^\Gamma$ has a finite kernel and cokernel (see for instance the argument in \cite[Lemme 2.2.2]{maks}). On the other hand, $\Sel^\str_0(\rho,\rho^+)$ is finite by \cite[Thm. 2.4.7]{maks}, so the same conclusion holds for $\Sel^\str_\infty(\rho,\rho^+)^\Gamma$. Thus $\Hom_G(T_p,A_\infty')_\Gamma$ is finite by Lemma \ref{lem:relation_cohomologie_et_unités} and Diagram (\ref{eq:gros_diagramme_commutatif_Sel_infty}).
\end{proof}
 
\subsection{Weight one modular forms}\label{sec:weight_one}
We consider in this paragraph $E$-valued Artin representations $(\rho,W)$ that are two-dimensional and irreducible with $d^+=1$. The proof by Khare and Wintenberger of Serre's modularity conjecture implies that $\bC \otimes_E W$ is the Deligne-Serre representation of a cuspidal newform $f(q)=\sum_{n\geq 1}a_nq^n\in E[[q]]$ of weight 1 \cite[Cor. 10.2 (ii)]{kharewintenberger}. It has level equal to the Artin conductor $N$ of $\rho$, and for all primes $\ell \nmid N$, we have 
\[ \textrm{tr}(\rho(\sigma_\ell))=a_\ell, \]
where $\sigma_\ell$ is the Frobenius substitution at $\ell$ (see \cite{deligne1974formes}). Some aspects of Iwasawa theory for $f$ and an odd prime $p$ were studied in \cite{maks} and \cite{GV2020} under the hypothesis that $p\nmid N$ (\textit{i.e.}, $\rho$ is unramified at $p$) and that $f$ is regular at $p$. The last assumption means that
\[\alpha\neq \beta,\]
where $\alpha$ and $\beta$ are the eigenvalues of $\sigma_p$. We propose in this paragraph a new conjectural description of the $p$-adic $L$-function of $f$ that is considered in \textit{loc. cit.}. We therefore assume that $p\nmid N$ and $\alpha\neq \beta$. 

As $\sigma_p$ has distinct eigenvalues, there exist exactly two choices of $p$-stabilizations of $\rho$, namely the two lines $W_\alpha=W_p^{\sigma_p=\alpha}$ and $W_\beta=W_p^{\sigma_p=\beta}$. Both $p$-stabilizations are motivic, hence $\eta$-admissible for all $\eta\in\widehat{\Gamma}$ by Lemma \ref{lem:schanuel+motivic_implies_admissible}. 


Let us consider $W_p^+=W_\beta$ (so $W_p^-\simeq W_\alpha$), that we let correspond to the $p$-stabilization $f_\alpha(q)=f(q)-\beta f(q^p)$ of $f$. Then $f_\alpha$ is a normalized eigen cuspform of level $Np$ sharing the same eigenvalues with $f$ under the action of Hecke operators of prime-to-$p$ level, and moreover, $U_p f_\alpha = \alpha f_\alpha$. A theorem of Hida and Wiles asserts that $f_\alpha$ belongs to a primitive Hida family $\cF$ of tame level $N$. The uniqueness of $\cF$ up to Galois conjugacy is proved by Bellaïche and Dimitrov in \cite{bellaiche2016eigencurve}. The existence of a local parameterization of $\cF$ around $f_\alpha$ proven in \cite[Cor. 4.7.3]{maks} yields, in particular, the following result. 

\begin{proposition}
There exists a finite flat extension $\cO$ of $\Zp$ containing the Fourier coefficients of $f_\alpha$ and, for any integer $m>>0$, a $p$-ordinary newform $g_m(q)=\sum_{n\geq 1}a_n(g_m)\in\cO[[q]]$ of weight $(p-1)p^m+1$ and level $N$, such that 
\[\lim_{m\to +\infty} g_{m,\alpha_m}(q) = f_\alpha(q),\]
where $g_{m,\alpha_m}(q)$ is the unique $p$-ordinary $p$-stabilization of $g_m$. In particular, for all integers $n\geq 1$ coprime with $p$, we have  $\lim_{m\to +\infty} a_n(g_m) = a_n$ in $\cO$.

\end{proposition}

Assume henceforth that $p$ does not divide the order of the finite group $\rho(G_\bQ)$. The residual representation $\ob{\rho}$ is then irreducible and $p$-distiguished, and $W_p$ has a unique Galois-stable $\cO$-lattice $T_p$ (up to homothety). One can attach to $\cF$ a two-variable $p$-adic $L$-function à la Mazur-Kitagawa, and the choice of Shimura's periods for each member of the family $\cF$ can be made in a canonical way by the results of Emerton, Pollack and Weston (see \cite[\S3.3-4]{emerton2006variation}). We obtain, for all integers $m>>0$, a complex period $\Omega_m^+\in\bC^\times$ and a $\cO$-valued measure $\theta_m$ on $\Gamma$ such that, for all characters $\eta \in\widehat{\Gamma}$ of conductor $p^n$,

\begin{equation}\label{eq:formule_theta_m}
\eta(\theta_m)=\left(\frac{p}{\alpha_m}\right)^{n} \left(1-\frac{\eta(p)}{\alpha_m}\right)\left(1-\frac{\eta(p)\beta_m}{p}\right)\cdot j\left(\frac{L(g_m\otimes \eta,1)}{\g(\eta)(-2i\pi)\Omega_m^+}\right),
\end{equation}
where $\alpha_m$ and $\beta_m$ are the roots of the $p$-th Hecke polynomial of $g_m$ with $\alpha_m\in\cO^\times$ (see \cite[Prop. 4.4.1]{maks}).

The construction of Emerton-Pollack-Weston also provides us with a $p$-adic measure $\theta_{f,\alpha}\in\Lambda=\cO[[\Gamma]]$ which is defined as the specialization at $f_\alpha$ of the $p$-adic measure attached to $\cF$. Changing the choice of the canonical periods in $\cF$ results in multiplying $\theta_{f,\alpha}$ by a unit in $\cO$. By \cite[Lemme 4.8.1]{maks}, we have 
\begin{equation}\label{eq:lim_theta_m}
	\lim_{m\to +\infty}\theta_m=\theta_{f,\alpha}
\end{equation}
in $\Lambda$ (endowed with its topology of local ring). 

\begin{lemme}
	Let $\omega_p^+$ (resp. $\omega_{\infty}^+$) be a $T_p$-optimal basis of $W_p^+$ (resp. of $W^{\sigma_\infty=1}$) and assume that \textbf{EX}\(_{\rho,\rho^+}\) and \textbf{EZC}\(_{\rho,\rho^+}\) hold for the $p$-stabilization $(\rho^+,W_p^+)$ given by $W_p^+=W_\beta$ as above. For all characters $\eta \in\widehat{\Gamma}$ of conductor $p^n$, we have
	\begin{equation}\label{eq:formule_theta_f_alpha}
		\eta(\theta_{\rho,\rho^+})=\left(\frac{p}{\alpha}\right)^{n} \left(1-\frac{\eta(p)}{\alpha}\right)\left(1-\frac{\eta(p)\beta}{p}\right) \Reg_{\omega_p^+}(\rho\otimes\eta)\cdot j\left(\frac{L(f\otimes \eta,1)}{\g(\eta)(-2i\pi)\Reg_{\omega_\infty^+}(\rho\otimes\eta)}\right).
	\end{equation}
\end{lemme}
\begin{proof}
	Notice first that $d^+=d^-=1$ and $\sigma_p$ acts by multiplication by $\alpha$ on $W_p^-$, so $\det(\rho^-)(\sigma_p)=\alpha$. Let $\eta\in\widehat{\Gamma}$. As $\eta$ is even, we have $\tau(\eta)=\g(\eta^{-1})=p^n\g(\eta)^{-1}$ by Lemma \ref{lem:properties_galois_gauss_sums} (1). The equality \eqref{eq:formule_theta_f_alpha} then follows from Lemma \ref{lem:formule_interpolation_en_s=1} if $\eta\neq \mathds{1}$, as $\eta(p)=0$. In the case where $\eta=\mathds{1}$, \eqref{eq:formule_theta_f_alpha} follows from the same lemma, after noticing that $e=\dim \HH^0(\Qp,W_p^-)=0$ if and only if $\alpha\neq 1$, in which case we obtain $\cL(\rho,\rho^+)=1$ and $\cE(\rho,\rho^+)=\left(1-\frac{1}{\alpha}\right)\left(1-\frac{\beta}{p}\right)$ from the definition.
\end{proof}

 The interpolation formulae \eqref{eq:formule_theta_m} and \eqref{eq:formule_theta_f_alpha} are very similar, except that there is an extra $p$-adic regulator in the expression of $\eta(\theta_{\rho,\rho^+})$. This discrepancy comes from the fact that the Selmer groups $\HH^1_\ff(g_m\otimes \eta)$ and  $\HH^1_\ff((g_m\otimes \eta)^*(1))$ of the motive attached to $g_m\otimes \eta$ and its dual vanish by the work of Kato (\cite[Thm. 14.2 (1)]{kato2004p}), unlike with $\HH^1_\ff((\rho\otimes \eta)^*(1))$. 
 
As $\lim_{m\to +\infty} \eta(\theta_m)=\eta(\theta_{f,\alpha})$ by \eqref{eq:lim_theta_m}, it is tempting to raise the following question.

\begin{question}\label{question}
	Under the running assumptions $p\nmid N\cdot(\sharp \rho(G_{\bQ}))$ and $\alpha\neq \beta$, is it true that 
	\begin{equation}\label{eq:question}
		\theta_{f,\alpha}=u\cdot \theta_{\rho,\rho^+}
	\end{equation}
	for some unit $u\in \cO^\times$?
\end{question}
In other terms, the construction of $\theta_{f,\alpha}$ yields a potential definition of $\theta_{\rho,\rho^+}$. Note that Question \ref{question} makes sense. Indeed, both $\theta_{f,\alpha}$ and $\theta_{\rho,\rho^+}$ are well-defined up to a unit in $\cO^\times$, as they depend on the choice of $p$-normalized periods.

\begin{proposition}
	Assume Conjecture \ref{conj:IMC} for $(\rho,\rho^+)$. There exists an element $U \in \Lambda \otimes \Qp$ such that 
	\[\theta_{f,\alpha}=U\cdot \theta_{\rho,\rho^+}.\]
\end{proposition}
\begin{proof}
	We know from \cite[Thm. C]{maks} that $\theta_{f,\alpha}$ belongs to $\char_{\Lambda\otimes \Qp} X_\infty(\rho,\rho^+) \otimes \Qp$. The latter is generated by $\theta_{\rho,\rho^+}$ over $\Lambda \otimes \Qp$ under \textbf{IMC}$_{\rho,\rho^+}$. Therefore, there exists $U \in \Lambda \otimes \Qp$ such that \(\theta_{f,\alpha}=U\cdot \theta_{\rho,\rho^+}\), as wanted.
\end{proof}

\begin{remark}
	Assuming that $f$ is a dihedral form, Ferrara made a conjecture in \cite{ferraraarticle} which is equivalent to a weak version of \eqref{eq:question} where $u$ is taken in $\Frac(\cO)^\times$. When $f$ has CM by an imaginary quadratic field $k$ in which $p$ splits, he compared $\theta_m$ and $\theta_{f,\alpha}$ with Katz's $p$-adic $L$-function and deduced from a limit argument that his conjecture holds. In this special case, Conjecture \ref{conj:IMC} holds by Propositions \ref{prop:IMC_Katz} and \ref{prop:iqf_inert_and_split} under the technical assumption $p\nmid h_k$.
\end{remark}

\subsection{Adjoint of a weight one modular form}\label{sec:examples_adjoint}
We explore the connection between the ``variant of Gross-Stark conjecture'' formulated by Darmon, Lauder and Rotger in \cite{DLR2} and \textbf{EZC}$_{\rho,\rho^+}$ for $\rho=\rho_g\otimes \rho_h$, where $\rho_g,\rho_h$ are the Deligne-Serre representations attached to weight one newforms $g,h$ (see \S\ref{sec:weight_one}). We restrict ourselves to the adjoint case, i.e. $h$ is the dual form of $g$, in which Darmon-Lauder-Rotger's conjecture is proven by Rivero and Rotger in \cite{riveroadjoint}.  Letting $\ad \rho_g$ be the traceless adjoint of $\rho_g$, we have in this case
\[\rho_g \otimes \rho_h \simeq \rho_g \otimes \rho_g^* \simeq \ad \rho_g \oplus \mathds{1}\]
as $G_{\bQ}$-representations.

Consider a cuspidal newform $g$ of weight one with coefficients in $\cO_p$ and let $\rho=\ad \rho_g$. Then $\rho$ does not contain the trivial representation and $(d,d^+)=(3,1)$. Assume that $\rho_g$ is unramified at $p$, residually irreducible, and $\alpha\neq \pm\beta$ modulo the maximal ideal of $\cO_p$, where $\alpha$ and $\beta$ are the roots of the $p$-th Hecke polynomial of $g$. Assume also that $g$ does not have real multiplication by a real quadratic field in which $p$ splits.

The Frobenius substitution $\sigma_p$ at $p$ acts on $W$ with distinct eigenvalues $1,\alpha/\beta,\beta/\alpha$, and there are accordingly three choices $W_1,W_{\alpha/\beta},W_{\beta/\alpha}$ of $p$-stabilizations of $W_p$. Note that these three lines are motivic, hence $\eta$-admissible for all $\eta\in\widehat{\Gamma}$ by Lemma \ref{lem:schanuel+motivic_implies_admissible}. 

The theorem of Rivero-Rotger for $\ad (g_\alpha)$ can be formulated in terms of the $\cL$-invariant introduced in \S\ref{sec:L-invariant} as follows.

\begin{proposition}\label{prop:rivero_rotger}
Take $W_p^+$ to be $E_p \otimes_E W_{\beta/\alpha}$. Then $e=\dim \HH^0(\Qp,W_p^-)=1$ and
\begin{equation}\label{eq:rivero_rotger}
I_p(g_\alpha)=\cL(\rho,\rho^+) \quad \mod E^\times,
\end{equation}
where $I_p(g_\alpha)\in E_p$ is the $p$-adic period attached to $g_\alpha$ (see \cite[\S1]{riveroadjoint} for its definition). Moreover, for any $E$-basis $\omega_p^+$ of the line $W_{\beta/\alpha}$, we have
\begin{equation}\label{eq:gross-stark_unit}
	\Reg_{{\omega}^+_{p}}(\rho)= \log_p(u_{g_\alpha}) \mod E^\times,
\end{equation}
where $u_{g_\alpha}\in E\otimes \cO_H^\times$ is the Gross-Stark unit appearing in the Elliptic Stark Conjecture and its variants \cite{DLR1,DLR2}.
\end{proposition} 
\begin{proof}
As $\alpha/\beta$ and $\beta/\alpha$ are $\neq 1$, one has $\HH^0(\Qp,W_p^{-})=E_p\otimes W_1$ which is a line, so \(e=1\). Consider the matrices $A^\pm,B^\pm,O^-$ defined in \eqref{eq:def_mat_A_B_O} using motivic bases $\{t_1^+,t_1^-\}$ and $\{\psi_1,\psi'_1\}$. That is, $t_1^+\in W_{\beta/\alpha}\subset W_{p}^+$, $t_1^-\in W_1$ and $\psi_1,\psi'_1\in \Hom_{G_{\bQ}}(W,U'_0)$ where $U'_0=E \otimes \cO_H[\tfrac{1}{p}]^\times$. 

\cite[Thm. A]{riveroadjoint} states that 
$$I_p(g_\alpha)=\det \begin{pmatrix}A^+ & B^+ \\ A^- & B^- \end{pmatrix} (\det A^+)^{-1} \mod E^\times.$$
As $\det(O^-)=\ord_p(\psi_1'(t_1^-))\in E^\times$, we conclude from Lemma \ref{lem:formule_L_invariant_quot_det} that \eqref{eq:rivero_rotger} holds.

The $p$-adic Stark regulator $\Reg_{{\omega}^+_{p}}(\rho)$, computed with respect to the bases $\omega_{p}^+=t_1^+$ and $\omega_{\ff}=\psi_1$, is equal to $\det A^+=\log_p(\psi_1(t_1^+))$. Recall that $u_{g_\alpha}$ is, by definition, any generator of the $E$-line of $\psi_1(W)\subset E\otimes \cO_H^\times$ on which $\sigma_p$ acts with eigenvalue $\beta/\alpha$. As $\psi_1(t_1)$ belongs to the same $E$-line and is non-trivial, their $p$-adic logarithms must coincide modulo $E^\times$.
\end{proof}

The $p$-adic iterated integral of \eqref{eq:rivero_rotger} can be recast as the leading term of a certain $p$-adic $L$-function as we now recall.

Consider as in \S\ref{sec:weight_one} the unique primitive Hida family  $\cG$ passing through $g_\alpha$. Bellaiche-Dimitrov's theorem \cite{bellaiche2016eigencurve} implies that $\cG$ can be parameterized by the weight variable $k$ around $g_\alpha$. Let $L^\HH_p(\ad(\cG),k,s)$ be Hida's two-variable $p$-adic $L$-function attached to $\ad(\cG)$ in \cite{hidaGL3} and put $L^\HH_p(\ad(g_\alpha),s)=L^\HH_p(\ad(\cG),k,s)_{\big|k=1}$. 
By \cite[Lem. 4.2]{DLR2} and \cite[Prop. 2.5]{riveroadjoint}, the quantity $I_p(g_\alpha)$ can be interpreted as the value at $s=1$ of the first derivative of $L^\HH_p(\ad(g_\alpha),s)$. Therefore, Proposition \ref{prop:rivero_rotger} reads:
\begin{equation}\label{eq:derivative_hida}
	(L^\HH_p)'(\ad(g_\alpha),1)= \cL(\rho,\rho^+)\mod E^\times.
\end{equation}

\textbf{EZC}$_{\rho,\rho^+}$ gives a similar formula for the $p$-adic $L$-function defined by $L_p(\rho,\rho^+,s)=\kappa^{1-s}(\theta'_{\rho,\rho^+})$, where $\theta'_{\rho,\rho^+}$ is computed with respect to a basis $\omega_p^+$ of the $E$-line $W_{\beta/\alpha}\subseteq W_p^+$, namely:
\begin{equation}\label{eq:derivative_motivic}
	L'_p(\rho,\rho^+,1)= \Reg_{{\omega}^+_{p}}(\rho)\cdot\cL(\rho,\rho^+)=\log_p (u_{g_\alpha})\cdot\cL(\rho,\rho^+) \mod E^\times.
\end{equation}
Note that $\log_p (u_{g_\alpha})$ should be transcendental according to the weak $p$-adic Schanuel conjecture (Conjecture \ref{conj:schanuel}). In particular, Hida's construction does not provide the right candidate for the $p$-adic $L$-function of $(\rho,\rho^+)$. A similar observation for higher weight specializations of $L^\HH_p(\ad(\cG),k,s)$ is already raised by Hida in his monograph \cite{hidagenuine} (see also \cite[\S4]{hidatilouineurban}). Given a weight $k\geq 2$ specialization $g_k$ of $\cG$, $L^\HH_p(\ad(\cG),k,s)$ interpolates the critical $L$-values of $g_k$ divided by periods of automorphic nature which do not match with the motivic periods given for instance by the conjectures of Coates and Perrin-Riou (\cite{coatesperrinriou}). Further, Hida related in \cite{hidamodules} the quotient of these periods with the size of the congruence module of $g_k$, and it was proposed in \cite{hidagenuine} and in \cite[\S4]{hidatilouineurban} that, up to fudge factors arising from primitivity issues, his $p$-adic $L$-function could be written as
\begin{equation}\label{eq:conj_Hida}
	L^\HH_p(\ad(\cG),k,s)=\frac{L_p(\ad(\cG),k,s)}{C(k)},
\end{equation}
where $L_p(\ad(\cG),k,s)$ is a two-variable Iwasawa function interpolating $L$-values of the $g_k$'s normalized with respect to motivic periods, and $C(k)$ is a generator of the congruence ideal of $\cG$. 

It is plausible that $L_p(\ad(\cG),k,s)_{|k=1}$ equals $L_p(\rho,\rho^+,s)$ up to a fudge factor. In light of \eqref{eq:derivative_hida}, \eqref{eq:derivative_motivic} and \eqref{eq:conj_Hida}, one may then expect to find a simple relation between $C(k)_{|k=1}$ and $\log_p (u_{g_\alpha})$. We shall provide some evidences concerning this latter fact in a sequel, where it will be shown that, under mild assumptions on $g$, there exists a generator $C(k)$ of the congruence ideal of $\cG$ such that
\[C(k)_{\big|k=1}=\left(1-\frac{\beta}{\alpha}p^{-1}\right)\left(1-\frac{\alpha}{\beta}\right)(\#\Sha(T_p))\cdot \log_p (u_{g_\alpha})\in E^\times\cdot \log_p (u_{g_\alpha}),\]
where $T_p$ is a Galois-stable $\cO_p$-lattice of $W_p$, $\Sha(T_p)$ is the Tate-Shafarevitch group of $T_p$ and $u_{g_\alpha}$ is assumed to be $T_p$-normalized.

\bibliography{bib}
\bibliographystyle{alpha}
\end{document}